%% file: main.tex
\documentclass[11pt]{article}
\usepackage{graphicx,psfrag}
\usepackage{amssymb,amsfonts,amsthm}
\usepackage{todonotes}
\usepackage{qtree}
\usepackage{amsmath,amsthm,amssymb}
\usepackage{amsfonts}
\usepackage{relsize}
\usepackage[T1]{fontenc}
\usepackage[utf8]{inputenc}
\usepackage{url}
\providecommand{\doi}[1]{\texttt{doi:#1}}
\usepackage{graphicx}
\usepackage{color}
\usepackage{subcaption,mathtools}
\usepackage{tkz-graph}
\usetikzlibrary{trees,arrows}
\usepackage{tikz}
\usetikzlibrary{positioning}
\usepackage{qtree}
\usepackage{subfiles}
\usepackage{float}
\usepackage{adjustbox}
\usepackage{xcolor}
\usepackage{soul}

\usetikzlibrary{arrows,automata}
\tikzset{
	EdgeStyle/.append style = {->} }

\newcommand{\be}{\begin{equation}}
\newcommand{\ee}{\end{equation}}

\newcommand {\N}{\mathbb{N}} 
\newcommand {\Z}{\mathbb{Z}}            

\newcommand{\Cl}{\mathrm{Cl}}

\newcommand{\Mod}[1]{\ (\mathrm{mod}\ #1)}

\numberwithin{equation}{section}
\newcounter{AbcT}

\newcommand{\nc}{\newcommand}
\nc{\meet}{\wedge}
\nc{\op}{\operatorname}\nc{\FP}{\op{FP}}\nc{\FS}{\op{FS}}\nc{\FPhat}
{\widehat{\op{FP}}}

\newtheorem {Theorem}    {Theorem}[section]

\newtheorem {Problem}    [Theorem]{Problem}

\newtheorem {Lemma}      [Theorem]    {Lemma}
\newtheorem {Corollary}   [Theorem] {Corollary}
\newtheorem {Proposition}[Theorem]    {Proposition}

\theoremstyle{remark}
\newtheorem {Remark}		 [Theorem]    {\bf{Remark}}

\newtheorem {Definition} [Theorem]    {\bf{Definition}}

\newcommand\restr[2]{{
  \left.\kern-\nulldelimiterspace 
  #1 
  \vphantom{\big|} 
  \right|_{#2} 
  }}

\begin{document}

\title{On The Generation Problem in Thompson's Groups $F_n$}

\author{Gili Golan and Eytan Sapir}
\date{}
\maketitle
\begin{abstract}
We study the generation problem in the Higman--Thompson groups $F_n$ via the
core and closure of subgroups of $F_n$ and associated automata. We give
sufficient conditions for a subset $X \subseteq F_n$ to generate $F_n$,
and provide an algorithm which verifies these conditions when $X$ is finite.
As an application, we answer a question of Aiello and Nagnibeda, motivated by
Savchuk's problem on maximal subgroups of Thompson's group $F$. Specifically,
we show that for every $n\geq 2$, the Higman--Thompson group $F_n$ contains a
maximal subgroup of infinite index which does not fix any point in $(0,1)$. The subgroup
we construct is isomorphic to $F_{2n-1}$.
\end{abstract}

\subfile{Introduction}

\subfile{Preliminaries}
\subfile{Thegenerationproblem}
\subfile{Oncoreautomata}
\subfile{Onmaximalsubgroups}

\subfile{Openproblems}

\newpage

\subfile{Bibliography}
\end{document}

%% file: Introduction.tex
\section{Introduction}\label{section:intro}

Thompson's group $F$ was introduced by R. Thompson in the  1960s and has since become a central example in geometric and combinatorial group theory. It can be realized as the group of all piecewise-linear orientation-preserving homeomorphisms of the unit interval whose breakpoints are dyadic rationals and whose slopes are powers of $2$. Brown studied the corresponding $n$-adic analogues $F_n$, $n\ge 2$, which are known as Higman--Thompson groups: $F_n$ consists of the piecewise-linear homeomorphisms of $[0,1]$ with breakpoints in $\Z[\frac1n]$ and slopes powers of $n$; thus $F=F_2$ \cite{B1}. These groups also admit a combinatorial description by pairs of finite $n$-ary trees with the same number of leaves, as well as a description as diagram groups in the sense of Guba and Sapir \cite{GubaSapir1997}.

The present paper studies the generation problem for $F_n$. For a finitely generated group $G$, the generation problem asks whether there is an algorithm which, given a finite subset $X\subseteq G$, decides whether $X$ generates $G$. This problem is undecidable in some finitely generated groups, for instance in the direct product of a rank-two free group with itself; see, for example, \cite{LyndonSchupp1977}. For Thompson's group $F=F_2$, the first author solved the generation problem in \cite{G1} using the Stallings $2$-core of a subgroup. Our first goal is to extend the relevant core and semi-core methods to the Higman--Thompson groups $F_n$ and to obtain effective sufficient conditions for generation.

The core of a subgroup of a diagram group was introduced by Guba and Sapir. For subgroups of Thompson's group $F$, the construction adapts naturally to the language of tree diagrams and automata \cite{GS1,G2}. 
We use the analogous construction for $F_n$. If $H\le F_n$, its core, denoted $\mathcal L(H)$, is a rooted edge-labeled automaton constructed from tree diagrams for a generating set of $H$ by identifying the corresponding roots and leaves and then applying two types of  foldings. Such an automaton accepts a tree diagram when the branches of the two trees can be read from the root and paired branches end at the same vertices. An element of $F_n$ is accepted if it has an accepted tree-diagram representative. The set of all elements accepted by $\mathcal L(H)$ is the closure of $H$, denoted $\Cl(H)$. This operation satisfies the usual properties of a closure operator: $H\le \Cl(H)$, $\Cl(\Cl(H))=\Cl(H)$, and $H_1\le H_2$ implies $\Cl(H_1)\le \Cl(H_2)$.

A related object, the semi-core, was introduced by the first author for subsets of $F$ \cite{G1}. We adapt the  construction to $F_n$. If $X\subseteq F_n$, the semi-core $\mathcal L_{sem}(X)$ is obtained from the tree diagrams of the elements of $X$ by identifying the corresponding roots and leaves and then applying only foldings of type $1$. Unlike the core, the semi-core depends on the chosen generating set rather than only on the subgroup it generates. This dependence is useful: identifications in $\mathcal L_{sem}(X)$ can be used to produce explicit elements of the subgroup $\langle X\rangle$ with prescribed local behavior on $n$-adic intervals.

We shall use one simple invariant. For a finite $n$-ary word $w$, let $t(w)$ be the sum of the digits of $w$ modulo $n-1$. Every point $\alpha\in \Z[\frac1n]\cap(0,1)$ has a finite $n$-ary expansion, unique up to appending zeros, and we define $t(\alpha)$ to be the sum of the digits in such an expansion modulo $n-1$. For $i\in\{0,\dots,n-2\}$, let
\[
    D_{n,i}=\{\alpha\in \Z[\tfrac1n]\cap(0,1)\mid t(\alpha)=i\}.
\]
The sets $D_{n,0},\dots,D_{n,n-2}$ are precisely the orbits of the action of $F_n$ on the $n$-adic points in $(0,1)$. Thus, for $F_2$ there is only one such orbit, while for $F_n$ with $n>2$ this congruence invariant is one of the main new features. 


Our main generation criterion is the following theorem. Here, if a word $w$ labels a path in an automaton, $w^+$ denotes the terminal vertex of that path.

\begin{Theorem}\label{intro_t_1}
Let $X\subseteq F_n$, and let $H=\langle X\rangle$. Assume that the following conditions hold.
\begin{itemize}
    \item[$(1)$] $[F_n,F_n]\subseteq \Cl(H)$.
    \item[$(2)$] $H[F_n,F_n]=F_n$.
    \item[$(3)$] For each $i\in\{0,\dots,n-2\}$, there exist an element $h_i\in H$ and a point $\alpha_i\in D_{n,i}$ such that
    \[
        h_i(\alpha_i)=\alpha_i,\qquad h_i'(\alpha_i^-)=1,\qquad h_i'(\alpha_i^+)=n.
    \]
    \item[$(4)$] There exists an inner $n$-ary word $u$, that is, a word not of the form $0^k$ or $(n-1)^k$, such that for every two finite $n$-ary words $v_1,v_2$ with $t(v_1)=t(v_2)$, the paths labeled by $uv_1$ and $uv_2$ in $\mathcal L_{sem}(X)$ have the same terminal vertex; that is,
    \[
        (uv_1)^+=(uv_2)^+ \quad \text{in } \mathcal L_{sem}(X).
    \]
\end{itemize}
Then $X$ generates $F_n$.
\end{Theorem}

The first two conditions separate the problem into the part detected by the closure and the part detected by the abelianization. Condition $(1)$ says that the closure of $H$ contains the derived subgroup of $F_n$; for finitely generated $H$, this can be checked on the finite core by verifying that every vertex is a father and that there are exactly $n-1$ inner vertices. Condition $(2)$ says that the image of $H$ in the abelianization of $F_n$ is the whole abelianization. Condition $(3)$ supplies, in each congruence class $D_{n,i}$, an element of $H$ with a prescribed one-sided slope behavior at a fixed point. We give a tuple algorithm which decides this condition for finitely generated subgroups satisfying $(1)$.

Condition $(4)$ is the semi-core condition. It says that, after reading a fixed inner prefix $u$, the relevant rooted part of the semi-core has the same inner structure as the core of $F_n$: among paths starting with $u$, the terminal vertex depends only on the value of $t$. Under condition $(1)$, the corresponding identifications already hold in the core $\mathcal L(H)$. The  semi-core condition is used to ensure that these identifications can be realized by elements of the subgroup $H=\langle X\rangle$ itself, rather than only by elements of its closure.

For Thompson's group $F=F_2$, this extra condition is superfluous. In \cite{G1}, it was proved that the analogues of conditions $(1)$, $(2)$ and $(3)$ already imply that $H=F$. This gives a necessary and sufficient generation criterion for $F$ and solves the generation problem in $F$. Moreover, in later work on maximal subgroups of $F$ \cite{G2}, it was proved that, for subgroups of $F$, the analogues of conditions $(1)$ and $(2)$ already imply the analogue of condition $(3)$. Thus, for $F$, generation is detected by the closure and the image in the abelianization alone. One of the natural questions left by the present paper is whether analogous simplifications hold for $F_n$ when $n>2$.

The second application of these methods concerns maximal subgroups. Savchuk initiated the study of maximal subgroups of infinite index in Thompson's group $F$ by proving that, for every $\alpha\in(0,1)$, the stabilizer of $\alpha$ in $F$ is a maximal subgroup of infinite index, and he asked whether there are other examples \cite{Sav1,Sav2}. In answer to this question, the first author and Mark Sapir constructed an explicit maximal subgroup of infinite index in $F$ which does not fix any point of $(0,1)$ \cite{GS1}. This subgroup is the preimage of Jones' oriented subgroup $\vec F$ under an injective endomorphism of $F$; recall that $\vec F$ is isomorphic to $F_3$.

The work of the first author and Mark Sapir also provided a general method for proving the existence of many other maximal subgroups of $F$ \cite{GS1,G2}.  Aiello and Nagnibeda constructed three more explicit maximal subgroups of infinite index in $F$, relying on the techniques from \cite{GS1}; see \cite{AN2}. In \cite{G2}, the first author proved that $F$ has infinitely many non-isomorphic maximal subgroups of infinite index by constructing, for each prime $p$, a maximal subgroup of $F$ isomorphic to $F_{p+1}$.

The search for maximal subgroups naturally extends to the Higman--Thompson groups $F_n$. Savchuk's proof readily adapts to this setting, showing that for every $\alpha\in(0,1)$, the stabilizer of $\alpha$ in $F_n$ is a maximal subgroup of infinite index. Aiello and Nagnibeda provided the first example of a maximal subgroup of infinite index in $F_3$ which does not fix a point of $(0,1)$ \cite{AN}. Their example is the preimage of the Jones subgroup of $F_3$ under an injective endomorphism of $F_3$.\footnote{This subgroup is denoted $\vec F_3$ in \cite{AN}. This notation should not be confused with the subgroup $\vec F_3$ of $F$ appearing in the family $\vec F_n$ in \cite{G2}.} They asked whether, for every $n\ge2$, the group $F_n$ contains a maximal subgroup of infinite index which does not fix any point of $(0,1)$.

We answer this question affirmatively for every $n\ge2$.

\begin{Theorem}\label{intro_t_2}
For every integer $n\ge 2$, the group $F_n$ contains a maximal subgroup of infinite index which does not fix any point of $(0,1)$. Moreover, the subgroup can be chosen to be isomorphic to $F_{2n-1}$.
\end{Theorem}

The proof of Theorem \ref{intro_t_2} uses the generation criterion above together with a detailed analysis of core automata. We start from the tree semigroup presentation associated with the core of $F_{2n-1}$. We then construct a sequence of semigroup presentations by applying basic splittings, and pass to the corresponding rooted semi tree automata. After that, we apply foldings of type $2$. The splitting steps preserve the associated semigroup and, by a theorem of Guba and Sapir, the isomorphism type of the corresponding diagram group; the folding steps preserve the diagram group as well. This yields a core automaton for a subgroup $H\le F_n$ with $H\cong F_{2n-1}$. Theorem \ref{intro_t_1} is then used to show that adjoining any element of $F_n\setminus H$ generates all of $F_n$, and hence that $H$ is maximal.

The paper is organized as follows. In Section \ref{section:preliminaries}, we recall the definitions of $F_n$ as a group of piecewise-linear homeomorphisms and as a group of tree diagrams, and we review the core, closure, abelianization, and derived subgroup of $F_n$. In Section \ref{section:gen_problem}, we prove the generation criterion and develop the tuple algorithm used to verify condition $(3)$ above. In Section \ref{section:core_automata}, we study rooted tree automata and the operations on the corresponding semigroup presentations that preserve the associated diagram groups. In Section \ref{section:maximal_subgroups}, we construct the maximal subgroup isomorphic to $F_{2n-1}$ and prove Theorem \ref{intro_t_2}. We conclude in Section \ref{section:open_problems} with some open problems, and with a brief discussion of subsequent related work.

\vskip .2cm

\textbf{Acknowledgements}.
The results proved in this paper formed part of the second author's M.Sc. thesis,
submitted at Ben-Gurion University of the Negev in September 2025 under the
supervision of the first author. The research of both authors was partially
supported by Israel Science Foundation grant Nos.~2322/19 and~2275/24.

%% file: Preliminaries.tex
\section{Preliminaries}
\label{section:preliminaries}
\subsection{Preliminaries on $F_n$}
We start this section by listing some definitions and notations that are relevant for the rest of this paper.
\begin{Definition}
    We call a rational number \emph{$n$-adic} if it can be expressed as $\frac{a}{n^k}$ for some integers $a,k$. Note that every $n$-adic number may be expressed as a finite sum of integer powers of $n$. Denote by $D_n:=\Z[\frac{1}{n}]\cap(0,1)$ the set of all $n$-adic numbers in the open unit interval.
\end{Definition}

\begin{Definition}\label{d_pre_thompsons_group_F_n}
    \emph{Thompson's group $F_n$} (\cite{B1,Wladis2007}) is the group of all piecewise-linear homeomorphisms of the closed unit interval with a finite set of break points contained in $\Z[\frac{1}{n}]$, such that all slopes are integer powers of $n$. The group action is composition from left to right (for elements $f,g\in F_n$ and $x\in [0,1]$, $fg(x)=g(f(x))$). We denote $f^g=g^{-1}fg$, and $[f,g]=fgf^{-1}g^{-1}$.
    
    The standard minimal generating set of $F_n$ consists of $n$ elements, denoted $x_0,\dots,x_{n-1}$. The explicit piecewise-linear formulas for these generators can be found in \cite{BrownGeoghegan1984}.
\end{Definition}

\begin{Definition}
    A \emph{finite $n$-ary word} is a finite string composed of the characters $0,1,\dots,(n-1)$. We say that a finite non-empty $n$-ary word $w$ is an \emph{inner $n$-ary word} if $w\notin\{0^\ell,(n-1)^\ell\}$ for all $\ell\in\N$. The \emph{length} of an $n$-ary word $\sigma$ is denoted by $|\sigma|$.

    Given an $n$-ary word $\sigma\equiv\sigma_1\dots\sigma_l$, we denote  
    \[.\sigma:=0.\sigma_1\dots\sigma_l=\sum_{i=1}^{l} \sigma_i\cdot n^{-i}\]
    and  
    \[[\sigma]:=[.\sigma,.\sigma+\frac{1}{n^l}]\]

    Note that every number in $D_n$ can be expressed as a finite $n$-ary word.

    We distinguish between finite $n$-ary words based on the sum of their characters modulo $n-1$. Specifically, for a word $\sigma = \sigma_1 \dots \sigma_\ell$, we define
    \[t(\sigma) := \sum_{i=1}^\ell \sigma_i \pmod{n-1}.\]
    We extend this definition to $n$-ary numbers: for $\alpha \in D_n$, we define $t(\alpha) := t(\sigma)$, where $\sigma$ is an $n$-ary representation of $\alpha$ (i.e., $\alpha=.\sigma$). We also set $t(0) = t(1) = 0$.

    For each $i \in \{0, \dots, n-2\}$, we define
    \[D_{n,i} := \{ \alpha \in D_n \mid t(\alpha) = i \}\]
\end{Definition}

\begin{Remark}\label{r_pre_Dni_dense}
    $D_{n,i}$ is dense in $[0,1]$ for all $i\in\{0,\dots,n-2\}$. 
\end{Remark}

\begin{Definition}
    An \emph{$n$-ary tree} $T$ is a finite directed rooted tree (with edges directed away from the root) where each vertex has either $n$ outgoing edges (such a vertex is called a \emph{father}) or $0$ outgoing edges (such a vertex is called a \emph{leaf}). A pair $(T_1,T_2)$ is called an \emph{$n$-ary tree-diagram} (or simply, a \emph{tree-diagram}) if $T_1$ and $T_2$ are $n$-ary trees with the same number of leaves.
\end{Definition}

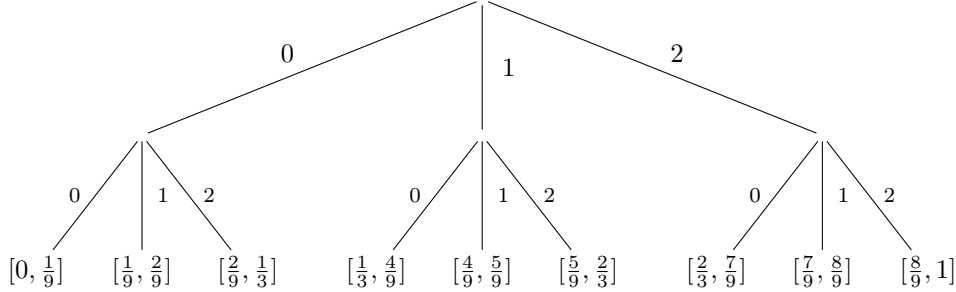
\begin{figure}[H]
\centering
\begin{tikzpicture}[grow'=down, level distance=1.8cm,
  level 1/.style={sibling distance=4.5cm},
  level 2/.style={sibling distance=1.4cm},
  every node/.style={inner sep=2pt}]

\node {} 
    child { node {} 
        child { node {\small $[\frac{8}{9}, 1]$} edge from parent node[left, pos=0.5, xshift=12pt] {\scriptsize 2} }
        child { node {\small $[\frac{7}{9}, \frac{8}{9}]$} edge from parent node[pos=0.5, xshift=8pt] {\scriptsize 1} }
        child { node {\small $[\frac{2}{3}, \frac{7}{9}]$} edge from parent node[right, pos=0.5, xshift=-12pt] {\scriptsize 0} }
        edge from parent node[above left, pos=0.5, xshift=14pt] {\small 2} 
    }
    child { node {} 
        child { node {\small $[\frac{5}{9}, \frac{2}{3}]$} edge from parent node[left, pos=0.5, xshift=12pt] {\scriptsize 2} }
        child { node {\small $[\frac{4}{9}, \frac{5}{9}]$} edge from parent node[pos=0.5, xshift=8pt] {\scriptsize 1} }
        child { node {\small $[\frac{1}{3}, \frac{4}{9}]$} edge from parent node[right, pos=0.5, xshift=-12pt] {\scriptsize 0} }
        edge from parent node[pos=0.5, xshift=10pt] {\small 1} 
    }
    child { node {} 
        child { node {\small $[\frac{2}{9}, \frac{1}{3}]$} edge from parent node[left, pos=0.5, xshift=12pt] {\scriptsize 2} }
        child { node {\small $[\frac{1}{9}, \frac{2}{9}]$} edge from parent node[pos=0.5, xshift=8pt] {\scriptsize 1} }
        child { node {\small $[0, \frac{1}{9}]$} edge from parent node[right, pos=0.5, xshift=-12pt] {\scriptsize 0} }
        edge from parent node[above right, pos=0.5, xshift=-14pt] {\small 0} 
    };
\end{tikzpicture}
\caption{A ternary tree. Edges are labeled by '0','1','2'. Each vertex corresponds to a triadic ($3$-adic) interval, and the leaves form a division of the unit interval to triadic sub-intervals.}
\end{figure}

All paths in an $n$-ary tree in this paper will be \emph{rooted paths} (that is, their initial vertex is the root), so we will usually omit the word ``rooted''. A path ending in a leaf is called a \emph{branch}. Each path $p$ in an $n$-ary tree can be naturally labeled by an $n$-ary word $lab(p)$. Explicitly, we label the path consisting only of the root by the empty string. Then, if $p$ is a path, we label the paths extending $p$ by a single edge as $lab(p)i$ for $i\in\{0,\dots,n-1\}$. Throughout this paper, we will often not distinguish between a rooted path, its label, and the vertex at the end of the path. For the end vertex $v$ of a path labeled by a non-empty $n$-ary word $\sigma\equiv\sigma_1\dots\sigma_l$, we denote $.v:=.\sigma=\sum_{j=1}^l{\sigma_j\cdot n^{-j}}$ and $[v]:=[\sigma]=[.v,.v+\frac{1}{n^l}]$. Note that given an $n$-ary tree with leaves $v_1,\dots,v_m$, the intervals $[v_1],\dots,[v_m]$ partition the unit interval into $n$-adic intervals, providing an \emph{$n$-adic subdivision}.

An element of $F_n$ can be viewed as an equivalence class of  pairs of $n$-ary trees with the same number of leaves. Let $(T_1,T_2)$ be an $n$-ary tree-diagram. Let the (labels of the) branches of $T_1$ be $u_1,\dots,u_m$, and the branches of $T_2$ be $v_1,\dots,v_m$. For any $1\le j\le m$, we say that $(T_1,T_2)$ has the \emph{pair of branches $u_j\rightarrow v_j$}. The tree-diagram corresponds to a function $f \in F_n$ that maps each interval $[u_j]$ linearly onto the interval $[v_j]$.

A \emph{caret under $w$} is a tree consisting of a root vertex $w$ and $n$ leaf nodes under $w$ (often, we omit the reference to $w$ and simply call it a \emph{caret}). A tree-diagram $(T_1',T_2')$ is obtained from $(T_1,T_2)$ by adding a \emph{common caret} if $T'_1$ and $T'_2$ are obtained by adding carets under leaves $u$ and $v$ respectively, where $u\rightarrow v$ is a pair of branches of $(T_1,T_2)$. 
Elements of $F_n$ are defined as equivalence classes of tree-diagrams, where two diagrams are equivalent (denoted by $\sim$) if one can pass from one to the other via a finite sequence of adding or removing common carets. If a tree-diagram cannot be obtained by adding a common caret to another tree-diagram it is called \emph{reduced}; otherwise it is called \emph{reducible}. Every tree-diagram has a unique equivalent reduced tree-diagram (see \cite{CFP1996} for example).

Note that for every two tree-diagrams $(T_1,T_2),(T_3,T_4)$, there exist (respectively) equivalent tree-diagrams $(T_1',T_2'),(T_3',T_4')$ such that $T_2'=T_3'$. With this notation, the product of two equivalence classes of tree-diagrams $(T_1,T_2)\cdot (T_3,T_4)$ is the equivalence class of $(T_1',T_4')$. This construction provides an isomorphism between $F_n$ as a group of homeomorphisms and its description as a group of equivalence classes of tree-diagrams (for more details about the point of view as tree-diagrams, see \cite{CFP1996} for the $F_2$ case, \cite{B1} for the generalization to $F_n$, and \cite{GubaSapir1997} for the broader context of diagram groups).

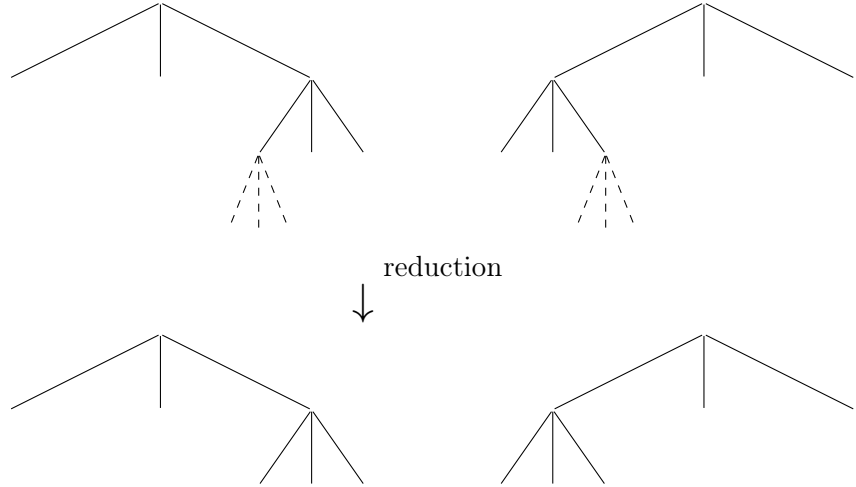
\begin{figure}[H]
\centering
\begin{minipage}{\textwidth}
\centering
\begin{tikzpicture}[grow'=down, level distance=1cm,
  level 1/.style={sibling distance=2cm},
  level 2/.style={sibling distance=0.7cm},
  level 3/.style={sibling distance=0.4cm},
  every node/.style={inner sep=0pt}]
\node {}
    child { node {}
        child { node {} }
        child { node {} }
        child { node {}
            child[dashed] { node {} }
            child[dashed] { node {} }
            child[dashed] { node {} }
        }
    }
    child { node {} }
    child { node {} };
\end{tikzpicture}
\hspace{1.5cm}
\begin{tikzpicture}[grow'=down, level distance=1cm,
  level 1/.style={sibling distance=2cm},
  level 2/.style={sibling distance=0.7cm},
  level 3/.style={sibling distance=0.4cm},
  every node/.style={inner sep=0pt}]
\node {}
    child { node {} }
    child { node {} }
    child { node {}
        child { node {}
            child[dashed] { node {} }
            child[dashed] { node {} }
            child[dashed] { node {} }
        }
        child { node {} }
        child { node {} }
    };
\end{tikzpicture}
\end{minipage}

\vspace{0.3cm}
\scalebox{1.5}{$\downarrow$} \raisebox{0.5cm}{\text{reduction}}
\vspace{0.1cm}

\begin{minipage}{\textwidth}
\centering
\begin{tikzpicture}[grow'=down, level distance=1cm,
  level 1/.style={sibling distance=2cm},
  level 2/.style={sibling distance=0.7cm},
  every node/.style={inner sep=0pt}]
\node {}
    child { node {}
        child { node {} }
        child { node {} }
        child { node {} }
    }
    child { node {} }
    child { node {} };
\end{tikzpicture}
\hspace{1.5cm}
\begin{tikzpicture}[grow'=down, level distance=1cm,
  level 1/.style={sibling distance=2cm},
  level 2/.style={sibling distance=0.7cm},
  every node/.style={inner sep=0pt}]
\node {}
    child { node {} }
    child { node {} }
    child { node {}
        child { node {} }
        child { node {} }
        child { node {} }
    };
\end{tikzpicture}
\end{minipage}
\caption{The top and bottom tree-diagrams correspond to the same element of $F_3$. The bottom is obtained from the top one by a reduction of a single common caret.}
\end{figure}

\begin{Remark}\label{r_pre_unique_reduced}
    For any tree-diagram $(T_1,T_2)$, there exists a unique reduced tree-diagram $(R_1,R_2)$ equivalent to $(T_1,T_2)$ (see \cite{CFP1996,Wladis2007}). Thus, Thompson's group $F_n$ can be viewed as the group of all reduced tree-diagrams.

\end{Remark}

\begin{Definition}
    Let $f\in F_n$, and let $u,v$ be two $n$-ary words. We say that \emph{$f$ has the pair of branches $u\rightarrow v$} if there exists a tree-diagram $(T_1,T_2)$ representing $f$ such that $u\rightarrow v$ is a pair of branches of $(T_1,T_2)$. Equivalently, $f$ has the pair of branches $u\rightarrow v$ if and only if $f$ maps the interval $[u]$ linearly onto the interval $[v]$.
\end{Definition}

A well-known fact is that $F_n$ acts transitively on $D_{n,i}$ for all $i\in\{0,\dots,n-2\}$ (see, for example, \cite{Stein}). We will later use the following slightly stronger claim. 

\begin{Lemma}\label{lPairOfBranches}
    Let $u$ and $v$ be inner $n$-ary words. There exists a reduced tree-diagram of an element of $F_n$ with the pair of branches $u\rightarrow v$ if and only if $t(u)=t(v)$.
\end{Lemma}
More generally, one can show the following.
\begin{Remark}\label{rPreGeneralPairOfBranches}
    Let $u_1,...,u_k,v_1,...,v_k$ be inner $n$-ary words such that the following conditions hold: \begin{itemize}
    \item[$1.$] For all $1\le j\le k-1$, the right endpoint of $u_j$ is not bigger than the left endpoint of $[u_{j+1}]$, and the right endpoint of $[v_j]$ is not bigger than the left endpoint of $[v_{j+1}]$ for all $1\leq j\leq k-1$.
    \item[$2.$] For all $1\le j\le k-1$, $[u_j]\cap[u_{j+1}]\neq\emptyset$ if and only if $[v_j]\cap[v_{j+1}]\neq\emptyset$.
    \item[$3.$] $t(u_j)=t(v_j)$ for all $1\le j\le k$
\end{itemize}
Then there is an element $f\in F_n$ with the pairs of branches $u_j\rightarrow v_j$ for all $1\leq j \leq k$.
\end{Remark}

\subsection{The core and closure of a subgroup of $F_n$}
\label{sec:preliminaries:core_and_closure}

The core of a subgroup of a diagram group was originally defined in \cite{GS1} using the language of directed $2$-complexes. Later, it was described specifically for Thompson's group  $F$ in the narrower language of tree-diagrams and tree-automata (\cite{G2}). We follow the latter approach and provide a definition of the core for subgroups of $F_n$ using automata.

Recall that an \emph{automaton} $\mathcal{A}$ is a directed, edge-labeled graph. Every automaton considered in this paper has a distinguished vertex $r$ called the \emph{root}. We denote such an automaton by $\mathcal{A}_r$ and refer to it as a \emph{rooted automaton}. Throughout this paper, a \emph{path} in a rooted automaton $\mathcal{A}_r$ refers to a finite directed path starting from the root.
\begin{Definition}
    Let $\mathcal{A}_r$ be a rooted automaton. We say that $\mathcal{A}_r$ is a \emph{rooted $n$-ary tree-automaton} (or simply, an \emph{$n$-ary tree-automaton}) if the following conditions hold:
    \begin{itemize}
        \item[(1)] Every vertex $x$ in $\mathcal{A}_r$ has either $0$ or $n$ outgoing edges.
        \item[(2)] If a vertex $x$ in $\mathcal{A}_r$ has $n$ outgoing edges (in which case $x$ is called a \emph{father}), these edges are labeled by $i$ for each $i\in\{0,\dots,n-1\}$. The vertex at the end of the edge labeled by $i$ is called the \emph{$i$-th child} of $x$ (we sometimes refer to the $0$-th child as the \emph{leftmost child} and the $(n-1)$-th child as the \emph{rightmost child}).
        \item[(3)] If $x_1$ and $x_2$ are distinct fathers in $\mathcal{A}_r$, then there exists $i\in\{0,\dots,n-1\}$ such that the $i$-th child of $x_1$ differs from the $i$-th child of $x_2$.
        \item[(4)] Every vertex $x$ in $\mathcal{A}_r$ is reachable by a path starting from $r$.   
    \end{itemize}
\end{Definition}
Let $\mathcal{A}_r$ be an $n$-ary tree-automaton. A vertex in $\mathcal{A}_r$ with no outgoing edges is called a \emph{leaf}. Each path in $\mathcal{A}_r$ can be labeled by an $n$-ary word $\sigma$; we will usually not distinguish between a path and its label. Note that an $n$-ary word labels at most one path in $\mathcal{A}_r$. If an $n$-ary word $\sigma$ labels a path in $\mathcal{A}_r$, we say that $\sigma$ is \emph{readable} on $\mathcal{A}_r$, and we denote the end vertex of this path by $\sigma^+$. 

Similarly, given a vertex $u$ in an $n$-ary tree $T$, we say that the vertex $u$ or that the path $u$ is \textit{readable} on $\mathcal{A}_r$ if the word labeling the rooted path to $u$ is readable on $\mathcal{A}_r$; in this case, we denote the end vertex of this path in $\mathcal{A}_r$ by $u^+$. An $n$-ary tree $T$ is said to be \emph{readable} on $\mathcal{A}_r$ if every leaf of $T$ is readable on $\mathcal{A}_r$. 
\begin{Definition}\label{d_pre_core_readable_accepted}
    Let $\mathcal{A}_r$ be an $n$-ary tree-automaton and $(T_1,T_2)$ be a tree-diagram.
    \begin{itemize}
        \item We say that $(T_1,T_2)$ is \emph{readable} on $\mathcal{A}_r$ if both $T_1$ and $T_2$ are readable on $\mathcal{A}_r$.
        \item We say that $(T_1,T_2)$ is \emph{accepted} by $\mathcal{A}_r$ if it is readable on $\mathcal{A}_r$ and, for every pair of branches $u\rightarrow v$ in $(T_1,T_2)$, we have $u^+=v^+$ in $\mathcal{A}_r$.
    \end{itemize}
\end{Definition}
\begin{Lemma}\label{l_core_automata_properties}
    Let $\mathcal{A}_r$ be an $n$-ary tree-automaton. The following properties hold:
    \begin{itemize}
        \item[(1)] If a tree-diagram $(T_1,T_2)$ is accepted by $\mathcal{A}_r$, then the reduced tree-diagram equivalent to $(T_1,T_2)$ is also accepted by $\mathcal{A}_r$.
        \item[(2)] If a tree-diagram $(T_1,T_2)$ is accepted by $\mathcal{A}_r$, then any equivalent tree-diagram $(R_1,R_2)$ is accepted by $\mathcal{A}_r$ if and only if $R_1$ (respectively, $R_2$) is readable on $\mathcal{A}_r$.
        \item[(3)] If $(T_1,T_2)$ and $(R_1,R_2)$ are tree-diagrams accepted by $\mathcal{A}_r$, then the reduced tree-diagram of their product $(T_1,T_2)\cdot (R_1,R_2)$ is also accepted by $\mathcal{A}_r$.
        \end{itemize}
\end{Lemma}

The proof of this lemma for $F=F_2$ was  given in \cite{G2}. The proof for $F_n$ is almost identical, so we omit it.

We say that an element $f\in F_n$ is \emph{accepted} by an $n$-ary tree-automaton $\mathcal{A}_r$ if the reduced tree-diagram representing $f$ is accepted by $\mathcal{A}_r$. Equivalently, by Lemma \ref{l_core_automata_properties}, $f$ is accepted by $\mathcal{A}_r$ if there exists a tree-diagram representing $f$ that is accepted by $\mathcal{A}_r$. By Lemma \ref{l_core_automata_properties}, given an $n$-ary tree-automaton $\mathcal{A}_r$, the set of all elements in $F_n$ accepted by $\mathcal{A}_r$ forms a subgroup of $F_n$.
Following \cite{G2} we make the following definitions.  
\begin{Definition}\label{d_pre_group_accepted_by_automaton}
    Let $\mathcal{A}_r$ be an $n$-ary tree-automaton. We define the \emph{diagram group over $\mathcal{A}_r$}, denoted by $\mathcal{DG}(\mathcal{A}_r)$, as the group of all elements of $F_n$ accepted by $\mathcal{A}_r$.
\end{Definition}

Note that diagram groups over rooted tree-automata are a special case of
the diagram groups studied by Guba and Sapir \cite{GubaSapir1997,GubaSapir1999subgroups}.

\begin{Definition}\label{d_core_closed_group}
    Let $H$ be a subgroup of $F_n$. We say that $H$ is \emph{closed} if there exists some $n$-ary tree-automaton $\mathcal{A}_r$ such that $H=\mathcal{DG}(\mathcal{A}_r)$.
\end{Definition}

Let $H$ be a subgroup of $F_n$. We want to find the smallest closed subgroup of $F_n$ that contains $H$. To do so, we define a rooted $n$-ary tree-automaton called the core of $H$.

\begin{Definition}\label{d_core_stallings2_core}
    Let $H$ be a subgroup of $F_n$. Let $X=\{(T_1^i,T_2^i):i\in \mathcal{I}\}$ be a generating set of $H$, where for each $i\in\mathcal{I}$, $(T_1^i,T_2^i)$ is reduced. The \emph{core} of $H$ is defined as follows:

    For each $i\in\mathcal{I}$, we consider $T_1^i$ and $T_2^i$ as directed rooted trees, where edges are directed from fathers to their children. We identify all the roots of all the trees to a single vertex. Next, for each $i\in\mathcal{I}$ and each pair of branches $u\rightarrow v$ in $(T_1^i,T_2^i)$, we identify the leaves $u$ and $v$. 
    We then apply further identifications of two types, called \emph{foldings}:
    \begin{itemize}
        \item[(1)] A \emph{folding of type $1$} is defined as follows: if a vertex $x$ has several outgoing edges with the same label, we identify all these edges into a single edge, and all of their end vertices into a single vertex. We repeat this step as long as it is applicable.  As a result (if $X$ is infinite, then in the limit state, after possibly infinitely many foldings) we get a directed edge-labeled graph
    where every vertex $x$ has either zero or $n$ outgoing edges (in which case we will refer to their end vertices as the
    children of $x$).
        
        \item[(2)] A \emph{folding of type $2$} is defined as follows:  If $x$ and $y$ are distinct vertices in the directed graph obtained such that both $x$ and $y$ have $n$ outgoing edges and such that each child of $x$ coincides with the respective child of $y$, we identify the vertices $x$ and $y$, and each outgoing edge of $x$ with the respective outgoing edge of $y$. We repeat this step as long as it is applicable (if $X$ is infinite we may have to apply infinitely many foldings).
    \end{itemize}

    Note that if no foldings of type $1$ are applicable, then applying a folding of type $2$ cannot result in a folding of type $1$ becoming applicable. Thus, at the end of this process, no more foldings (of either type) are applicable. Note that each vertex has either $0$ or $n$ outgoing edges. If it has $n$ outgoing edges, they are labeled by $0,1,\dots,n-1$. Since no foldings of the second type are applicable, if $x_1$ and $x_2$ are distinct fathers, then there exists $i\in\{0,\dots,n-1\}$ such that the $i$-th child of $x_1$ differs from the $i$-th child of $x_2$. Additionally, each vertex $x$ in this graph was originally a vertex in $T_1^i$ or $T_2^i$ for some $i\in\mathcal{I}$, so there is a rooted path ending in $x$. This implies that the resulting graph is an $n$-ary tree-automaton. It is called the \emph{core} of $X$, denoted by $\mathcal{L}(X)$. Note that $\mathcal{L}(X)$ accepts all the elements of $X$.

    The \emph{semi-core} of $X$ is defined to be the automaton obtained after applying all foldings of type $1$, without applying any foldings of the second type. Note that the semi-core of $X$ might not be an $n$-ary tree-automaton, since there could be two distinct fathers with the same children.

    In \cite{GS1}, it was shown that for any subgroup $H$ of $F_n$ and any generating sets $X_1, X_2$ of $H$, it holds that $\mathcal{L}(X_1)=\mathcal{L}(X_2)$. Thus, the core does not depend on the generating set. This leads us to define the core of $H$ as the core of any of its generating sets. We denote the core of $H$ by $\mathcal{L}(H)$. 
\end{Definition}
Note that unlike the core, the semi-cores of different generating sets of a subgroup $H\leq F_n$ might not be isomorphic. Note also that if $H\leq F_n$, then every element in $H$ is accepted by the core $\mathcal{L}(H)$.

\begin{figure}[H]
\centering
\begin{tikzpicture}[->, >=stealth', shorten >=1pt, auto, node distance=2.8cm, semithick]
  \tikzstyle{every state}=[fill=white, draw=black, text=black, minimum size=10mm, font=\small]

  \node[state] (r) {$r$};
  \node[state] (a1) [below=of r] {$a_1$};
  \node[state] (f) [left=of a1, xshift=-1cm] {$f$};
  \node[state] (g) [right=of a1, xshift=1cm] {$g$};
  \node[state] (a0) [below=of a1] {$a_0$};

  \path (r) edge [bend right=15] node[above left] {0} (f)
        (r) edge node {1} (a1)
        (r) edge [bend left=15] node[above right] {2} (g);

  \path (f) edge [loop left] node {0} (f)
        (f) edge [bend right=20] node[below left] {2} (a0)
        (f) edge [bend left=10] node[above] {1} (a1);

  \path (g) edge [bend left=20] node[below right] {0} (a0)
        (g) edge [bend right=10] node[above] {1} (a1)
        (g) edge [loop right] node {2} (g);

  \path (a1) edge [loop right] node {0,2} (a1)
        (a1) edge [bend left=25] node[right] {1} (a0);

  \path (a0) edge [loop left] node {0,2} (a0)
        (a0) edge [bend left=25] node[left] {1} (a1);

\end{tikzpicture}
\caption{The core of $F_3$, $\mathcal{L}(F_3)$.}
\end{figure}
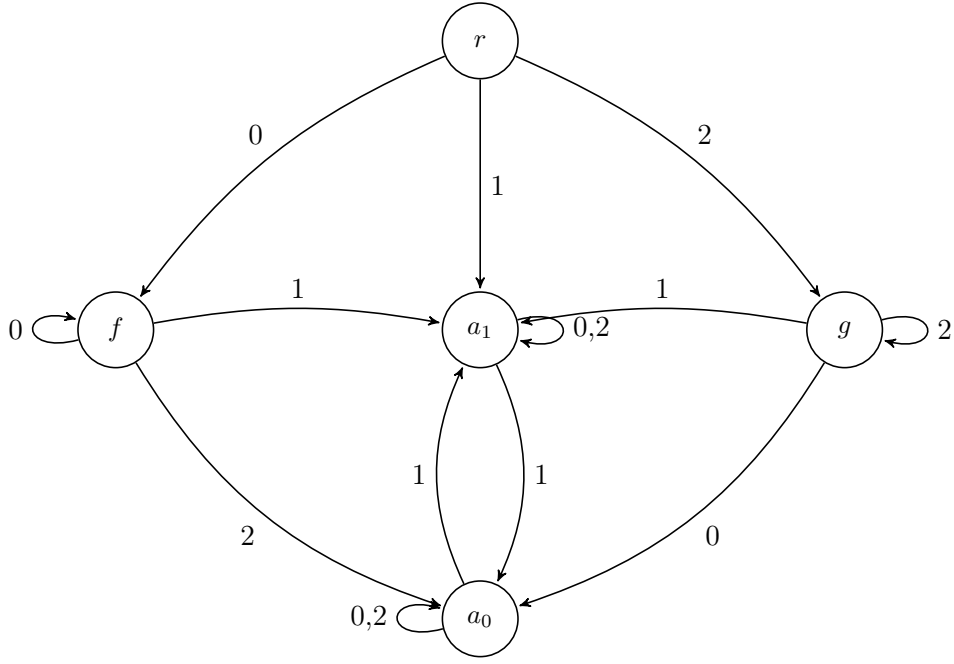


\begin{Definition}\label{d_core_closure}
    Let $H$ be a subgroup of $F_n$. We define the \emph{closure of $H$}, denoted by $\Cl(H)$, as the subgroup of $F_n$ of all elements accepted by $\mathcal{L}(H)$.
\end{Definition}

The closure operation satisfies  standard properties of a closure operator. Namely, if $H$ is a subgroup of $F_n$, then $H\le\Cl(H)$ and $\Cl(\Cl(H))=\Cl(H)$; furthermore, if $H_1\le H_2$ are subgroups of $F_n$, then $\Cl(H_1)\le\Cl(H_2)$.

Let $H$ be a subgroup of $F_n$. A function $f \in  F_n$ is said to be a \textit{piecewise-$H$} function if there is a finite subdivision of the interval $[0, 1]$ such that on each interval $J$ in the subdivision, $f$
coincides with the restriction of some function $h\in H$ to the interval $J$. We note that since all breakpoints of elements in $F_n$ are
$n$-adic fractions, a function $f\in F_n$ is a piecewise-$H$ function if and only if there is an $n$-adic
subdivision of the unit interval into finitely many $n$-adic sub-intervals such that on each $n$-adic interval in the subdivision, $f$ coincides with some function in $H$. Equivalently, a function $f \in F_n$ is a
piecewise-$H$ function if and only if it has a (not necessarily reduced) tree-diagram $(T_1, T_2)$
such that each pair of branches $u\to v$ of $(T_1, T_2)$ is a pair of branches of some element in $H$. The following lemma was proved in \cite{G1} for the case $n=2$. The proof for any $n$ is similar. 

\begin{Lemma}\label{l_pre_piecewiseH}
    Let $H$ be a subgroup of $F_n$. Then the closure of $H$ is the subgroup of $F_n$ of all $n$-adic piecewise-$H$ functions (piecewise-$H$ functions glued in $n$-adic points). In particular, $H$ is closed  
    if and only if every $n$-adic piecewise-$H$ function $f \in F_n$ belongs to $H$. 
\end{Lemma}




The following remark follows directly from Lemma \ref{l_pre_piecewiseH}.

\begin{Remark}\label{r_pre_closure_orbit}
    Let $H$ be a subgroup of $F_n$. The orbits of the action of $H$ on $[0,1]$ coincide with the orbits of the action of $\Cl(H)$.  
\end{Remark}

We will also make use of the following lemma, which was proved in \cite{G1} for $F$ (see \cite[Lemma 4.6]{G1}). The proof adapts readily to the general case; however, we provide an intuitive explanation below.

\begin{Lemma}\label{l_pre_pair_of_branches_H}Let $H\le F_n$, and let $u,v$ be two $n$-ary words that label paths in $\mathcal{L}(H)$. Then $u^+=v^+$ if and only if there exists an integer $k\ge 0$ such that for every $n$-ary word $w$ of length at least $k$, there exists an element $h\in H$ that has the pair of branches $uw\rightarrow vw$.\end{Lemma}

First, suppose there exists $k \ge 0$ such that for every word $w$ with $|w| \ge k$, there is an element $h \in H$ with the pair of branches $uw \to vw$. By Lemma \ref{l_pre_piecewiseH}, it follows that there exists an element $f \in \Cl(H)$ with the pair of branches $u \to v$. So $f$ is accepted by $\mathcal{L}(H)$ and therefore $u^+ = v^+$.

The other direction can be shown by induction on the iterative construction of the core $\mathcal{L}(H)$. Recall that the core is obtained by identifying roots and leaves of tree-diagrams for a generating set, followed by a sequence of foldings. It can be shown that each type of folding preserves the validity of the claim: specifically, a type 1 folding leaves the value of $k$ unchanged, while a type 2 folding increments it by 1.

\subsection{On the derived subgroup $[F_n,F_n]$}
\label{sec:preliminaries:F_n'}

Recall that $F_n$ admits an infinite presentation \cite{B1} with generators $x_0, x_1, x_2, \dots$ subject to the relations:
$$x_k^{-1} x_i x_k = x_{i+n-1} \quad \text{for all } k < i.$$


The abelianization of $F_n$, $F_n / [F_n,F_n]$, is isomorphic to $\mathbb{Z}^n$. Imposing commutativity on the generators yields the relations $x_i = x_{i+n-1}$ for all $i \ge 1$. Therefore, the image of any generator $x_m$ (for $m \ge 1$) from the infinite generating set in the abelianization is simply $[x_j]$, where $j \in \{1, \dots, n-1\}$ is the unique integer such that $m \equiv j \pmod{n-1}$. Thus, the abelianization is the free abelian group generated by the $n$ distinct equivalence classes $[x_0], [x_1], \dots, [x_{n-1}]$. The derived subgroup $[F_n,F_n]$ is precisely the kernel of the natural projection onto this free abelian group.

Recall that the elements of $F_n$ act as piecewise-linear homeomorphisms of $[0,1]$. Since $0$ and $1$ are fixed points, the slope of any commutator must be trivial at $0$ and $1$. This motivates the following definition.

\begin{Definition}\label{d_pre_inner_support}
    We define the \emph{inner-support of $F_n$} to be $\widetilde{F}_n:=\{f\in F_n \mid f'(0^+)=1 \text{ and } f'(1^-)=1\}$. 
\end{Definition}

Note that for each $n$, $[F_n,F_n]\leq \widetilde{F_n}$ (by the chain rule for derivatives at $0^+$ and $1^-$), and that $\widetilde{F}_2=[F_2,F_2]$. In fact, the following is the standard abelianization map of $F_2$: $$f\rightarrow(\log_2f'(0^+),\log_2f'(1^-))$$

We will now show that $\widetilde{F}_n$ is precisely the closure of the derived subgroup of $F_n$. To do so, we use the following Lemma. Although it is well-known, we include a proof for completeness.

\begin{Lemma}\label{l_pre_commutator_pair_of_branches}
    Let $u,v$ be inner $n$-ary words such that $t(u)=t(v)$. Then there is a tree-diagram of an element of $[F_n,F_n]$ with the pair of branches $u\rightarrow v$.
\end{Lemma}

\begin{proof}
    First assume that $[u]$ and $[v]$ are disjoint, and without loss of generality assume that $.u<.v$. Then there exist $3$ $n$-ary words $w_1,w_2,w_3$ such that $t(u)=t(w_1)=t(w_2)=t(w_3)=t(v)$, the intervals $[u],[w_1],[w_2],[w_3],[v]$ are pairwise disjoint and $.u<.w_1<.w_2<.w_3<.v$. Then, by Remark \ref{rPreGeneralPairOfBranches}, there exist elements $f,g\in F_n$ such that $f$ has the pairs of branches $u\rightarrow w_1$ and $w_3\rightarrow w_2$, and $g$ has the pairs of branches $w_1\rightarrow w_2$ and $v\rightarrow w_3$. Then their commutator $[f,g]=fgf^{-1}g^{-1}$ has the pair of branches $u\rightarrow v$.

    In the general case, take an inner $n$-ary word $w$ such that $t(u)=t(w)=t(v)$, $[u]$ and $[w]$ are disjoint and $[v]$ and $[w]$ are disjoint. Then by the first case we get two elements $h_1,h_2\in [F_n,F_n]$ such that $h_1$ has the pair of branches $u\rightarrow w$ and $h_2$ has the pair of branches $w\rightarrow v$. So their product $h_1h_2$ has the pair of branches $u\rightarrow v$.
\end{proof}

\begin{Lemma}
    $\widetilde{F}_n = \Cl ([F_n,F_n])$.
\end{Lemma}

\begin{proof}
    By Lemma \ref{l_pre_commutator_pair_of_branches}, every inner $n$-ary word $w$ labels a path on the core $\mathcal{L}([F_n,F_n])$, and $w^+$ depends only on $t(w)$. Hence $\widetilde{F}_n$ is accepted by $\mathcal{L}([F_n,F_n])$, so $\widetilde{F}_n\subseteq\Cl([F_n,F_n])$.

    In the other direction, let $f\in\Cl([F_n,F_n])$. Then by Lemma \ref{l_pre_piecewiseH} there exist $f_1,\dots,f_k\in [F_n,F_n]$ such that $f$ is an $n$-adic piecewise gluing of $f_1,\dots,f_k$. In particular, $f$ agrees with $f_1$ on some neighborhood of $0$ and $f$ agrees with $f_k$ on some neighborhood of $1$. Note that every element of the derived subgroup fixes some neighborhood of $0$ and of $1$, so $f'(0^+)=f_1'(0^+)=1$ and $f'(1^-)=f_k'(1^-)=1$, therefore $f\in\widetilde{F}_n$. 
\end{proof}

As a corollary of Lemma \ref{l_pre_commutator_pair_of_branches}, we obtain the following well-known result.

\begin{Remark}\label{r_pre_commutator_orbit}
    Let $i\in\{0,\dots,n-2\}$ and let $\alpha\in D_{n,i}$. Then the orbit of $\alpha$ under the action of $[F_n,F_n]$ is $D_{n,i}$.
\end{Remark}

%% file: Thegenerationproblem.tex
\section{On the generation problem}\label{The generation problem}
\label{section:gen_problem}
\subsection{Sufficient Conditions for Generation}
In this section, we follow \cite[section 7]{G1} closely, adapting most of the results for $F_n$.

Let $X$ be a finite subset of $F_n$. We are interested in determining whether $X$ generates $F_n$. Let $H:=\langle X\rangle$ be the subgroup of $F_n$ generated by $X$. We make the following observation: $H=F_n$ if and only if both $H[F_n,F_n]=F_n$ and $[F_n,F_n]\subseteq H$. Verifying the first condition is the same as checking if the image of $X$ under the abelianization map generates $\mathbb{Z}^n$ - and this is decidable. Thus, for the rest of this section, we focus on determining whether $[F_n,F_n]\subseteq H$. We start with a condition to check whether $[F_n,F_n]\subseteq\Cl(H)$.

\begin{Lemma}\label{l_gen_closure_commutator}
    Let $H$ be a subgroup of $F_n$. Then $[F_n,F_n]\subseteq \Cl(H)$ if and only if the following conditions hold:
    \begin{itemize}
        \item[(1)] Every $n$-ary word labels a path on the core $\mathcal{L}(H)$.
        \item[(2)] For any pair of inner $n$-ary words $u,v$, we have $u^+=v^+$ in $\mathcal{L}(H)$ if and only if $t(u)=t(v)$.
    \end{itemize}
\end{Lemma}
\begin{proof}
   First, assume that conditions (1) and (2) hold. Let $f\in [F_{n},F_{n}]$, and let $(T_1,T_2)$ be a tree-diagram representing $f$. Denote by $u_1\rightarrow v_1,\dots,u_k\rightarrow v_k$ the pairs of branches of $(T_1,T_2)$. By condition (1), all the words $u_1,\dots,u_k$ and $v_1,\dots,v_k$ label paths on $\mathcal{L}(H)$. Since $f\in [F_n,F_n]$, we have $f'(0^+)=f'(1^-)=1$, which implies $u_1=v_1$ and $u_k=v_k$. Note that for all $1<i<k$, the labels $u_i$ and $v_i$ are inner $n$-ary words, and $t(u_i)=t(v_i)$ (by Lemma \ref{lPairOfBranches}). By condition (2), $u_i^+=v_i^+$ in $\mathcal{L}(H)$. Therefore, $(T_1,T_2)$ is accepted by $\mathcal{L}(H)$, meaning $f\in\Cl(H)$.
   
   Conversely, assume $[F_n,F_n]\subseteq\Cl(H)$. Let $u,v$ be inner $n$-ary words. Note that if $t(u)\neq t(v)$, then by Lemma \ref{lPairOfBranches} and Lemma \ref{l_pre_pair_of_branches_H} we get that $u^+\neq v^+$. Now assume that $t(u)=t(v)$. By Lemma \ref{l_pre_commutator_pair_of_branches}, there exists a reduced tree-diagram $(T_1,T_2)$ representing an element of $[F_n,F_n]$ with the pair of branches $u\rightarrow v$. Since $[F_n,F_n]\subseteq\Cl(H)$, this tree-diagram is accepted by $\mathcal{L}(H)$, so $u$ and $v$ label paths on $\mathcal{L}(H)$ and $u^+=v^+$. Additionally, for any $m\in\N$, one can construct a reduced tree-diagram of an element in $[F_n,F_n]$ with the pair of branches $0^m\rightarrow 0^m$ (respectively, $(n-1)^m\rightarrow(n-1)^m$). This implies that $0^m$ and $(n-1)^m$ label paths on $\mathcal{L}(H)$. Consequently, every $n$-ary word labels a path on $\mathcal{L}(H)$.
\end{proof}

Conditions (1) and (2) of Lemma \ref{l_gen_closure_commutator} can be reformulated in a way that is easier to verify algorithmically. We call a vertex $x$ in an $n$-ary tree-automaton an \emph{inner vertex} if there is a rooted path $p$ such that $lab(p)$ is an inner $n$-ary word and $p^+=x$. We then obtain the following corollary:

\begin{Corollary}\label{c_gen_closure_commutator}
        Let $H$ be a subgroup of $F_n$. Then $[F_n,F_n]\subseteq \Cl(H)$ if and only if the following conditions hold:
        \begin{itemize}
        \item[(1)] Every vertex in $\mathcal{L}(H)$ is a father vertex (i.e., has $n$ outgoing edges).
        \item[(2)] There are exactly $n-1$ inner vertices in $\mathcal{L}(H)$.
    \end{itemize}
\end{Corollary}

Note that by Corollary \ref{c_gen_closure_commutator}, the problem of determining whether the closure of the subgroup generated by a finite subset $X \subseteq F_n$ contains $[F_n,F_n]$ is decidable.

\begin{Remark}\label{r_gen_action_transitive}
    Let $H$ be a subgroup of $F_n$ and assume $[F_n,F_n]\subseteq\Cl(H)$. Then the following hold:
    \begin{itemize}
        \item[(1)] The action of $H$ on $D_{n,i}$ is transitive for any $i\in\{0,\dots,n-2\}$.
        \item[(2)] For any $x,a,b\in(0,1)$ such that $a<b$, there exists an element $g\in H$ such that $g(x)\in(a,b)$.
    \end{itemize}
\end{Remark}
\begin{proof}
    The assertions follow from Remarks \ref{r_pre_closure_orbit} and \ref{r_pre_commutator_orbit} 
    and the minimality of the action of $[F_n,F_n]$ on $(0,1)$. 
\end{proof}

\begin{Definition}
    Let $H \le F_n$. We say an interval $(a,b) \subseteq [0,1]$ is an \emph{orbital} of $H$ if $H$ fixes $a$ and $b$ (i.e., $h(a)=a$ and $h(b)=b$ for all $h \in H$), but $H$ does not fix any point $x \in (a,b)$ (i.e., for every such $x$, there exists $h \in H$ with $h(x) \neq x$). This notation is standard in the study of groups of functions acting on $\mathbb{R}$.
    
    Furthermore, for any subset $I \subseteq [0,1]$, we denote by $H_I$ the subgroup of $H$ consisting of all elements that fix every point in $I$.
\end{Definition}
\begin{Lemma}\label{L_gen_orbitalu}
    Let $H$ be a subgroup of $F_n$ such that the following hold:
    \begin{itemize}
        \item[$(a)$] $[F_n,F_n]\subseteq \Cl(H)$.
        \item[$(b)$] There exists some $\tilde{f}\in H$ such that $\tilde{f}'(0^+)\neq 1$ and $\tilde{f}'(1^-)=1$.
    \end{itemize}
    Let $u$ be a finite $n$-ary word such that $.u>0$. Then $(0,.u)$ is an orbital of $H_{[u]}$.
\end{Lemma}
\begin{proof}
    Since $\tilde{f}'(1^-)=1$ and $\tilde{f}'(0^+)\neq 1$, there exists a minimal $b\in (0,1)$ such that $\tilde{f}$ fixes $[b,1]$, and some $a'\in(0,b)$ such that $\tilde{f}$ does not fix any point in $(0,a')$.
    
    We will first show that $(0,b)$ is an orbital of the group $H_{[b,1]}$. Let $a:=\inf\{c>0 \mid H_{[b,1]} \text{ fixes } c\}$. Since $\tilde{f}$ does not fix any point in $(0,a')$, we have $0<a'\le a\le b$. Note also that $H_{[b,1]}$ fixes $a$ (by continuity). By the definition of $a$, $H_{[b,1]}$ does not fix any $c\in(0,a)$, so $(0,a)$ is an orbital of $H_{[b,1]}$. Assume for contradiction that $a\neq b$. Then $a<b$. Let $\operatorname{Fix}(H_{[b,1]}):=\{x\in [a,1] \mid H_{[b,1]} \text{ fixes } x\}$.
    
    We will now show that there exists some $h\in H$ such that $h(a)>a$ and $h(a)\notin \operatorname{Fix}(H_{[b,1]})$. Assume for contradiction that for all $h\in H$, if $h(a)>a$ then $h(a)\in \operatorname{Fix}(H_{[b,1]})$. Then $orb_H(a)\cap(a,1)\subseteq \operatorname{Fix}(H_{[b,1]})$, where $orb_H(a)$ is the orbit of $a$ under the action of $H$. By Remark \ref{r_gen_action_transitive}, $orb_H(a)$ is dense in $[0,1]$, so $\operatorname{Fix}(H_{[b,1]})$ is dense in $[a,1]$. Therefore, each element in $H_{[b,1]}$ fixes a dense subset of $[a,1]$, and by continuity, it must fix $[a,1]$ point-wise. Note that since $a<b$, $H_{[a,1]}\subseteq H_{[b,1]}$, and by the above, $H_{[a,1]}=H_{[b,1]}$. Since $a<b$, by Remark \ref{r_gen_action_transitive}, there exists $g\in H$ such that $g(a)\in (a,b)$. Consequently, $H_{[a,1]}\subseteq H_{[g(a),1]}\subseteq H_{[b,1]}=H_{[a,1]}$, which implies $H_{[g(a),1]}=H_{[a,1]}$. Note that since $g(a)>a$, it follows that $g^{-1}(a)\in (0,a)$. Recall that $(0,a)$ is an orbital of $H_{[a,1]}=H_{[b,1]}$, so $H_{[a,1]}$ does not fix $g^{-1}(a)$. Therefore, $g^{-1}H_{[a,1]}g=H_{[g(a),1]}$ does not fix $a$, which is a contradiction since $H_{[g(a),1]}=H_{[a,1]}$.
    
    Thus, there exists some $h\in H$ such that $h(a)>a$ and $H_{[b,1]}$ does not fix $h(a)$. Note that $h^{-1}(a)\in(0,a)$, so $H_{[b,1]}$ does not fix $h^{-1}(a)$. We consider the following cases:
    \begin{itemize}
        \item[$(1)$] $h(b)<b$. Then $H_{[h(b),1]}\subseteq H_{[b,1]}$. However, since $H_{[b,1]}$ does not fix $h^{-1}(a)$, $h^{-1}H_{[b,1]}h=H_{[h(b),1]}$ does not fix $a$, which is a contradiction.
        \item[$(2)$] $h(b)\ge b$. Then $h^{-1}(b)\le b$, and so $H_{[h^{-1}(b),1]}\subseteq H_{[b,1]}$. However, since $H_{[b,1]}$ does not fix $h(a)$, $hH_{[b,1]}h^{-1}= H_{[h^{-1}(b),1]}$ does not fix $a$, which is also a contradiction.
    \end{itemize}
    Since we reached a contradiction in both cases, we must have $a=b$, meaning $(0,b)$ is an orbital of $H_{[b,1]}$.
    
    We now return to the original claim, which is to show that $(0,.u)$ is an orbital of $H_{[u]}$. Note that $H_{[u]}$ fixes $.u$, so we only need to show that for all $x\in(0,.u)$ there exists some $h_1\in H_{[u]}$ such that $h_1(x)\neq x$.
    
    Let $x\in (0,.u)$. By Remark \ref{r_gen_action_transitive}, there exists some $h_2\in H$ such that $x<h_2(b)<.u$. Then $(0,h_2(b))$ is an orbital of $H_{[h_2(b),1]}=h_2^{-1}H_{[b,1]}h_2$, so there exists some $h_1 \in H_{[h_2(b),1]}\subseteq H_{[u]}$ such that $h_1(x)\neq x$. Since $h_2(b)<.u$, we have that $[u]\subseteq [h_2(b),1]$, so $h_1\in H_{[u]}$, which completes the proof.
\end{proof}

The following lemma is well known (for example, a proof appears in \cite{Bleak2008}).
\begin{Lemma}\label{L_gen_orbitalinc}
     Let $H$ be a subgroup of $F_n$. Let $(a,b)$ be an orbital of $H$. Then for any $x,y\in(a,b)$ there exists an element $h\in H$ such that $h(x)>y$.
\end{Lemma}
\begin{Lemma}\label{L_gen_interval}
        Let $H$ be a subgroup of $F_n$ such that the following hold.
    \begin{itemize}
        \item[$(a)$] $[F_n,F_n]\subseteq \Cl(H)$.
        \item[$(b)$] There exists some $\tilde{f}\in H$ such that $\tilde{f}'(0^+)\neq 1$ and $\tilde{f}'(1^-)=1$.
    \end{itemize}
    Let $u'$ be a finite $n$-ary word and let $a,b\in(0,1)$, such that $a<b$, be $n$-adic. Then there exists $g\in H$ such that $g([a,b])\subseteq[u']$.
\end{Lemma}
\begin{proof}
    Denote $i:=t(b)-t(u')\Mod{n-1}$ and $v\equiv u'i0(n-1)$. Then $v$ is an inner $n$-ary word and $t(v)=t(b)$. By Remark \ref{r_gen_action_transitive}, there exists $h\in H$ such that $h(b)=.v$. So $h(a)<h(b)=.v$. Let $u\equiv u'i1$ and note that $.u=.v(n-1)^\N$. By Lemma \ref{L_gen_orbitalu} $(0,.u)$ is an orbital of $H_{[u]}$, so $h(a)<h(b)=.v<.u$. By Lemma \ref{L_gen_orbitalinc}, there exists an element $f\in H_{[u]}$ such that $f(h(a))>.v$. Note that $f(h(b))=f(.v)<f(.u)=.u$. So for $g:=hf$, $g([a,b])\subseteq[.v,.u]\subseteq[u']$.
\end{proof}
\begin{Definition}
    Let $H\le F_n$ and let $s,w$ be finite $n$-ary words. We say that $w$ is \emph{$H$-equivalent to a $0$-extension of $s$} if there exists some $h\in H$ and some natural $k$ such that $h$ has the pair of branches $w\rightarrow s0^k$.
\end{Definition}
For an $n$-ary word $s'$, we write $.s'(n-1)^{\mathbb{N}}$ to denote the right end point of $[s']$. We prove the following two technical lemmas, which we will use later.
\begin{Lemma}\label{L_gen_fix_change_slope}
         Let $H$ be a subgroup of $F_n$ such that the following hold:
    \begin{itemize}
        \item[$(a)$] $[F_n,F_n]\subseteq \Cl(H)$.
        \item[$(b)$] There exists some $\tilde{f}\in H$ such that $\tilde{f}'(0^+)\neq 1$ and $\tilde{f}'(1^-)=1$.
    \end{itemize}
    Let $i\in\{0,\dots,n-2\}$, and assume there exist an element $h\in H$ and inner $n$-ary words $s,s',w$ such that the following hold:
    \begin{itemize}
        \item[$(1)$] $t(s)=i$.
        \item[$(2)$] $.s'(n-1)^\N=.s$.
        \item[$(3)$] $h$ fixes $[s']$ pointwise and has the pair of branches $s0\rightarrow s$.
        \item[$(4)$] $w$ is $H$-equivalent to a $0$-extension of $s$.
    \end{itemize}
    Let $\alpha\in (0,.w)$ and let $a,b\in \N$. Then there exists an element $g_l\in H$ which is a conjugate of a power of $h$ such that $g_l$ fixes $[\alpha,.w]$ pointwise and $g_l$ has the pair of branches $w0^a\rightarrow w0^b$.
\end{Lemma}
\begin{proof}
    By condition $(4)$, there exist an element $f\in H$ and an integer $c \ge 0$ such that $f$ has the pair of branches $w\rightarrow s0^c$. By Lemma \ref{L_gen_orbitalu}, $(0,.s)$ is an orbital of $H_{[s]}$. Note that $f(\alpha)<f(.w)=.s$, so by Lemma \ref{L_gen_orbitalinc} there exists some $g\in H_{[s]}$ such that $g(f(\alpha))>.s'$. Denote $g_l:=fgh^{a-b}g^{-1}f^{-1}$. Note that since $f$ has the pair of branches $w0^a\rightarrow s0^{a+c}$, $g$ has the pair of branches $s0^{a+c}\rightarrow s0^{a+c}$, and $h^{a-b}$ has the pair of branches $s0^{a+c}\rightarrow s0^{b+c}$, it follows that $g_l$ has the pair of branches $w0^a\rightarrow w0^b$. Additionally, for all $x\in [\alpha,.w]$, $g(f(x))\in [.s',.s]=[s']$. Therefore, for all $x\in[\alpha,.w]$, it holds that $h^{a-b}(g(f(x)))=g(f(x))$, so $g_l(x)=x$. Thus, $g_l$ is as required.
\end{proof}
\begin{Lemma}\label{L_gen_coincide_change_slope}
         Let $H$ be a subgroup of $F_n$ such that the following hold:
    \begin{itemize}
        \item[$(a)$] $[F_n,F_n]\subseteq \Cl(H)$.
        \item[$(b)$] There exists some $\tilde{f}\in H$ such that $\tilde{f}'(0^+)\neq 1$ and $\tilde{f}'(1^-)=1$.
    \end{itemize}
    Let $i\in\{0,\dots,n-2\}$, and assume there exist elements $g,h\in H$ and inner $n$-ary words $s,s',w_1,w_2$ such that the following hold:
    \begin{itemize}
        \item[$(1)$] $t(s)=i$.
        \item[$(2)$] $.s'(n-1)^\N=.s$.
        \item[$(3)$] $h$ fixes $[s']$ pointwise and has the pair of branches $s0\rightarrow s$.
        \item[$(4)$] $w_1,w_2$ are $H$-equivalent to a $0$-extension of $s$.
        \item[$(5)$] There exist $c_1,c_2\in\N$ such that $g$ has the pair of branches $w_10^{c_1}\rightarrow w_20^{c_2}$.
    \end{itemize}
    Let $\alpha_1\in (0,.w_1)$ and let $d_1,d_2\in\N$. Then there are elements $g_l,g_r\in H$ which are conjugates of powers of $h$, such that the element $g_1:=g_lgg_r$ coincides with $g$ on $[\alpha_1,.w_1]$ and has the pair of branches $w_10^{d_1}\rightarrow w_20^{d_2}$.
\end{Lemma}
\begin{proof}
    By applying Lemma \ref{L_gen_fix_change_slope} with the element $h$, the words $s,s',w\equiv w_1$, and the parameters $\alpha\equiv\alpha_1, a\equiv d_1, b\equiv c_1$, there exists some $g_l\in H$ which is a conjugate of a power of $h$ such that $g_l$ fixes $[\alpha_1,.w_1]$ pointwise and has the pair of branches $w_10^{d_1}\rightarrow w_10^{c_1}$. Note that $g(\alpha_1)<g(.w_1)=.w_2$, so by applying Lemma \ref{L_gen_fix_change_slope} again with the element $h$, the words $s,s',w\equiv w_2$, and the parameters $\alpha\equiv g(\alpha_1), a\equiv c_2, b\equiv d_2$, we obtain an element $g_r\in H$ which is a conjugate of a power of $h$ such that $g_r$ fixes $[g(\alpha_1),.w_2]$ pointwise and has the pair of branches $w_20^{c_2}\rightarrow w_20^{d_2}$. 
    
    Let $g_1:=g_lgg_r$. Since $g_l$ fixes $[\alpha_1,.w_1]$ and $g_r$ fixes $[g(\alpha_1),.w_2]$, it follows that $g_1$ coincides with $g$ on $[\alpha_1,.w_1]$. Since $g_l$ has the pair of branches $w_10^{d_1}\rightarrow w_10^{c_1}$, $g$ has the pair of branches $w_10^{c_1}\rightarrow w_20^{c_2}$, and $g_r$ has the pair of branches $w_20^{c_2}\rightarrow w_20^{d_2}$, we get that $g_1$ has the pair of branches $w_10^{d_1}\rightarrow w_20^{d_2}$.
\end{proof}
\begin{Remark}\label{r_gen_words_equiv}
         Let $H$ be a subgroup of $F_n$ such that the following hold.
    \begin{itemize}
        \item[$(a)$] $[F_n,F_n]\subseteq \Cl(H)$.
        \item[$(b)$] There exists an inner $n$-ary word $u$, such that for every two $n$-ary words $v_1,v_2$ satisfying $t(v_1)=t(v_2)$, there exists an element $h_{uv_1,uv_2}\in H$ with a pair of branches $uv_1\rightarrow uv_2$.
    \end{itemize}
    Let $u_1,u_2$ be two inner $n$-ary words. Then there exists some natural $k\ge 0$ such that for every two $n$-ary words $w_1,w_2$: if $|w_1|,|w_2|\ge k$ and $t(u_1w_1)=t(u_2w_2)$, then there exists an element $h_{u_1w_1,u_2w_2}\in H$ with the pair of branches $u_1w_1\rightarrow u_2w_2$.
\end{Remark}
\begin{proof}
    Take some $i_1,i_2\in\{0,\dots,n-1\}$ such that $t(u_1)=t(ui_1),t(u_2)=t(ui_2)$. By Lemma \ref{l_gen_closure_commutator}, $u_1^+=(ui_1)^+$ and $u_2^+=(ui_2)^+$ in $\mathcal{L}(H)$. By Lemma \ref{l_pre_pair_of_branches_H}, with the words $u_1,ui_1$ there exists some $k_1\in \N$ such that for every $n$-ary word $w_1$ of length greater or equal to $k_1$, there exists an element $h_{u_1w_1,ui_1w_1}\in H$ with the pair of branches $u_1w_1\rightarrow ui_1w_1$. Similarly there exists some $k_2\in\N$ such that for every $n$-ary word $w_2$ with length greater or equal to $k_2$ there exists an element $h_{ui_2w_2,u_2w_2}\in H$ with the pair of branches $ui_2w_2\rightarrow u_2w_2$.
    
    Set $k:=\max\{k_1,k_2\}$, and let $w_1,w_2$ be $n$-ary words of length greater or equal to $k$ such that $t(u_1w_1)=t(u_2w_2)$. Then, $t(ui_1w_1)=t(ui_2w_2)$, which implies $t(i_1w_1)=t(i_2w_2)$. By condition $(b)$, there is an element $h_{ui_1w_1,ui_2w_2}\in H$ with the pair of branches $ui_1w_1\rightarrow ui_2w_2$. Set $h_{u_1w_1,u_2w_2}:=h_{u_1w_1,ui_1w_1}h_{ui_1w_1,ui_2w_2}h_{ui_2w_2,u_2w_2}\in H$. Then $h_{u_1w_1,u_2w_2}$ has the pair of branches $u_1w_1\rightarrow u_2w_2$.
\end{proof}
We will later use the following corollary from the last remark.
\begin{Corollary}\label{c_gen_ext_equiv}
    Let $H$ be a subgroup of $F_n$ such that the following hold:
    \begin{itemize}
        \item[$(a)$] $[F_n,F_n]\subseteq \Cl(H)$.
        \item[$(b)$] There exists an inner $n$-ary word $u$ such that for every two $n$-ary words $v_1,v_2$ satisfying $t(v_1)=t(v_2)$, there exists an element $h_{uv_1,uv_2} \in H$ with the pair of branches $uv_1\rightarrow uv_2$.
    \end{itemize}
    Let $u_1,u_2$ be two inner $n$-ary words. Then there exists an integer $k\ge 0$ such that for every $n$-ary word $w$, if $|w|\ge k$ and $t(u_1w)=t(u_2)$, then $u_1w$ is $H$-equivalent to a $0$-extension of $u_2$.
\end{Corollary}
\begin{Lemma}\label{L_gen_commutator_coincide}
    Let $H\le F_n$ be a subgroup such that the following conditions hold: 
    \begin{itemize}
        \item[$(1)$] $[F_n,F_n]\subseteq \Cl (H)$.
        \item[$(2)$] There exists some $\tilde{f}\in H$ such that $\tilde{f}'(0^+)\neq 1$ and $\tilde{f}'(1^-)=1$.
        \item[$(3)$] For each $i\in\{0,\dots,n-2\}$, there exist inner $n$-ary words $s_i,s_i'$ and an element $h_i\in H$, such that the following hold: $.s_i\in D_{n,i}$, $.s_i=.s'_i(n-1)^\N$, $h_i$ fixes $[s_i']$ pointwise, and $h_i$ has the pair of branches $s_i0\rightarrow s_i$.
        \item[$(4)$] There exists an inner $n$-ary word $u$, such that for every two $n$-ary words $v_1,v_2$ satisfying $t(v_1)=t(v_2)$, there exists an element $h_{uv_1,uv_2}\in H$ with a pair of branches $uv_1\rightarrow uv_2$.
    \end{itemize}
    Let $f\in\widetilde{F}_n$, and $0<a<b<1$. Then there exists an element $g_1\in H\cap\widetilde{F}_n$ such that $g_1$ coincides with $f$ on $[a,b]$ (where $\widetilde{F}_n$ is the inner-support subgroup of $F_n$, defined in Definition \ref{d_pre_inner_support}).
\end{Lemma}
\begin{proof}
    Since $f\in\widetilde{F}_n$, it fixes a small neighborhood of $0$ and a small neighborhood of $1$, so there are some $a_1<a$ and $b_1>b$ such that $f$ fixes $[0,a_1]$ and $[b_1,1]$. Let $(T_1,T_2)$ be a tree-diagram representing $f$, and denote its pairs of branches by $u_0\rightarrow v_0,\dots,u_{k-1}\rightarrow v_{k-1}$. We can assume that $a_1\notin[u_0]\cup[u_1]$ and $b_1\notin[u_{k-2}]\cup[u_{k-1}]$ (if not, we can take another tree-diagram of $f$ in which this holds). Thus $u_0\equiv v_0, u_1\equiv v_1, u_{k-2}\equiv v_{k-2}, u_{k-1}\equiv v_{k-1}$.
    
    We will first show that there exists some $g\in H$ that is a product of conjugates of $h_i$ (for various $i\in\{0,\dots,n-2\}$), such that for each $j\in\{1,\dots,k-2\}$, $g$ has the pair of branches $u_j\rightarrow v_j$. Let $f_0=id$ be the identity element. We construct elements $f_1,\dots,f_{k-2}$ recursively, such that for each $j\in\{1,\dots,k-2\}$, the following conditions hold:
    \begin{itemize}
        \item[$(A)$] $f_j$ has the pair of branches $u_j\rightarrow v_j$.
        \item[$(B)$] $f_j$ coincides with $f_{j-1}$ on $[.u_1,.u_j]$.
        \item[$(C)$] $f_{j+1}=l_jf_jr_j$, where $l_j$ and $r_j$ are conjugate to a power of $h_{j+1\pmod{n-1}}$.
    \end{itemize}
    Then for $g:=f_{k-2}$ we will have the desired result.
    
    For $j=1$, we take $f_1=id$, and note that the conditions hold. Now let $1\le j\le k-3$ and assume conditions $(A),(B),(C)$ hold for $f_j$. Since $j\le k-3$, $u_{j+1},v_{j+1}$ are inner $n$-ary words. Thus, there exist $n$-ary words $p,q$ and integers $d_1,d_2\ge 0$ such that $u_{j+1}\equiv p0^{d_1}$ and $v_{j+1}\equiv q0^{d_2}$. Note that either $p$ is an inner $n$-ary word, or $p \equiv (n-1)^{m_1}$ for some integer $m_1 > 0$ (in which case we must have $d_1 > 0$, since $u_{j+1}$ is an inner word). The analogous statement holds for $q$ with respect to some integer $m_2$ and $d_2$. Note that $t(p)=t(u_{j+1}) \equiv j+1 \pmod{n-1}$ and $t(q)=t(v_{j+1}) \equiv j+1 \pmod{n-1}$. Let $(T_1',T_2')$ be a tree-diagram of $f_j$ that has the pair of branches $u_j\rightarrow v_j$. Note that there exist integers $c_1,c_2\ge 0$, such that $(T_1',T_2')$ has a pair of branches $p0^{c_1}\rightarrow q0^{c_2}$.
    
    We first consider the case where $p$ and $q$ are inner $n$-ary words. By Corollary \ref{c_gen_ext_equiv}, there exists some integer $k\ge 0$ such that both $p0^k$ and $q0^k$ are $H$-equivalent to a $0$-extension of $s_{j+1\pmod{n-1}}$. By adding the suffix $0^k$ to $p0^{c_1},q0^{c_2}$, we get that $f_j$ has the pair of branches $p0^{c_1+k}\rightarrow q0^{c_2+k}$. By Lemma \ref{L_gen_coincide_change_slope} with $i=j+1\pmod{n-1}$, $g=f_j$, $h=h_{j+1\pmod{n-1}}$, $s=s_{j+1\pmod{n-1}}$, $s'=s'_{j+1\pmod{n-1}}$, $w_1=p$, $w_2=q$, $c_1=c_1+k$, $c_2=c_2+k$, $d_1,d_2$, and $\alpha_1=.u_1$, we obtain elements $l_j,r_j\in H$ which are conjugates of powers of $h_{j+1\pmod{n-1}}$ such that $l_jf_jr_j$ coincides with $f_j$ on $[.u_1,.u_{j+1}]$, and has the pair of branches $p0^{d_1}\rightarrow q0^{d_2}$. Note that $p0^{d_1}\equiv u_{j+1}$ and $q0^{d_2}\equiv v_{j+1}$, so $l_jf_jr_j$ has the pair of branches $u_{j+1}\rightarrow v_{j+1}$. We set $f_{j+1}:=l_jf_jr_j$, and note that it satisfies conditions $(A),(B),(C)$.

    If $p$ (respectively, $q$) is not an inner $n$-ary word, then we have $p\equiv (n-1)^{m_1}$ and $d_1>0$ (respectively, $q\equiv (n-1)^{m_2}$ and $d_2>0$). Then $c_1>0$ (respectively, $c_2>0$). So we let $p'\equiv p0$, $c_1'=c_1-1$, $d_1'=d_1-1$ (respectively, $q'\equiv q0$, $c_2'=c_2-1$, $d_2'=d_2-1$), and note that $p'$ is an inner $n$-ary word and $c_1',d_1'\ge 0$ (respectively, $q'$ is an inner $n$-ary word and $c_2',d_2'\ge 0$). Thus, by a similar argument to the one given in the previous paragraph, we get elements $l_j,r_j$ that are conjugate to powers of $h_{j+1\pmod{n-1}}$ such that $l_jf_jr_j$ coincides with $f_j$ on $[.u_1,.u_{j+1}]$, and has the pair of branches $p0^{d_1}\rightarrow q0^{d_2}$.
    
    We take $g:=f_{k-2}\in H$ and note that $g$ has the pairs of branches $u_j\rightarrow v_j$ for $1\le j\le k-2$. So it agrees with $f$ on $[.u_1,.u_{k-1}]$. Since $g$ is a product of conjugates of powers of $h_i$ (for various $i\in\{0,\dots,n-2\}$), for each $i\in\{0,\dots,n-2\}$ there exists an integer $l_i$ (which is the sum of powers of $h_i$ that appeared in this construction of $g$) such that $g'(0^+)=\prod_{i=0}^{n-2}{(h_i'(0^+))^{l_i}}$ and $g'(1^-)=\prod_{i=0}^{n-2}{(h_i'(1^-))^{l_i}}$. By applying Lemma \ref{L_gen_interval}, we get that for all $i\in\{0,\dots,n-2\}$ there exists an element $o_i\in H$ such that $o_i^{-1}([.u_1,.u_{k-1}])\subseteq [s'_i]$.
    
    Let $g_1:=g\cdot\prod_{i=0}^{n-2}{(h_i^{o_i})^{-l_i}}$. Note that $g_1\in H$. Additionally, note that $g_1'(0^+)=g'(0^+)\cdot\prod_{i=0}^{n-2}{((h_i^{o_i})'(0^+))^{-l_i}}=\prod_{i=0}^{n-2}{(h_i'(0^+))^{l_i}}\cdot\prod_{i=0}^{n-2}{(h_i'(0^+))^{-l_i}}=1$. Similarly, $g_1'(1^-)=1$. So $g_1\in\widetilde{F}_n$. Finally, for each $i\in\{0,\dots,n-2\}$, $h_i$ fixes $[s'_i]$, so $h_i^{o_i}$ fixes $[.u_1,.u_{k-1}]$. Thus $g_1$ coincides with $g$ (and therefore with $f$) on $[.u_1,.u_{k-1}]$. Since $[a,b]\subseteq [a_1,b_1]\subseteq [.u_1,.u_{k-1}]$, we are done.
\end{proof}
\begin{Lemma}\label{L_gen_contains_commutator}
    Let $H\le F_n$ be a subgroup such that the following conditions hold: 
    \begin{itemize}
        \item[$(1)$] $[F_n,F_n]\subseteq \Cl (H)$.
        \item[$(2)$] There exists some $\tilde{f}\in H$ such that $\tilde{f}'(0^+)\neq 1$ and $\tilde{f}'(1^-)=1$.
        \item[$(3)$] For each $i\in\{0,\dots,n-2\}$, there exist inner $n$-ary words $s_i,s_i'$ and an element $h_i\in H$, such that the following hold: $.s_i\in D_{n,i}$, $.s_i=.s'_i(n-1)^\N$, $h_i$ fixes $[s_i']$ pointwise, and $h_i$ has the pair of branches $s_i0\rightarrow s_i$.
        \item[$(4)$] There exists an inner $n$-ary word $u$, such that for every two $n$-ary words $v_1,v_2$ satisfying $t(v_1)=t(v_2)$, there exists an element $h_{uv_1,uv_2}\in H$ with a pair of branches $uv_1\rightarrow uv_2$.
    \end{itemize}
    Then $H$ contains $[F_n,F_n]$.
\end{Lemma}
\begin{proof}
    Denote by $\restr{F_n}{[a,b]}$ the subgroup of elements in $F_n$ with support contained in $[a,b]$. It is a standard fact that for $a\in D_{n,0},b\in D_{n,1}$ (if $n=2$, denote $D_{2,1}:=D_{2,0}$) such that $a<b$, $\restr{F_n}{[a,b]}$ is isomorphic to $F_n$ (an isomorphism can be obtained, for example by a conjugation with a piecewise linear function - see \cite{B1,CFP1996}).

    Note that the union of the derived subgroups $[\restr{F_n}{[a,b]},\restr{F_n}{[a,b]}]$ over all $a<b$ is a normal subgroup of $[F_n,F_n]$. Since $[F_n,F_n]$ is simple (\cite{B1}), this union equals $[F_n,F_n]$. So it suffices to show that $H$ contains $[\restr{F_n}{[a,b]}, \restr{F_n}{[a,b]}]$ for any given $a<b$ such that $a\in D_{n,0},b\in D_{n,1}$. 

    Let $a\in D_{n,0}, b\in D_{n,1}$ such that $a<b$. Since $\restr{F_n}{[a,b]} \cong F_n$, we can choose a generating set $\{y_0, \dots, y_{n-1}\}$ for it. We will construct elements $g_0, \dots, g_{n-1} \in H$ such that the following conditions hold:
    \begin{itemize}
        \item[$(i)$] For each $j\in\{0,\dots,n-1\}$, $g_j$ coincides with $y_j$ on $[a,b]$.
        \item[$(ii)$] For every pair $i<j\in\{0,\dots,n-1\}$, the intersection between the support of $g_i$ and the support of $g_j$ is contained in $[a,b]$.
    \end{itemize}
    These conditions imply that for all $i\neq j\in\{0,\dots,n-1\}$, $[g_i,g_j] = [y_i,y_j]$. Consequently, the derived subgroup of the group generated by $\{g_0,\dots,g_{n-1}\}$ equals the derived subgroup of $\restr{F_n}{[a,b]}$ (this holds because, by conditions $(i)$ and $(ii)$, $g_i$ and $g_j$ commute outside $[a,b]$ for all $i \neq j$).
    
    By Lemma \ref{L_gen_commutator_coincide}, there exists some element $g_0\in H\cap\widetilde{F}_n$ that coincides with $y_0$ on $[a,b]$. Let $a_0,b_0\in (0,1)$ be such that $[a,b]\subseteq (a_0,b_0)$ and the support of $g_0$ is contained in $[a_0,b_0]$. For $i\in\{1,\dots,n-1\}$, we will construct the element $g_i$ and numbers $a_i,b_i\in(0,1)$ recursively. Assume we have already constructed $a_{i-1},b_{i-1},g_{i-1}$. We apply Lemma \ref{L_gen_commutator_coincide} to get an element $g_i\in H\cap\widetilde{F}_n$ that coincides with $y_i$ on $[a_{i-1},b_{i-1}]$. Let $a_i,b_i\in(0,1)$ be such that $[a_{i-1},b_{i-1}]\subseteq (a_i,b_i)$ and the support of $g_i$ is contained in $[a_i,b_i]$.
    
    For convenience, we denote $a_{-1}:=a$ and $b_{-1}:=b$. Observe that for each $i\in\{0,\dots,n-1\}$, $g_i$ coincides with $y_i$ on $[a_{i-1},b_{i-1}]$. Since $[a,b]\subseteq[a_{i-1},b_{i-1}]$, condition $(i)$ holds. Furthermore, for each $i\in\{0,\dots,n-1\}$, $\operatorname{supp}(g_i)\subseteq ([a_i,b_i]\setminus[a_{i-1},b_{i-1}])\cup[a,b]$. Therefore, for any $i<j\in\{0,\dots,n-1\}$, $\operatorname{supp}(g_i)\cap \operatorname{supp}(g_j)\subseteq [a,b]$, meaning condition $(ii)$ holds.
\end{proof}
We provide another way of stating condition $(3)$ of the previous lemma.
\begin{Remark}\label{r_gen_tuples_equiv_conditions}
    Let $H\le F_n$, let $h\in H$, and let $i\in\{0,\dots,n-2\}$. Then there exists $\alpha_i\in D_{n,i}$ such that $h(\alpha_i)=\alpha_i,h'(\alpha_i^-)=1,h'(\alpha_i^+)=n$ if and only if there exist inner $n$-ary words $s_i,s_i'$ such that the following holds. \begin{itemize}
        \item[$(1)$]$t(s_i)=i$
        \item[$(2)$]$.s_i'(n-1)^\N=.s_i$
        \item[$(3)$]$h$ fixes $[s_i']$ pointwise and has the pair of branches $s_i0\rightarrow s_i$.
    \end{itemize}
\end{Remark}
We get the following theorem as a corollary of Lemma \ref{L_gen_contains_commutator}:
\begin{Theorem}\label{t_Gen_main_theorem}
    Let $H$ be a subgroup of $F_n$. Then $H=F_n$ if and only if the following conditions hold:
    \begin{itemize}
        \item[$(1)$] $[F_n,F_n]\subseteq\Cl (H)$
        \item[$(2)$] $H[F_n,F_n]=F_n$
        \item[$(3)$] For each $i\in\{0,\dots,n-2\}$, there exist an element $h_i\in H$, and $\alpha_i\in D_{n,i}$, such that $h_i(\alpha_i)=\alpha_i,h_i'(\alpha_i^-)=1,h_i'(\alpha_i^+)=n$.
        \item[$(4)$] There exists an inner $n$-ary word $u$, such that for every two $n$-ary words $v_1,v_2$ satisfying $t(v_1)=t(v_2)$, there exists an element $h_{uv_1,uv_2}\in H$ with a pair of branches $uv_1\rightarrow uv_2$.
    \end{itemize}
\end{Theorem}
\begin{proof}
    First, note that if $H=F_n$, conditions $(1)-(4)$ hold. Now assume conditions $(1)-(4)$ hold. Note that conditions $(1),(3),(4)$ are the same as their respective conditions in Lemma \ref{L_gen_contains_commutator}. 
    
    Note that there exists some element $\tilde{f}\in F_n$ such that $\tilde{f}'(0^+)\neq 1$ and $\tilde{f}'(1^-)=1$. Since $[F_n,F_n]\subseteq \widetilde{F}_n$, all elements of $[F_n,F_n]$ have slope $1$ at the points $0$ and $1$. By condition $(2)$ of this theorem there exists an element of $H$ with the same slopes as $\tilde{f}$ at $0$ and $1$, so condition $(2)$ of Lemma \ref{L_gen_contains_commutator} holds. Therefore, by Lemma \ref{L_gen_contains_commutator}, $[F_n,F_n]\subseteq H$. By condition $(2)$, $H=F_n$.
\end{proof}
Note that as discussed earlier in this section, conditions $(1)$ and $(2)$ of Theorem \ref{t_Gen_main_theorem} are verifiable. In the next section we will provide an algorithm that decides if condition $(3)$ holds. We do not have an algorithm deciding whether condition $(4)$ holds. However, we provide a stronger condition which is verifiable.

Recall that given a subset $X$ of $F_n$, the \emph{semi-core} of $X$, denoted by $\mathcal{L}_{sem}(X)$, is an automaton constructed in a similar way to the core of $X$. Namely, we identify all the roots and we identify all the pairs of leaves at the end of the pairs of branches of elements of $X$. Finally, we apply all possible foldings of type 1.

Note that given a subset $X$ of $F_n$, $\mathcal{L}_{sem}(X)$ might not be an $n$-ary tree-automaton (since foldings of type $2$ may be applicable to $\mathcal{L}_{sem}(X)$). However, it is true that every $n$-ary word labels at most one (rooted) path on the semi-core of $X$, so the notion of a vertex at the end of a path can be defined like in an $n$-ary tree-automaton (i.e., if $u$ is an $n$-ary word which labels a path on $\mathcal{L}_{sem}(X)$, $u^+$ is the vertex at the end of $u$). The following claim was shown in \cite{G1} (it was shown for the case of $F$, the proof for $F_n$ is nearly identical).
\begin{Remark}\label{r_gen_semi_core_pair_of_branches}
  Let $X$ be a subset of $F_n$ and let $H$ be the subgroup of $F_n$ generated by $X$. Let $u,v$ be $n$-ary words which label paths on $\mathcal{L}_{sem}(X)$. If $u^+=v^+$ then there exists an element $h\in H$ which has the pair of branches $u\rightarrow v$.
\end{Remark}
\begin{Remark}\label{r_gen_semi_core_paths}
  Let $X$ be a subset of $F_n$. Let $u$ be a finite $n$-ary word which labels a path on $\mathcal{L}(X)$. Then $u$ also labels a path on $\mathcal{L}_{sem}(X)$.
\end{Remark}
\begin{proof}
    Let $p=v_1\dots v_l$ be the path labeled by $u$ on $\mathcal{L}(X)$. Thus, after applying some finite sequence of foldings of type $2$ to $\mathcal{L}_{sem}(X)$, we obtain an automaton where the word $u$ labels the same path. We prove the case where only a single folding of type $2$ is needed, and the general case follows by induction. Assume the folding identified the father vertices $x_1$ and $x_2$ in $\mathcal{L}_{sem}(X)$ into a single vertex $x_{12}$. If none of the vertices $v_1,\dots,v_l$ traverse the identified vertex $x_{12}$, the path clearly exists in $\mathcal{L}_{sem}(X)$. Otherwise, the path $p$ can be lifted to a path $p'$ in $\mathcal{L}_{sem}(X)$ that goes through either $x_1$ or $x_2$ whenever $p$ goes through $x_{12}$. Consequently, $p'$ is a valid path in $\mathcal{L}_{sem}(X)$ labeled by $u$.
\end{proof}

In particular, by Lemma \ref{l_gen_closure_commutator} and Remark \ref{r_gen_semi_core_paths} if $X$ is a subset of $F_n$, such that the closure of the group generated by $X$ contains the derived subgroup of $F_n$ then every $n$-ary word labels a path on $\mathcal{L}_{sem}(X)$. We conclude the subsection by proving Theorem \ref{intro_t_1}. Recall the Theorem, restated in an equivalent way.
\begin{Theorem}\label{gen_t_main_2}
    Let $X\subseteq F_n$ and let $H$ be the subgroup generated by $X$. Assume that the following hold. \begin{itemize}
        \item[$(1)$] $[F_n,F_n]\subseteq\Cl(H)$.
        \item[$(2)$] $H[F_n,F_n]=F_n$.
        \item[$(3)$] For each $i\in\{0,\dots,n-2\}$, there exist an element $h_i\in H$, and $\alpha_i\in D_{n,i}$, such that $h_i(\alpha_i)=\alpha_i,h_i'(\alpha_i^-)=1,h_i'(\alpha_i^+)=n$.
        \item[$(4)$] There exists an inner $n$-ary word $u$, such that for every two $n$-ary words $v_1,v_2$, if $t(v_1)=t(v_2)$, then $(uv_1)^+=(uv_2)^+$ in $\mathcal{L}_{sem}(X)$.
    \end{itemize}
    Then $H=F_n$.
\end{Theorem}
\begin{proof}
    Note that by Remark \ref{r_gen_semi_core_pair_of_branches}, condition $(4)$ of this theorem implies condition $(4)$ of Theorem \ref{t_Gen_main_theorem}, so we are done.
\end{proof}
Note that if $X$ is finite, then $\mathcal{L}_{sem}(X)$ is finite, and then condition $(4)$ of Theorem \ref{gen_t_main_2} is verifiable.
In the next subsection we provide an algorithm for checking condition $(3)$ of Theorem \ref{gen_t_main_2}.
\subsection{The Tuples Algorithm}
\label{section:tuples_algorithm}

In this section, we present an algorithm verifying condition $(3)$ of Theorem \ref{t_Gen_main_theorem}. We outline the definitions and claims needed to build the algorithm. We omit the supporting claims and proofs, which are very similar to those found in Section 8 of \cite{G1}.

\begin{Problem}\label{pTuples1}
Let $X$ be a finite subset of $F_n$, and let $H = \langle X \rangle$. Assume that $[F_n,F_n]\subseteq \Cl(H)$. Are there $h_0,\dots,h_{n-2}\in H$, and $\alpha_0,\dots,\alpha_{n-2}\in (0,1)$, such that for each $i\in\{0,\dots,n-2\}$ we have $\alpha_i\in D_{n,i}$, $h_i(\alpha_i)=\alpha_i$, $h_i'(\alpha_i^{-})=1$, and $h_i'(\alpha_i^{+})=n$?
\end{Problem}

For each $i\in\{0,\dots,n-2\}$, we define the set:
\begin{equation*}
S_H^i := \{ (a, b) \mid \exists h \in H,\ \exists\alpha \in D_{n,i} \text{ such that } \log_n{h'(\alpha^-)} = a,\ \log_n{h'(\alpha^+)} = b,\ \text{and } h(\alpha) = \alpha \}
\end{equation*}
One can show that $S_H^i$ is a subgroup of $\Z\times\Z$, and that $(0,1)\in S_H^i$ if and only if $(1,0)\in S_H^i$. Thus, Problem \ref{pTuples1} is equivalent to verifying whether $S_H^i=\Z\times\Z$ for all $i\in\{0,\dots,n-2\}$.

\begin{Definition}\label{dTuplesRX}
We define an equivalence relation $\sim_X$ on the set of finite $n$-ary words by $u\sim_X v$ if $u^+=v^+$ in the semi-core $\mathcal{L}_{sem}(X)$. The equivalence class of $u$ is denoted by $[u]_X$. 
\end{Definition}

\begin{Definition}\label{dTuplestuple}
Let $(T_1,T_2)$ be a tree-diagram of $h \in H$, and let $u_1\rightarrow v_1, u_2\rightarrow v_2$ be consecutive pairs of branches (i.e., $u_1$ and $u_2$ are consecutive leaves of $T_1$). There exist unique $n$-ary words $u,v$, unique indices $i,j \in \{0,\dots,n-2\}$, and unique integers $m_1,m_2,k_1,k_2 \in \N$ such that:
$$u_1\equiv ui(n-1)^{m_1},\ u_2\equiv u(i+1)0^{m_2},\ v_1\equiv vj(n-1)^{k_1},\ v_2\equiv v(j+1)0^{k_2}$$
We say that $(m_1-k_1, m_2-k_2, [u]_X\rightarrow[v]_X, i\rightarrow j)$ is the \emph{tuple assigned to this consecutive pair of branches}. The set of all such tuples over all tree-diagrams of elements in $H$ is denoted by $\mathcal{T}_H$. The difference between this definition and the one given in \cite{G1} for $F$ is that in $F$ we do not need the pair $i\rightarrow j$.
\end{Definition}

$\mathcal{T}_H$ naturally forms a groupoid under a partial addition operation and an inverse operation. Namely, given two tuples $t_1=(a,b,[u]_X\rightarrow[v]_X,i\rightarrow j)$ and $t_2=(c,d,[v]_X\rightarrow[w]_X,j\rightarrow k)$, we define $t_1+t_2=(a+c,b+d,[u]_X\rightarrow[w]_X,i\rightarrow k)$ and $t_1^{-1}=(-a,-b,[v]_X\rightarrow[u]_X,j\rightarrow i)$.

Tuples of the form $(a,b,[u]_X\rightarrow[u]_X,i\rightarrow i)$ are called \emph{spherical tuples}. For a fixed word $u$ and index $i$, the set of spherical tuples forms a group, denoted $\mathcal{T}_H([u]_X,i)$. We define a group homomorphism $\Psi: \mathcal{T}_H([u]_X,i) \rightarrow \Z\times\Z$ by $\Psi(a,b,[u]_X\rightarrow[u]_X,i\rightarrow i)=(a,b)$.

The following core propositions allow us to verify Problem \ref{pTuples1} algorithmically.

\begin{Proposition}\label{prop_tuples_main}
Let $i,j\in\{0,\dots,n-2\}$ and let $u$ be a finite $n$-ary word such that $t(u(i+1))=j$. Then $S_H^j = \Psi(\mathcal{T}_H([u]_X,i))$.
\end{Proposition}

\begin{Proposition}\label{prop_tuples_generators}
Let $Y$ be the set of all tuples associated with the reduced tree-diagrams of the generators $X \cup X^{-1}$, together with all trivial spherical tuples $(0,0,[u]_X\rightarrow[u]_X,i\rightarrow i)$ (note that since $X$ is finite, $\mathcal{L}_{sem}(X)$ is also finite so there are finitely many such tuples). 
Then $Y$ generates the groupoid $\mathcal{T}_H$. Furthermore, let $N$ be the number of vertices in $\mathcal{L}_{sem}(X)$, and let $M'(i)$ be the set of spherical tuples in $\mathcal{T}_H([\emptyset]_X,i)$ of word length at most $N$ with respect to the generating set $Y$. Defining $M(i) = \Psi(M'(i))$, $M(i)$ is finite and generates the subgroup $S_H^{i+1}$ (where $i+1$ is taken modulo $n-1$).
\end{Proposition}

Using Proposition \ref{prop_tuples_generators}, we formulate a procedure to decide if $S_H^i = \Z\times\Z$ for all $i$, completing the verification of condition (3) of Theorem \ref{t_Gen_main_theorem}.

\subsubsection*{The Algorithm}

\noindent\textbf{Input:} A finite generating set $X$ for $H\leq F_n$ satisfying $[F_n,F_n]\subseteq\Cl(H)$.

\noindent\textbf{Output:} \textbf{True} if $S_H^i=\Z\times\Z$ for all $i\in\{0,\dots,n-2\}$; otherwise \textbf{False}.

\begin{enumerate}
    \item Compute the finite automaton $\mathcal{L}_{sem}(X)$ and let $N$ be its number of vertices.
    \item Initialize $Y$ as the set of all tuples associated with the reduced tree-diagrams of elements of $X^{\pm 1}$.
    \item Add all trivial spherical tuples $(0,0,[u]_X\rightarrow [u]_X,i\rightarrow i)$ to $Y$, for every equivalence class $[u]_X$ represented by a vertex of $\mathcal{L}_{sem}(X)$ and every $i\in\{0,\dots,n-2\}$.
    \item Compute $M'(i)$, the set of all spherical tuples in $\mathcal{T}_H([\emptyset]_X,i)$ generated by composing at most $N$ elements from $Y$.
    \item Compute $M(i)=\{\Psi(t)\mid t\in M'(i)\}$.
    \item If, for each $i\in\{0,\dots,n-2\}$, the set $M(i)$ generates $\Z\times\Z$, return \textbf{True}; otherwise return \textbf{False}.
\end{enumerate}

%% file: Oncoreautomata.tex
\section{On core automata}
\label{section:core_automata}
Our goal in this section is to develop several tools regarding automata, which we will use to prove Theorem \ref{intro_t_2}. We start by introducing a slight generalization of the rooted $n$-ary tree automata from Subsection \ref{sec:preliminaries:core_and_closure}. Specifically, we define a \emph{rooted tree automaton} similarly to a rooted $n$-ary tree automaton, but without requiring that each father vertex has the same number of children.
\begin{Definition}
    Let $\mathcal{A}_r$ be a rooted automaton. We say that $\mathcal{A}_r$ is a \textit{rooted tree-automaton} if the following conditions hold.
    \begin{itemize}
        \item[$(1)$]Every vertex $a$ in $\mathcal{A}_r$ has a finite number of outgoing edges denoted by $n(a)\in\N$. If $n(a)>0$ we call $a$ a \emph{father}. Otherwise, we call $a$ a \emph{leaf}.
        \item[$(2)$]The outgoing edges of a father vertex $a$ are labeled by $i=0,\dots,n(a)-1$. The vertex at the end of the edge labeled by $i$ is called the $i$-th child of $a$.
        \item[$(3)$]If $a_1$ and $a_2$ are distinct fathers in $\mathcal{A}_r$ then either $n(a_1)\neq n(a_2)$, or $n(a_1)=n(a_2)$ and there exists $0\le i\le n(a_1)-1$ such that the $i$-th child of $a_1$ is not equal to the $i$-th child of $a_2$.
        \item[$(4)$]For every vertex $a$ in $\mathcal{A}_r$ there is a path starting at $r$ and ending in $a$.   
    \end{itemize}
We say that $\mathcal{A}_r$ is a \emph{rooted semi tree-automaton} if conditions $(1),(2)$ and $(4)$ from above hold.

Like in rooted $n$-ary tree-automata, we only consider rooted paths. Given a path $p$ in a rooted semi tree-automaton, we denote the end vertex of $p$ by $p^+$.
\end{Definition}
\begin{Remark}
    Let $\mathcal{A}_r$ be a rooted semi tree-automaton. Let $\mathcal{A}_r'$ be the automaton obtained from $\mathcal{A}_r$ by applying all possible foldings of type $2$ (i.e., identifying father vertices with the same children). Then $\mathcal{A}_r'$ is a rooted tree-automaton.
\end{Remark}
In our setting, a \emph{tree} is a connected, acyclic graph. A \emph{rooted tree} $\mathcal{T}$ is a directed tree with a unique distinguished vertex $r$, called the \emph{root}, such that there exists a unique directed path from $r$ to every other vertex in $\mathcal{T}$ (equivalently, $r$ has in-degree $0$ and all other vertices have in-degree $1$). For any vertex $v \in \mathcal{T}$ with out-degree $n(v) > 0$, its outgoing edges are naturally labeled by $0, \dots, n(v)-1$. 

Given a rooted tree $\mathcal{T}$, a subgraph $\mathcal{T}'$ is called a \emph{full sub-tree} of $\mathcal{T}$ if it is a rooted sub-tree sharing the same root $r$, and it satisfies the following condition: for any vertex $v \in \mathcal{T}'$ that is not a leaf in $\mathcal{T}'$ (i.e., $v$ has at least one outgoing edge in $\mathcal{T}'$), all of its children and outgoing edges in $\mathcal{T}$ are also contained in $\mathcal{T}'$.

Let $\mathcal{A}_r$ be a rooted semi tree-automaton. We follow \cite{G1,G2} to construct rooted trees $\mathcal{T}_{\mathcal{A}_r},\mathcal{T}_{\mathcal{A}_r}^{min}$ associated with $\mathcal{A}_r$.
\begin{Definition}
    Let $\mathcal{A}_r$ be a semi rooted tree-automaton. The \emph{path tree associated with $\mathcal{A}_r$}, denoted $\mathcal{T}_{\mathcal{A}_r}$, is defined as follows. The vertices of $\mathcal{T}_{\mathcal{A}_r}$ correspond to the paths on $\mathcal{A}_r$. For every pair of vertices $p,q$ of $\mathcal{T}_{\mathcal{A}_r}$, there is an edge from $p$ to $q$ labeled by $i$ if $q$ is a path on $\mathcal{A}_r$ extending $p$ by an edge labeled by $i$. Each vertex $p$ in $\mathcal{T}_{\mathcal{A}_r}$ is labeled by $p^+$ (i.e., by the endvertex of the path $p$ in $\mathcal{A}_r$). We use the same notation for paths on $\mathcal{A}_r$ and vertices on $\mathcal{T}_{\mathcal{A}_r}$. Specifically, we denote the label of a vertex $p$ in $\mathcal{T}_{\mathcal{A}_r}$ by $p^+$.

    A full sub-tree $\mathcal{T}'$ of $\mathcal{T}_{\mathcal{A}_r}$ is said to be a \textit{minimal tree associated with $\mathcal{A}_r$} if for every father vertex $a$ of $\mathcal{A}_r$ there is exactly one father vertex $p$ in $\mathcal{T}'$ labeled by $a$. We usually denote such a tree by $\mathcal{T}_{\mathcal{A}_r}^{min}$. Note that while a minimal tree associated with $\mathcal{A}_r$ is not unique, it ``holds all the information about $\mathcal{A}_r$'', in the sense that $\mathcal{A}_r$ can be reconstructed from any minimal tree associated with it.

    See Figure \ref{tree-automaton example} for an example of a rooted tree-automaton and trees associated with it.
\end{Definition}

\begin{figure}
\centering


\begin{tikzpicture}[->,>=stealth',shorten >=1pt,auto,node distance=2.5cm,
                    semithick]
  \tikzstyle{every state}=[fill=white,draw=black,text=black,minimum size=10mm]

  \node[state] (r) {$r$};
  \node[state] (a1) [below left of=r] {$a_1$};
  \node[state] (a2) [below right of=r] {$a_2$};
  \node[state] (a3) [left of=a1] {$a_3$};

  \path (r) edge node[above] {0} (a1)
        (r) edge node[above] {1} (a2)
        (a1) edge [loop below] node {0} (a1)
        (a1) edge [bend left] node {1} (a2)
        (a1) edge node {2} (a3)
        (a2) edge [loop right] node {0} (a2)
        (a2) edge [bend left] node {1} (a1);
\end{tikzpicture}
\caption*{A rooted tree-automaton $\mathcal{A}_r$}

\begin{tikzpicture}[grow'=down, level distance=1.4cm,
  level 1/.style={sibling distance=6cm},
  level 2/.style={sibling distance=3cm},
  level 3/.style={sibling distance=1.2cm},
  every node/.style={circle,draw,minimum size=5mm,inner sep=0pt}]

\node {$r$}
    child {node {$a_2$}
        child {node {$a_1$}}
        child {node {$a_2$}}}
    child {node {$a_1$}
        child {node {$a_3$}}
        child {node {$a_2$}}
        child {node {$a_1$}}}
;
\end{tikzpicture}
\caption*{A minimal tree $\mathcal{T}_{\mathcal{A}_r}^{min}$}

\medskip
\begin{tikzpicture}[grow'=down, level distance=1.4cm,
  level 1/.style={sibling distance=6cm},
  level 2/.style={sibling distance=3cm},
  level 3/.style={sibling distance=1.2cm},
  level 4/.style={sibling distance=0.5cm},
  every node/.style={circle,draw,minimum size=5mm,inner sep=0pt}]

\node {$r$}
    child {node {$a_2$}
        child {node {$a_1$}
            child[dotted] {node {$a_3$}}
            child[dotted] {node {$a_2$}
                child[dotted]
                child[dotted]}
            child[dotted] {node {$a_1$}
                child[dotted]
                child[dotted]
                child[dotted]}}
        child {node {$a_2$}
            child[dotted] {node {$a_1$}
                child[dotted]
                child[dotted]
                child[dotted]}
            child[dotted] {node {$a_2$}
                    child[dotted]
                child[dotted]}}}
    child {node {$a_1$}
        child {node {$a_3$}}
        child {node {$a_2$}
            child[dotted] {node {$a_1$}
                child[dotted]
                child[dotted]
                child[dotted]}
            child[dotted] {node {$a_2$}
                child[dotted]
                child[dotted]}}
        child {node {$a_1$}
            child[dotted] {node {$a_3$}}
            child[dotted] {node {$a_2$}
                child[dotted]
                child[dotted]}
            child[dotted] {node {$a_1$}
                child[dotted]
                child[dotted]
                child[dotted]}}}
;
\end{tikzpicture}
\caption*{The path tree $\mathcal{T}_{\mathcal{A}_r}$}

\medskip
\medskip
\caption{An example of a rooted tree-automaton $\mathcal{A}_r$, the path tree $\mathcal{T}_{\mathcal{A}_r}$ and a minimal rooted tree $\mathcal{T}_{\mathcal{A}_r}^{min}$ associated with $\mathcal{A}_r$. Each vertex in those trees, corresponding to a path $p$ in the automaton, is labeled by the end vertex of $p$. Every caret in each of the trees, rooted in some vertex $p$, corresponds to the outgoing edges of the vertex $p^+$ in the automaton.}\label{tree-automaton example}
\end{figure}

Each rooted tree-automaton gives rise to a semigroup presentation. We call such a semigroup presentation a \emph{tree semigroup presentation}. Alternatively, such presentations can be defined axiomatically:
\begin{Definition}
    Given a semigroup presentation $\mathcal{P}$, together with a generator $r$, we say that $\mathcal{P}_r$ is a \textit{tree semigroup presentation} if the following hold. \begin{itemize}
        \item[$(1)$]All the relations of $\mathcal{P}$ are of the form $a=a_0\cdots a_k$, where $a$ is a generator of $\mathcal{P}$ and $k\in\N$ ($k$ may vary among different relations).
        \item[$(2)$]Each generator $a$ of $\mathcal{P}$ is the left hand side of at most one relation of the form $a=a_0\cdots a_k$.
        \item[$(3)$]If $a,b$ are distinct generators of $\mathcal{P}$, and $a=a_0\cdots a_k,b=b_0\cdots b_l$ are relations then $k\neq l$, or $k=l$ and $a_i\neq b_i$ for some $i\in\{0,\dots,k\}$.
        \item[$(4)$]For every generator $a$ of $\mathcal{P}$, there exists a word $a_0\cdots a_m$, such that $a_0\cdots a_m$ is a positive derivation of $r$ (i.e., the word $a_0\cdots a_m$ can be obtained from $r$ by applying a finite number of relations of the form $b\rightarrow b_0\cdots b_k$, where in each application we switch the left hand side by the right hand side), and $a_i=a$ for some $i\in\{0,\dots,m\}$.
    \end{itemize}
    
    We say that $\mathcal{P}_r$ is a \emph{semi tree semigroup presentation} if it satisfies conditions $(1),(2),(4)$.
\end{Definition}
\begin{Definition}
    Let $\mathcal{A}_r$ be a rooted semi tree-automaton. We define the \emph{semi tree semigroup presentation associated with $\mathcal{A}_r$}, denoted by $\mathcal{P}(\mathcal{A}_r)$, as follows. The set of generators of $\mathcal{P}({\mathcal{A}_r})$ is the set of vertices of $\mathcal{A}_r$. The relations of $\mathcal{P}({\mathcal{A}_r})$ are all the relations of the form $a=a_0 \cdots a_{n(a)-1}$, where $a,a_0,\dots,a_{n(a)-1}$ are vertices of $\mathcal{A}_r$, and for each $i\in\{0,\dots,n(a)-1\}$ $a_i$ is the $i$-th child of $a$. Finally, $r$ is the specified generator of $\mathcal{P}(\mathcal{A}_r)$.
    
    Let $\mathcal{P}_r$ be a semi tree semigroup presentation. We define the \emph{rooted semi tree-automaton associated with $\mathcal{P}_r$}, denoted by $\mathcal{A}(\mathcal{P}_r)$, as follows. The vertices of $\mathcal{A}(\mathcal{P}_r)$ are the generators of $\mathcal{P}_r$, and the root of $\mathcal{A}(\mathcal{P}_r)$ is $r$. If $a=a_0\cdots a_k$ is a relation of $\mathcal{P}_r$, the children of $a$ in $\mathcal{A}_r$ are $a_0,\dots,a_k$.
\end{Definition}
These constructions are inverses of each other. More precisely, if $\mathcal{A}_r$ is a rooted semi tree-automaton then $\mathcal{A}(\mathcal{P}(\mathcal{A}_r))=\mathcal{A}_r$, and if $\mathcal{P}_r$ is a semi tree semigroup presentation then $\mathcal{P}(\mathcal{A}(\mathcal{P}_r))=\mathcal{P}_r$.

Notice that if $\mathcal{P}_r$ is a tree semigroup presentation, then $\mathcal{A}(\mathcal{P}_r)$ is a rooted tree-automaton. Also, if $\mathcal{A}_r$ is a rooted tree-automaton, then $\mathcal{P}(\mathcal{A}_r)$ is a tree semigroup presentation.
\begin{Definition}
    Let $\mathcal{P}_r$ be a semi tree semigroup presentation. Let $p$ be a path on $\mathcal{A}(\mathcal{P}_r)$. We denote by $\mathcal{T}_p$ the minimal full sub-tree of $\mathcal{T}_{\mathcal{A}(\mathcal{P}_r)}$ in which $p$ labels a path. Note that the path $p$ is a branch of $\mathcal{T}_p$ (i.e., it ends on a leaf in $\mathcal{T}_p$). Denote the branches of $\mathcal{T}_p$ by $p_1,\dots,p_m$ and let $i\in\{1,\dots,m\}$ such that $p=p_i$. We define the \emph{left word of $p$} as $p^{left}=p_i^{left}:=p_1^+\cdots p_{i-1}^+$. In other words, $p^{left}$ is the product of all the labels of branches that are located to the left of $p$ in $\mathcal{T}_p$. Similarly, we define the \emph{right word of $p$} as $p^{right}=p_i^{right}:=p_{i+1}^+\cdots p_m^+$. We follow the convention that an empty product is the empty word (which is not equal to any other word in the semigroup).
\end{Definition}
\begin{Definition}\label{d_core_semigroup_splitting}
    Let $\mathcal{P}_r$ be a semi tree presentation. Let $a=a_0\cdots a_k$ be a relation of $\mathcal{P}_r$ and let $0\le s\le f\le k$. We say that a semigroup presentation $\mathcal{P}_r'$ is a \textit{basic splitting of $\mathcal{P}_r$} if it satisfies the following. The set of generators of $\mathcal{P}_r'$ is the set of generators of $\mathcal{P}_r$ together with a new generator $x$. The set of relations of $\mathcal{P}_r'$ is obtained from the set of relations of $\mathcal{P}_r$ by removing the relation $a=a_0\cdots a_k$ and adding the relations $a=a_0\cdots a_{s-1}\cdot x\cdot a_{f+1}\cdots a_k$, and $x=a_s \cdots a_f$. Note that $\mathcal{P}_r'$ is a semi tree semigroup presentation.

    We say that a semi tree semigroup presentation $\mathcal{Q}_r$ is a \textit{splitting of $\mathcal{P}_r$} if there is a positive integer $l$ and semi tree presentations $\mathcal{P}_r^{0},\dots,\mathcal{P}_r^l$ such that $\mathcal{P}_r^{0}=\mathcal{P}_r$, $\mathcal{P}_r^{l}=\mathcal{Q}_r$, and for each $i\in\{0,\dots,l-1\}$, the semi tree semigroup presentation $\mathcal{P}_r^{i+1}$ is a basic splitting of $\mathcal{P}_r^{i}$. In other words, $\mathcal{Q}_r$ is a splitting of $\mathcal{P}_r$ if one can obtain $\mathcal{Q}_r$ by constructing a sequence of semi tree semigroup presentations, starting at $\mathcal{P}_r$, such that each semi tree semigroup presentation is a basic splitting of its predecessor. Note that if $\mathcal{Q}_r$ is a splitting of $\mathcal{P}_r$ then they represent the same semigroup.
\end{Definition}
    Note that if $\mathcal{Q}_r$ is a basic splitting of $\mathcal{P}_r$ then, in the notation of Definition \ref{d_core_semigroup_splitting} there is a natural bijection between paths on $\mathcal{A}(\mathcal{Q}_r)$ that do not end in the vertex $x$ and paths on $\mathcal{A}(\mathcal{P}_r$). A way to see this is by considering a path as the sequence of vertices it visits. Given a path $q$ in $\mathcal{A}(\mathcal{Q}_r)$, its image under the bijection is the path in $\mathcal{A}(\mathcal{P}_r)$ obtained by removing all occurrences of the vertex $x$ from $q$.
    
    The validity of the following claim is evident.
    
\begin{Remark}\label{r_core_path_pull_left_right_words}
Let $\mathcal{S}$ be a semigroup. Let $\mathcal{P}_r$ be a semi tree semigroup presentation of $\mathcal{S}$ (we assume that $\mathcal{S}$ admits such a semigroup presentation), and let $\mathcal{Q}_r$ be a basic splitting of $\mathcal{P}_r$. Using the same notations as in Definition \ref{d_core_semigroup_splitting}, let $q$ be a path on $\mathcal{Q}_r$ such that $q^+\neq x$ and let $p$ be the corresponding path on $\mathcal{P}_r$. Then, $p^{left}$ (resp. $p^{right}$) and $q^{left}$ (resp. $q^{right}$) are equal in $\mathcal{S}$.
\end{Remark}
\begin{Lemma}\label{l_core_splitting_does_not_change_equality}
    Let $\mathcal{S}$ be a semigroup. Let $\mathcal{P}_r$ be a semi tree semigroup presentation of $\mathcal{S}$. Assume that for every pair of paths $p,p'$ in $\mathcal{A}(\mathcal{P}_r)$, if $p^+=(p')^+$, then $p^{left}=(p')^{left}$ and $p^{right}=(p')^{right}$ in $\mathcal{S}$. Let $\mathcal{Q}_r$ be a splitting of $\mathcal{P}_r$. Then for each pair of paths $q,q'$ in $\mathcal{A}(\mathcal{Q}_r)$, if $q^+=(q')^+$, then $q^{left}=(q')^{left}$ and $q^{right}=(q')^{right}$ in $\mathcal{S}$.
\end{Lemma}
\begin{proof}
    We can reduce the proof to the case that $\mathcal{Q}_r$ is a basic splitting of $\mathcal{P}_r$ and deduce the general case by induction. Since $\mathcal{Q}_r$ is a basic splitting of $\mathcal{P}_r$ there is a relation $a=a_0\dots a_k$ in $\mathcal{P}_r$ such that the set of generators of $\mathcal{Q}_r$ is the same as the set of generators of $\mathcal{P}_r$ together with a new generator $x$. The set of relations of $\mathcal{Q}_r$ is obtained from the set of relations of $\mathcal{P}_r$, by removing the relation $a=a_0\cdots a_k$ and adding the relations $a=a_0 \cdots a_{s-1}\cdot x\cdot a_{f+1}\cdots a_k$, and $x=a_s\cdots a_f$ for some integers $0\le s\le f\le k$. Now, let $q,q'$ be paths on $\mathcal{A}(\mathcal{Q}_r)$ such that $q^+=(q')^+$. We first assume that $q^+\neq x$. Let $p,p'$ be the corresponding paths in $\mathcal{A}(\mathcal{P}_r$). By Remark \ref{r_core_path_pull_left_right_words}, $p^{left}=q^{left}$, $(p')^{left}=(q')^{left}$, $p^{right}=q^{right}$, and $(p')^{right}=(q')^{right}$ (where equality is as elements of the semigroup $\mathcal{S}$). Since $q^+=(q')^+$ (on $\mathcal{A}(\mathcal{Q}_r)$), it holds that $p^+=q^+=(q')^+=(p')^+$. So by the assumption of the lemma, $p^{left}=(p')^{left}$ and $p^{right}=(p')^{right}$ (as elements of $\mathcal{S}$). So $q^{left}=p^{left}=(p')^{left}=(q')^{left}$ and $q^{right}=p^{right}=(p')^{right}=(q')^{right}$ (as elements of $\mathcal{S}$).
    
    We are left with the case that $q^+=(q')^+=x$. Let $\tilde{q},\tilde{q}'$ be the paths on $\mathcal{A}(\mathcal{Q}_r)$ such that $lab(\tilde{q})=lab(q)0,lab(\tilde{q}')=lab(q')0$. Note that $\tilde{q}^+=(\tilde{q}')^+=a_s\neq x$. As in the previous case, $(\tilde{q})^{left}=(\tilde{q}')^{left}$. Note that $(\tilde{q})^{left}=q^{left}$ and $(\tilde{q}')^{left}=(q')^{left}$. So $q^{left}=(q')^{left}$. A similar argument, using the words $q(f-s),q'(f-s)$ shows that $q^{right}=(q')^{right}$ (recall that $x=a_s\cdots a_f$ is a relation we added to $\mathcal{Q}_r$). 
\end{proof}
The following claim can be easily verified, and we use it to prove the lemma that follows.
\begin{Remark}
    Let $\mathcal{A}_r$ be a rooted semi tree-automaton. Let $\mathcal{A}_r'$ be a rooted semi tree-automaton obtained from $\mathcal{A}_r$ by foldings of type $2$. Then $\mathcal{P}(\mathcal{A}_r)$ and $\mathcal{P}(\mathcal{A}_r')$ represent the same semigroup.
\end{Remark}
\begin{Lemma}\label{l_core_foldings_does_not_change_equality}
    Let $\mathcal{A}_r$ be a rooted semi tree-automaton. Let $\mathcal{A}_r'$ be a rooted semi tree-automaton, which is obtained from $\mathcal{A}_r$ by foldings of type $2$. Let $\mathcal{S}$ be the semigroup represented by $\mathcal{P}(\mathcal{A}_r)$ and by $\mathcal{P}(\mathcal{A}_r')$. Assume that for every pair of paths $p,p'$ in $\mathcal{A}_r$, such that $p^+=(p')^+$, we have $p^{left}=(p')^{left}$ and $p^{right}=(p')^{right}$ in $\mathcal{S}$. Then for each pair of paths $q,q'$ in $\mathcal{A}_r'$, such that $q^+=(q')^+$, we have $q^{left}=(q')^{left}$ and $q^{right}=(q')^{right}$ in $\mathcal{S}$.
\end{Lemma}
\begin{proof}
It is enough to show the case that $\mathcal{A}'_r$ is obtained from $\mathcal{A}_r$ by a single folding of type $2$, and the rest follows by induction. Denote by $x$ and $y$ the vertices of $\mathcal{A}_r$ identified by the folding. Denote the vertex in $\mathcal{A}_r'$ they are identify to by $z$. Note that since $x$ and $y$ are identified by a folding of type $2$, they must have the same number of outgoing edges, which we denote by $n$.

Let $q,q'$ be paths on $\mathcal{A}_r'$ such that $q^+=(q')^+$. Note that a word labels a path on $\mathcal{A}_r$ if and only if it labels a path on $\mathcal{A}_r'$. So, the paths $q,q'$ on $\mathcal{A}_r'$ can be lifted to paths $p,p'$ on $\mathcal{A}_r$, which are labeled $lab(q)$ and $lab(q')$, respectively. Note that either $p^+=(p')^+$, or $\{p^+,(p')^+\}=\{x,y\}$. In any case, $(p0)^+=(p'0)^+$ and $(p(n-1))^+=(p'(n-1))^+$. Also note that $q^{left}$ and $p^{left}=(p0)^{left}$ can be expressed in the following way. $q^{left}=b_1\cdots b_m$, $p^{left}=a_1\cdots a_m$, where for each $i\in\{1,\dots,m\}$, either $a_i=b_i$, or $a_i\in\{x,y\}$ and $b_i=z$. Since $x$ and $y$ are identified by a folding, they are equal in the semigroup $\mathcal{S}$, so $q^{left}=(p0)^{left}$ in $\mathcal{S}$. Similarly, $(q')^{left}=(p0')^{left}$, and $q^{right}=(p(n-1))^{right}$ and $(q')^{right}=(p'(n-1))^{right}$ in the semigroup $\mathcal{S}$. Since $(p0)^+=(p'0)^+$, it follows by the assumption that $(p0)^{left}=(p'0)^{left}$, so $q^{left}=(q')^{left}$. Similarly, $(p(n-1))^+=(p'(n-1))^+$, so by the assumption $(p(n-1))^{right}=(p'(n-1))^{right}$, so $q^{right}=(q')^{right}$.
\end{proof}
We state two definitions and a lemma that were given in \cite{G2} for $F_2$. The proof of the lemma readily adapts to the case of $F_n$.
\begin{Definition}
    Let $\mathcal{A}_r$ be an $n$-ary rooted tree automaton. An \textit{extension} of $\mathcal{A}_r$ is an automaton obtained by attaching to each leaf $l$ of $\mathcal{A}_r$ with a finite (possibly empty) or infinite $n$-ary tree, where each $n$-ary tree is viewed as an $n$-ary tree automaton in a natural way.

    If $\mathcal{A}_r'$ is an extension of $\mathcal{A}_r$, we say that $\mathcal{A}_r$ is a \textit{reduction} of $\mathcal{A}_r'$. If $\mathcal{A}_r$ is not the extension of any rooted $n$-ary tree automaton other than of itself, we say it is \textit{reduced}.
\end{Definition}
\begin{Definition}
    Let $\mathcal{A}_r$ be an $n$-ary rooted tree automaton. We say that $\mathcal{A}_r$ is a \textit{core automaton} if there exists a subgroup $H$ of $F_n$ such that $\mathcal{A}_r$ is isomorphic to the core $\mathcal{L}(H)$ (isomorphism between rooted tree automata is defined in a natural way).
\end{Definition}
\begin{Lemma}\label{l_core_core_automaton_char}
    Let $\mathcal{A}_r$ be an $n$-ary rooted tree automaton. Let $\mathcal{T}_{\mathcal{A}_r}^{min}$ be a minimal tree associated with $\mathcal{A}_r$. Then $\mathcal{A}_r$ is a core automaton if and only if the following conditions hold.
    \begin{itemize}
        \item[$(1)$] $\mathcal{A}_r$ is reduced.
        \item[$(2)$] For every two branches $u,v$ of $\mathcal{T}_{\mathcal{A}_r}^{min}$, such that $u^+=v^+$, it holds that $u^{left}=v^{left}$ and $u^{right}=v^{right}$ in the associated semigroup.
    \end{itemize}
\end{Lemma}
As a corollary of Lemmas \ref{l_core_splitting_does_not_change_equality}, \ref{l_core_foldings_does_not_change_equality} and \ref{l_core_core_automaton_char} we get the following.
\begin{Corollary}\label{c_core_splitting_does_not_change_core}
    Let $m,n\ge 2$ be natural numbers, and let $\mathcal{A}_r$ be an $n$-ary core automaton. Let $\mathcal{A}_r'$ be an $m$-ary rooted semi tree-automaton such that $\mathcal{P}(\mathcal{A}_r')$ is a splitting of $\mathcal{P}(\mathcal{A}_r)$. Assume that $\mathcal{A}_r''$ is a reduced $m$-ary rooted tree-automaton obtained from $\mathcal{A}_r'$ by foldings of type $2$. Then $\mathcal{A}_r''$ is an $m$-ary core automaton.
\end{Corollary}
Let $\mathcal{A}_r$ be a rooted $n$-ary semi tree automaton (i.e., a rooted semi tree automaton where each father vertex has exactly $n$ children). Let $f\in F_n$. We say that $f$ is \emph{accepted} by $\mathcal{A}_r$ if there exists a tree diagram representing $f$, which is accepted by $\mathcal{A}_r$ (where acceptance of a tree-diagram by a rooted $n$-ary semi tree automaton is defined to be the same as acceptance of a tree-diagram by a rooted $n$-ary tree automaton). The \emph{diagram group accepted by $\mathcal{A}_r$} is defined as the group of all elements in $F_n$ accepted by $\mathcal{A}_r$, and is denoted by $\mathcal{DG}(\mathcal{A}_r)$. Note that if $\mathcal{A}_r$ is a rooted $n$-ary tree automaton, these definitions coincide with Definition \ref{d_pre_group_accepted_by_automaton}.
\begin{Lemma}\label{l_core_foldings_type_2_diagram_groups_isomorphism}
    Let $\mathcal{A}_r$ be a rooted $n$-ary semi tree automaton. Let $\mathcal{A}_r'$ be a rooted $n$-ary semi tree automaton, obtained from $\mathcal{A}_r$ by a finite number of foldings of type $2$. Then $\mathcal{DG}(\mathcal{A}_r)=\mathcal{DG}(\mathcal{A}_r')$.
\end{Lemma}
\begin{proof}
    It is enough to show the case where $\mathcal{A}_r'$ is obtained from $\mathcal{A}_r$ by a single folding of type $2$, and the rest follows by induction. Denote the vertices of $\mathcal{A}_r$ identified by the folding by $x$ and $y$ and denote the vertex they were identified into by $z$. Note that any tree-diagram which is accepted by $\mathcal{A}_r$ is also accepted by $\mathcal{A}_r'$, so $\mathcal{DG}(\mathcal{A}_r)\subseteq \mathcal{DG}(\mathcal{A}_r')$. In the other direction, let $f\in\mathcal{DG}(\mathcal{A}_r')$. Then there is a tree-diagram $(T_1,T_2)$ representing $f$ that is accepted by $\mathcal{A}_r'$. Note that one can construct a tree-diagram $(T_1',T_2')$ which is equivalent to $(T_1,T_2)$ and is accepted by $\mathcal{A}_r$ by adding a caret under any pair of branches $u\rightarrow v$ of $(T_1,T_2)$ such that $u^+=v^+=z$ in $\mathcal{A}_r'$. So $f\in\mathcal{DG}(\mathcal{A}_r)$.
\end{proof}

%% file: Onmaximalsubgroups.tex
\section{Maximal subgroups}
\label{section:maximal_subgroups}
In this section, we answer a variant of a question posed by Savchuk \cite{Sav1,Sav2}, which was also asked by Aiello and Nagnibeda \cite{AN}, regarding the existence of maximal subgroups of $F_n$.

In \cite{Sav1}, Savchuk initiated the study of maximal subgroups of $F_2$ by proving that for any $\alpha\in(0,1)$, the stabilizer of $\alpha$ in $F_2$ is a maximal subgroup of infinite index. He asked whether there are other maximal subgroups of infinite index in $F_2$. Golan and Sapir answered this question affirmatively in \cite{GS1} by providing an explicit example of a maximal subgroup of infinite index in $F_2$ that does not fix any point. They also introduced a general method for constructing maximal subgroups of $F_2$ via core automata. Aiello and Nagnibeda \cite{AN2} found three additional examples of maximal subgroups of infinite index in $F_2$ that do not fix any point in $(0,1)$. Later, in \cite{G2}, Golan found a countable family of non-isomorphic maximal subgroups of infinite index in $F_2$ that do not fix any point in $(0,1)$. Aiello and Nagnibeda extended Savchuk's question to $F_n$ for a general $n\ge 2$, and provided an example of a maximal subgroup of infinite index in $F_3$ that does not fix any point in the open unit interval (\cite{AN}). The question for a general $n$ remained open.

We show that for every $n\ge 2$, $F_n$ has a maximal subgroup of infinite index that does not fix any point in $(0,1)$. Following the approach introduced by Golan and Sapir, we analyze subgroups accepted by core automata. We first note the following result, which is an immediate consequence of Theorem \ref{gen_t_main_2}.

\begin{Corollary}\label{C_maximal_1}
    Let $H$ be a strict subgroup of $F_n$ satisfying the following properties:
    \begin{enumerate}
        \item[$(a)$] $\pi_{ab}(H)=\Z^n$.
        \item[$(b)$] For each $i\in\{0,\dots,n-2\}$, there exist an element $h_i\in H$, and $\alpha_i\in D_{n,i}$, such that $h_i(\alpha_i)=\alpha_i,h_i'(\alpha_i^-)=1,h_i'(\alpha_i^+)=n$.
        \item[$(c)$] For any $g_1\in F_n\setminus H$, there exists an inner $n$-ary word $w$ such that for any two $n$-ary words $w_1,w_2$ such that $t(w_1)=t(w_2)$, it holds that $(ww_1)^+=(ww_2)^+$ in $\mathcal{L}_{sem}(H\cup\{g_1\})$.
        \item[$(d)$] For any $g_1\in F_n\setminus H$, $\mathcal{L}(H\cup\{g_1\})$ accepts $[F_n,F_n]$.
    \end{enumerate}
    Then $H$ is a maximal subgroup of $F_n$.
\end{Corollary}
\begin{proof}
    Let $H$ be a strict subgroup of $F_n$ such that conditions $(a)-(d)$ hold. Let $g\in F_n\setminus H$. Consider the group $H_1:=\langle H,g\rangle$. Conditions $(a)-(d)$ imply that conditions $(1)-(4)$ of Theorem \ref{gen_t_main_2} hold for the generating set $H\cup\{g\}$ of $H_1$, so by the theorem, $H_1=F_n$. Since $H$ is a strict subgroup, it is a maximal subgroup of $F_n$.
\end{proof}
\begin{Remark}
    The rooted $n$-ary tree-automaton defined by the following tree semigroup presentation is the core of $F_n$. The generators are\\$r,f,g,a_0,\dots,a_{n-2}$, and the relations are as follows.
    \begin{itemize}
        \item $r= f\cdot a_1\cdots a_{n-2}\cdot g$.
        \item $f= f\cdot a_1\cdots a_{n-2}\cdot a_0$.
        \item $g= a_0\cdots a_{n-2}\cdot g$.
        \item For each $i\in\{0,\dots,n-2\}$, $a_i=a_i\cdot a_{i+1}\cdots a_i$ (where addition is taken modulo $n-1$).
    \end{itemize}
    See Figure \ref{core of $F_n$} for a minimal tree associated with the core of $F_n$. Indeed, note that every $n$-ary word labels a path on this rooted $n$-ary tree-automaton, every word of the form $0^k$ where $k\in\N\setminus\{0\}$ has the same end vertex, every word of the form $(n-1)^k$ where $k\in\N\setminus\{0\}$ has the same end vertex and there are exactly $n-1$ inner vertices. Hence, intuitively, "everything that can be identified is identified".
\end{Remark}
\begin{figure}[H]
\centering

\begin{tikzpicture}[grow'=down, level distance=1.4cm,
  level 1/.style={sibling distance=2cm},
  level 2/.style={sibling distance=1.5cm}]

\node {$r$}
    child {node {$g$}
        child{node {$g$}}
        child{node {$a_{n-2}$}}
        child {node {$\cdots$} edge from parent[draw=none]}
        child{node {$a_0$}}}
    child {node {$a_{n-2}$}}
    child {node {$\cdots$} edge from parent[draw=none]}
    child {node {$a_{1}$}}
    child {node {$f$}
        child{node {$a_0$}}
        child{node {$a_{n-2}$}}
        child {node {$\cdots$} edge from parent[draw=none]}
        child{node {$a_1$}}
        child{node {$f$}}}
;
\end{tikzpicture}
\medskip
\begin{tikzpicture}[grow'=down, level distance=1.4cm,
  level 1/.style={sibling distance=2cm}]
\node {$a_i$}
    child {node {$a_i$}}
    child {node {$\cdots$} edge from parent[draw=none]}
    child {node {$a_{i+1}$}}
    child {node {$a_i$}};
\end{tikzpicture}
\caption{An image describing a minimal tree associated with the core of $F_n$. The children of a vertex labeled by $a_i$ are written separately of the tree. To see a minimal tree, one should think about those carets as a part of the tree.}\label{core of $F_n$}
\end{figure}
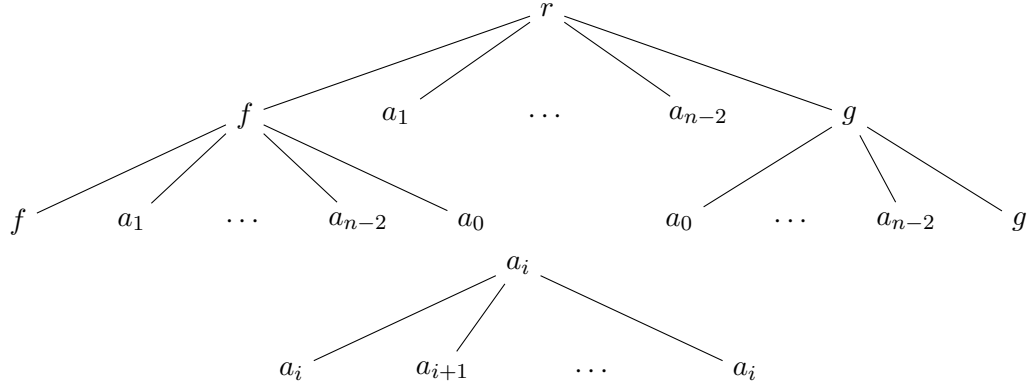
We now prove Theorem \ref{intro_t_2}. Recall the theorem.
\begin{Theorem}
    $F_n$ has a maximal subgroup of infinite index that does not fix any point in the open unit interval, and is isomorphic to $F_{2n-1}$.
\end{Theorem}
\begin{proof}
    Let $\mathcal{L}(F_{2n-1})$ be the core of $F_{2n-1}$. Let us look at the semigroup presentation $\mathcal{P}_r$ associated with the core of $F_{2n-1}$. It is generated by the elements $r,f,g,a_0,\dots,a_{2n-3}$, and has the following relations.
    \begin{itemize}
        \item $r= f\cdot a_1\cdots a_{2n-3}\cdot g$.
        \item $f= f\cdot a_1\cdots a_{2n-3}\cdot a_0$.
        \item $g= a_0\cdots a_{2n-3}\cdot g$.
        \item For each $i\in\{0,\dots,2n-3\}$, $a_i=a_i\cdot a_{i+1}\cdots a_i$ (where addition is taken modulo $2n-2$).
    \end{itemize}
    We construct a sequence of semigroup presentations, and the rooted semi tree automata corresponding to those presentations, such that each semigroup presentation is a basic splitting of its predecessor. We then apply foldings of type $2$ to the semi-tree automaton obtained by this process. Instead of writing all the generators and relations of those semigroup presentation, we provide figures illustrating minimal trees of those semigroup presentations. The first presentation in the sequence is $\mathcal{P}_r$ (see Figure \ref{f_Presentation_1}).
\begin{figure}[H]
\centering

\begin{tikzpicture}[grow'=down, level distance=1.4cm,
  level 1/.style={sibling distance=2cm},
  level 2/.style={sibling distance=1.5cm}]

\node {$r$}
    child {node {$g$}
        child{node {$g$}}
        child{node {$a_{2n-3}$}}
        child {node {$\cdots$} edge from parent[draw=none]}
        child{node {$a_0$}}}
    child {node {$a_{2n-3}$}}
    child {node {$\cdots$} edge from parent[draw=none]}
    child {node {$a_{1}$}}
    child {node {$f$}
        child{node {$a_0$}}
        child{node {$a_{2n-3}$}}
        child {node {$\cdots$} edge from parent[draw=none]}
        child{node {$a_1$}}
        child{node {$f$}}};
\end{tikzpicture}
\medskip
\begin{tikzpicture}[grow'=down, level distance=1.4cm,
  level 1/.style={sibling distance=2cm}]
\node {$a_i$}
    child {node {$a_i$}}
    child {node {$\cdots$} edge from parent[draw=none]}
    child {node {$a_{i+1}$}}
    child {node {$a_i$}};
\end{tikzpicture}
\caption{An illustration of minimal tree associated with the core of $F_{2n-1}$.}\label{f_Presentation_1}
\end{figure}
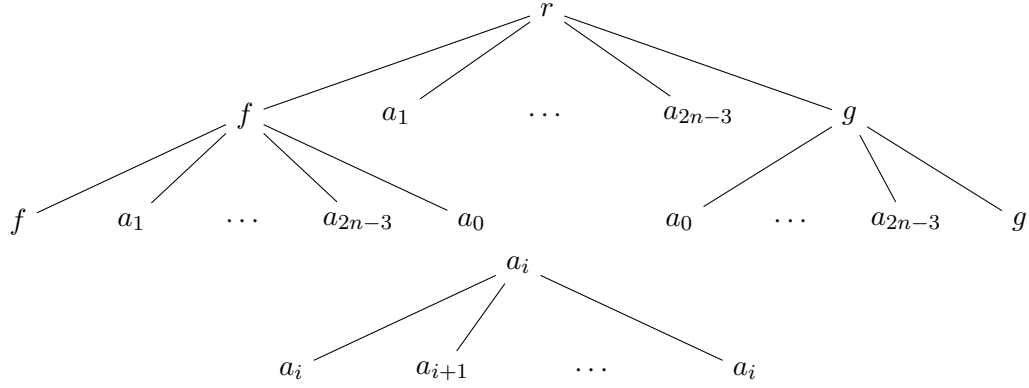

We construct a basic splitting of $\mathcal{P}_r$ by adding a new generator denoted by $f_0$, removing the relation $r= f\cdot a_1\cdots a_{2n-3}\cdot g$, and adding the relations $r=f_0\cdot a_n\cdots a_{2n-3}\cdot g$, and $f_0=f\cdot a_1\cdots a_{n-1}$ (see Figure \ref{f_Presentation_2})
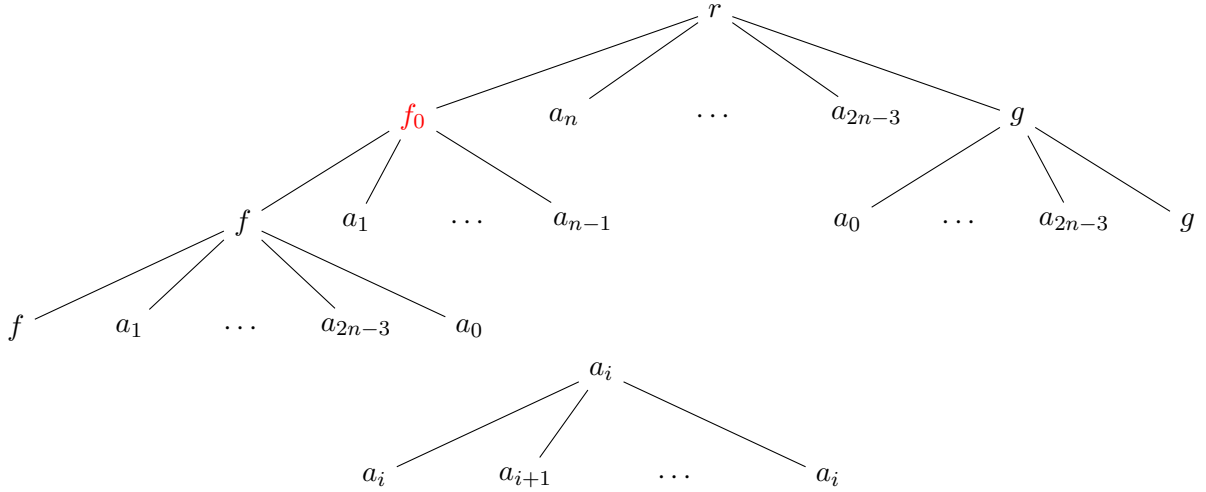
\begin{figure}[H]
\begin{tikzpicture}[grow'=down, level distance=1.4cm,
  level 1/.style={sibling distance=2cm},
  level 2/.style={sibling distance=1.5cm}]

\node {$r$}
    child {node {$g$}
        child{node {$g$}}
        child{node {$a_{2n-3}$}}
        child {node {$\cdots$} edge from parent[draw=none]}
        child{node {$a_0$}}}
    child {node {$a_{2n-3}$}}
    child {node {$\cdots$} edge from parent[draw=none]}
    child {node {$a_n$}}
    child {node[color=red] {$f_0$}
        child{node {$a_{n-1}$}}
        child {node {$\cdots$} edge from parent[draw=none]}
        child{node {$a_1$}}
        child{node {$f$}
            child{node {$a_0$}}
            child{node {$a_{2n-3}$}}
            child {node {$\cdots$} edge from parent[draw=none]}
            child{node {$a_1$}}
            child{node {$f$}}}};
\end{tikzpicture}
\medskip
\centering
\begin{tikzpicture}[grow'=down, level distance=1.4cm,
  level 1/.style={sibling distance=2cm}]
\node {$a_i$}
    child {node {$a_i$}}
    child {node {$\cdots$} edge from parent[draw=none]}
    child {node {$a_{i+1}$}}
    child {node {$a_i$}};
\end{tikzpicture}
\caption{The relation $r= f\cdot a_1\cdots a_{2n-3}\cdot g$ was replaced by the relations $r=f_0\cdot a_n\cdots a_{2n-3}\cdot g$, and $f_0=f\cdot a_1\cdots a_{n-1}$.}\label{f_Presentation_2}
\end{figure}

Next, we construct a basic splitting of the previous presentation by adding a new generator denoted by $c_{n-1}$, removing the relation $f= f\cdot a_1\cdots a_{2n-3}\cdot a_0$, and adding the relations $f=f\cdot a_1\cdots a_{n-2}\cdot c_{n-1}$, and $c_{n-1}=a_{n-1}\cdots a_{2n-3}\cdot a_0$ (see Figure \ref{f_Presentation_3}).

\begin{figure}[H]
\centering

\begin{tikzpicture}[grow'=down, level distance=1.4cm,
  level 1/.style={sibling distance=2cm},
  level 2/.style={sibling distance=1.5cm}]

\node {$r$}
    child {node {$g$}
        child{node {$g$}}
        child{node {$a_{2n-3}$}}
        child {node {$\cdots$} edge from parent[draw=none]}
        child{node {$a_0$}}}
    child {node {$a_{2n-3}$}}
    child {node {$\cdots$} edge from parent[draw=none]}
    child {node {$a_n$}}
    child {node[color=red] {$f_0$}
        child{node {$a_{n-1}$}}
        child {node {$\cdots$} edge from parent[draw=none]}
        child{node {$a_1$}}
        child{node {$f$}
            child{node[color=red] {$c_{n-1}$}
                child{node {$a_0$}}
                child{node {$a_{2n-3}$}}
                child{node {$\cdots$} edge from parent[draw=none]}
                child{node {$a_{n-1}$}}}
            child{node {$a_{n-2}$}}
            child {node {$\cdots$} edge from parent[draw=none]}
            child{node {$a_1$}}
            child{node {$f$}}}};
\end{tikzpicture}
\medskip
\begin{tikzpicture}[grow'=down, level distance=1.4cm,
  level 1/.style={sibling distance=2cm}]
\node {$a_i$}
    child {node {$a_i$}}
    child {node {$\cdots$} edge from parent[draw=none]}
    child {node {$a_{i+1}$}}
    child {node {$a_i$}};
\end{tikzpicture}
\caption{The relation $f= f\cdot a_1\cdots a_{2n-3}\cdot a_0$ was replaced by the relations $f=f\cdot a_1\cdots a_{n-2}\cdot c_{n-1}$, and $c_{n-1}=a_{n-1}\cdots a_{2n-3}\cdot a_0$}\label{f_Presentation_3}
\vspace{1.2cm}
\end{figure}
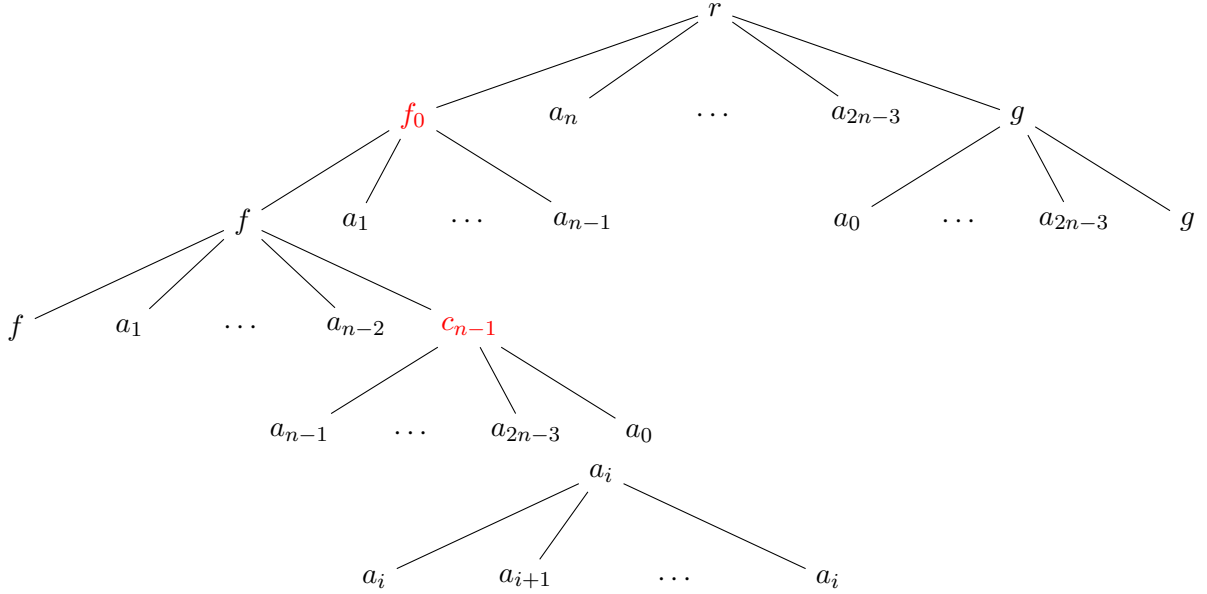
Next, we construct a basic splitting of the previous presentation by adding a new generator denoted by $c_0$, removing the relation $g= a_0\cdots a_{2n-3} \cdot g$, and adding the relations $g=c_0\cdot a_n\cdots a_{2n-3}\cdot g$, and $c_0=a_0\cdots a_{n-1}$ (see Figure \ref{f_Presentation_4}).
\begin{figure}[H]
\centering
\begin{tikzpicture}[grow'=down, level distance=1.4cm,
  level 1/.style={sibling distance=2cm},
  level 2/.style={sibling distance=1.5cm}]

\node {$r$}
    child {node {$g$}
        child{node {$g$}}
        child{node {$a_{2n-3}$}}
        child {node {$\cdots$} edge from parent[draw=none]}
        child{node[color=red] {$c_0$}
            child{ node{$a_{n-1}$}}
            child {node {$\cdots$} edge from parent[draw=none]}
            child{ node{$a_0$}}}}
    child {node {$a_{2n-3}$}}
    child {node {$\cdots$} edge from parent[draw=none]}
    child {node {$a_n$}}
    child {node[color=red] {$f_0$}
        child{node {$a_{n-1}$}}
        child {node {$\cdots$} edge from parent[draw=none]}
        child{node {$a_1$}}
        child{node {$f$}
            child{node {$c_{n-1}$}
                child{node {$a_0$}}
                child{node {$a_{2n-3}$}}
                child{node {$\cdots$} edge from parent[draw=none]}
                child{node {$a_{n-1}$}}}
            child{node {$a_{n-2}$}}
            child {node {$\cdots$} edge from parent[draw=none]}
            child{node {$a_1$}}
            child{node {$f$}}}};
\end{tikzpicture}
\medskip
\begin{tikzpicture}[grow'=down, level distance=1.4cm,
  level 1/.style={sibling distance=2cm}]
\node {$a_i$}
    child {node {$a_i$}}
    child {node {$\cdots$} edge from parent[draw=none]}
    child {node {$a_{i+1}$}}
    child {node {$a_i$}};
\end{tikzpicture}
\caption{The relation $g= a_0\cdots a_{2n-3} \cdot g$ was replaced by the relations $g=c_0\cdot a_n\cdots a_{2n-3}\cdot g$, and $c_0=a_0\cdots a_{n-1}$.}\label{f_Presentation_4}
\vspace{1.2cm}
\end{figure}
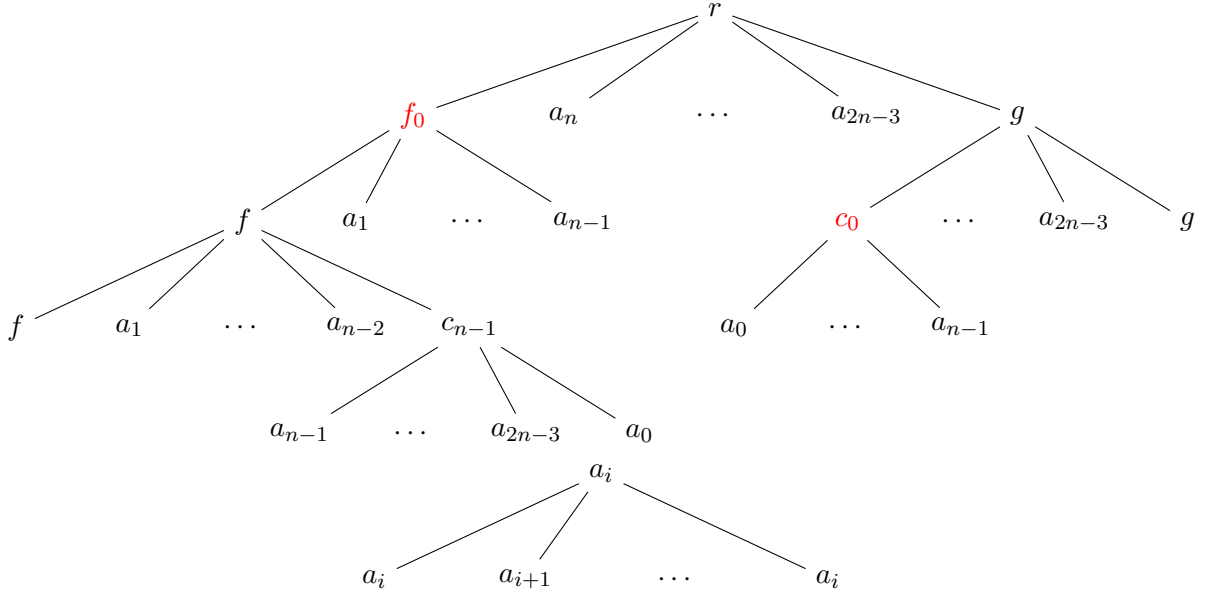
We now construct a splitting of the previous presentation by adding new generators denoted by $b_0,\dots,b_{2n-3}$. For each $i\in\{0,\dots,2n-3\}$, we remove the relation $a_i= a_i \cdot a_{i+1} \cdots a_i$. We add the following relations (see Figure \ref{f_Presentation_5}).
\begin{itemize}
    \item $a_0=b_0\cdot a_n\cdots a_{2n-3}\cdot a_0$ and $b_0=a_0\cdots a_{n-1}$.
    \item $a_i=a_i\cdots a_{n-2}\cdot b_i\cdot a_1\cdots a_i$ and $b_i=a_{n-1}\cdots a_{2n-3}\cdot a_0$, for $i\in\{1,\dots,n-2\}$.
    \item $a_{n-1}=a_{n-1}\cdots a_{2n-3}\cdot b_{n-1}$ and $b_{n-1}=a_0\cdots a_{n-1}$.
    \item $a_i=a_i\cdots a_{2n-3}\cdot b_i\cdot a_n\cdots a_i$ and $b_i=a_0\cdots a_{n-1}$, for $i\in\{n,\dots,2n-3\}$.
\end{itemize}
\begin{figure}[H]
\begin{minipage}{0.45\textwidth}
\centering
\begin{tikzpicture}[grow'=down, level distance=1.4cm,
  level 1/.style={sibling distance=1.8cm},
  level 2/.style={sibling distance=1.2cm},
  level 3/.style={sibling distance=0.8cm}]
\node {$r$}
    child {node {$g$}
        child{node {$g$}}
        child{node {$a_{2n-3}$}}
        child {node {$\cdots$} edge from parent[draw=none]}
        child{node {$a_n$}}
        child{node {$c_0$}
            child{node {$a_{n-1}$}}
            child{node {$\cdots$} edge from parent[draw=none]}
            child{node {$a_{0}$}}}}
    child {node {$a_{2n-3}$}}
    child {node {$\cdots$} edge from parent[draw=none]}
    child {node {$a_{n}$}}
    child {node {$f_0$}
        child{node {$a_{n-1}$}}
        child {node {$\cdots$} edge from parent[draw=none]}
        child{node {$a_1$}}
        child{node {$f$}
            child{node {$c_{n-1}$}
                child{node {$a_0$}}
                child{node {$a_{2n-3}$}}
                child{node {$\cdots$} edge from parent[draw=none]}
                child{node {$a_{n-1}$}}}
            child{node {$a_{n-2}$}}
            child {node {$\cdots$} edge from parent[draw=none]}
            child{node {$a_1$}}
            child{node {$f$}}}};
\end{tikzpicture}
\end{minipage}
\vspace{1.2cm}
\end{figure}

\begin{figure}[H]
\begin{minipage}{0.45\textwidth}
\centering
\begin{tikzpicture}[grow'=down, baseline=(current bounding box.north), level distance=1.4cm,
  level 1/.style={sibling distance=1cm}]
\node {$a_0$}
    child {node {$a_0$}}
    child {node {$a_{2n-3}$}}
    child {node {$\cdots$} edge from parent[draw=none]}
    child {node {$a_n$}}
    child {node[color=red] {$b_0$}
        child{node {$a_{n-1}$}}
        child{node {$\cdots$} edge from parent[draw=none]}
        child{node {$a_{0}$}}};
\end{tikzpicture}
\end{minipage}%
\hfill
\begin{minipage}{0.45\textwidth}
\centering
\begin{tikzpicture}[grow'=down, level distance=1.4cm,
  level 1/.style={sibling distance=1cm}]
\node {$a_{n-1}$}
    child {node[color=red] {$b_{n-1}$}
        child{node {$a_{n-1}$}}
        child{node {$\cdots$} edge from parent[draw=none]}
        child{node {$a_{0}$}}}
    child {node {$a_{2n-3}$}}
    child {node {$\cdots$} edge from parent[draw=none]}
    child {node {$a_{n-1}$}};
\end{tikzpicture}
\end{minipage}
\vspace{1.2cm}
\begin{minipage}{0.45\textwidth}
\centering
\begin{tikzpicture}[grow'=down, level distance=1.4cm,
  level 1/.style={sibling distance=1cm}]
\node {$a_i$}
    child {node {$a_i$}}
    child {node {$\cdots$} edge from parent[draw=none]}
    child {node {$a_1$}}
    child {node[color=red] {$b_i$}
        child{node {$a_0$}}
        child{node {$a_{2n-3}$}}
        child{node {$\cdots$} edge from parent[draw=none]}
        child{node {$a_{n-1}$}}}
    child {node {$a_{n-2}$}}
    child {node {$\cdots$} edge from parent[draw=none]}
    child {node {$a_i$}};
\end{tikzpicture}
\caption*{For $1 \le i \le n-2$}
\end{minipage}%
\hfill
\hspace{1.2cm}
\begin{minipage}{0.45\textwidth}
\centering
\begin{tikzpicture}[grow'=down, level distance=1.4cm,
  level 1/.style={sibling distance=1cm}]
\node {$a_i$}
    child {node {$a_i$}}
    child {node {$\cdots$} edge from parent[draw=none]}
    child {node {$a_n$}}
    child {node[color=red] {$b_i$}
        child{node {$a_{n-1}$}}
        child{node {$\cdots$} edge from parent[draw=none]}
        child{node {$a_{0}$}}}
    child {node {$a_{2n-3}$}}
    child {node {$\cdots$} edge from parent[draw=none]}
    child {node {$a_i$}};
\end{tikzpicture}
\caption*{For $n \le i \le 2n-3$}
\end{minipage}
\caption{The relations $a_i=a_i\cdot a_{i+1}\cdots a_i$, where $i\in\{0,\dots,2n-3\}$ were removed, and new relations and generators were added.}\label{f_Presentation_5}
\end{figure}

We conclude the sequence of splittings here and proceed to apply foldings of type $2$ to the automaton corresponding to the presentation. Note that the children of $c_0,b_0,b_{n-1},\dots,b_{2n-3}$ are the same, and that the children of $c_{n-1},b_1,\dots,b_{n-2}$ are the same. We apply foldings of type $2$, to identify $c_0,b_0,b_{n-1},\dots,b_{2n-3}$ to a vertex denoted by $c_0$. We also apply foldings of type $2$ to identify $c_{n-1},b_1,\dots,b_{n-2}$ to a vertex denoted by $c_{n-1}$ (see Figure \ref{f_Presentation_6}).

\newpage
\begin{figure}[H]
\begin{minipage}{0.45\textwidth}
\centering
\begin{tikzpicture}[grow'=down, level distance=1.4cm,
  level 1/.style={sibling distance=1.8cm},
  level 2/.style={sibling distance=1.2cm},
  level 3/.style={sibling distance=0.8cm}]
\node {$r$}
    child {node {$g$}
        child{node {$g$}}
        child{node {$a_{2n-3}$}}
        child {node {$\cdots$} edge from parent[draw=none]}
        child{node {$a_n$}}
        child{node[color=blue] {$c_0$}
            child{node {$a_{n-1}$}}
            child{node {$\cdots$} edge from parent[draw=none]}
            child{node {$a_{0}$}}}}
    child {node {$a_{2n-3}$}}
    child {node {$\cdots$} edge from parent[draw=none]}
    child {node {$a_{n}$}}
    child {node {$f_0$}
        child{node {$a_{n-1}$}}
        child {node {$\cdots$} edge from parent[draw=none]}
        child{node {$a_1$}}
        child{node {$f$}
            child{node[color=red] {$c_{n-1}$}
                child{node {$a_0$}}
                child{node {$a_{2n-3}$}}
                child{node {$\cdots$} edge from parent[draw=none]}
                child{node {$a_{n-1}$}}}
            child{node {$a_{n-2}$}}
            child {node {$\cdots$} edge from parent[draw=none]}
            child{node {$a_1$}}
            child{node {$f$}}}};
\end{tikzpicture}
\end{minipage}
\end{figure}
\begin{figure}[H]
\ContinuedFloat
\begin{minipage}{0.45\textwidth}
\centering
\begin{tikzpicture}[grow'=down, baseline=(current bounding box.north), level distance=1.4cm,
  level 1/.style={sibling distance=1cm}]
\node {$a_0$}
    child {node {$a_0$}}
    child {node {$a_{2n-3}$}}
    child {node {$\cdots$} edge from parent[draw=none]}
    child {node {$a_n$}}
    child {node[color=blue] {$b_0$}
        child{node {$a_{n-1}$}}
        child{node {$\cdots$} edge from parent[draw=none]}
        child{node {$a_{0}$}}};
\end{tikzpicture}
\end{minipage}%
\hfill
\begin{minipage}{0.45\textwidth}
\centering
\begin{tikzpicture}[grow'=down, level distance=1.4cm,
  level 1/.style={sibling distance=1cm}]
\node {$a_{n-1}$}
    child {node[color=blue] {$b_{n-1}$}
        child{node {$a_{n-1}$}}
        child{node {$\cdots$} edge from parent[draw=none]}
        child{node {$a_{0}$}}}
    child {node {$a_{2n-3}$}}
    child {node {$\cdots$} edge from parent[draw=none]}
    child {node {$a_{n-1}$}};
\end{tikzpicture}
\end{minipage}
\vspace{1.2cm}
\begin{minipage}{0.45\textwidth}
\centering
\begin{tikzpicture}[grow'=down, level distance=1.4cm,
  level 1/.style={sibling distance=1cm}]
\node {$a_i$}
    child {node {$a_i$}}
    child {node {$\cdots$} edge from parent[draw=none]}
    child {node {$a_1$}}
    child {node[color=red] {$b_i$}
        child{node {$a_0$}}
        child{node {$a_{2n-3}$}}
        child{node {$\cdots$} edge from parent[draw=none]}
        child{node {$a_{n-1}$}}}
    child {node {$a_{n-2}$}}
    child {node {$\cdots$} edge from parent[draw=none]}
    child {node {$a_i$}};
\end{tikzpicture}
\caption*{For $1 \le i \le n-2$}
\end{minipage}%
\hfill
\hspace{1.2cm}
\begin{minipage}{0.45\textwidth}
\centering
\begin{tikzpicture}[grow'=down, level distance=1.4cm,
  level 1/.style={sibling distance=1cm}]
\node {$a_i$}
    child {node {$a_i$}}
    child {node {$\cdots$} edge from parent[draw=none]}
    child {node {$a_n$}}
    child {node[color=blue] {$b_i$}
        child{node {$a_{n-1}$}}
        child{node {$\cdots$} edge from parent[draw=none]}
        child{node {$a_{0}$}}}
    child {node {$a_{2n-3}$}}
    child {node {$\cdots$} edge from parent[draw=none]}
    child {node {$a_i$}};
\end{tikzpicture}
\captionsetup{labelformat=empty}
\caption{For $n \le i \le 2n-3$}
\end{minipage}
\caption{$c_0,b_0,b_{n-1},\dots,b_{2n-3}$ are identified by foldings of type $2$ to a vertex denoted by $c_0$. $c_{n-1},b_1,\dots,b_{n-2}$ are identified by foldings of type $2$ to a vertex denoted by $c_{n-1}$.}\label{f_Presentation_6}
\end{figure}
\newpage
Let $\mathcal{Q}_r$ be the resulting semigroup presentation. It is generated by the elements $r,f,f_0,g,a_0,\dots,a_{2n-3},c_0,c_{n-1}$, and has the following relations.
    \begin{itemize}
        \item $r= f_0\cdot a_n\cdots a_{2n-3}\cdot g$.
        \item $f_0=f\cdot a_1\cdots a_{n-1}$.
        \item $f= f\cdot a_1\cdots a_{n-2}\cdot c_{n-1}$.
        \item $g=c_0\cdot a_n\cdots a_{2n-3}\cdot g$.
        \item $c_0= a_0\cdots a_{n-1}$.
        \item $c_{n-1}= a_{n-1}\cdots a_{2n-3}\cdot a_0$.
        \item $a_0= c_0\cdot a_n\cdots a_{2n-3}\cdot a_0$.
        \item $a_i= a_i\cdots a_{n-2}\cdot c_{n-1}\cdot a_1\cdots a_i$ for $1\le i\le n-2$
        \item $a_{n-1}= a_{n-1}\cdots a_{2n-3}\cdot c_0$.
        \item $a_i= a_i\cdots a_{2n-3}\cdot c_0\cdot a_n\cdots a_i$ for $n\le i\le 2n-3$.
    \end{itemize}
    Note that $\mathcal{Q}_r$ was obtained by applying foldings of type $2$ to a splitting of $\mathcal{P}_r$.

    Let $\mathcal{A}_r:=\mathcal{A}(\mathcal{Q}_r)$ (see Figure \ref{f A_r and the core of F_{2n-1}}). Note that $\mathcal{A}_r$ is an $n$-ary tree automaton. Let $H$ be the subgroup of $F_n$ accepted by $\mathcal{A}_r$. Note that $\mathcal{A}_r$ is reduced (since, for example, there is a non-empty path from every vertex of $\mathcal{A}_r$ to $a_0)$. So  by Corollary \ref{c_core_splitting_does_not_change_core}, $\mathcal{A}_r$ is a core automaton.

Guba and Sapir \cite{GubaSapir1997} showed that certain changes to a semigroup presentation do not change the isomorphism type of diagram groups associated with it. In particular, they showed that the diagram group of a semi tree semigroup presentation and the diagram group of a splitting of the presentation are isomorphic. Combining their result with Lemma \ref{l_core_foldings_type_2_diagram_groups_isomorphism}, we conclude that $H$ is isomorphic to $F_{2n-1}$.

We will show that the following hold:
    \begin{itemize}
        \item[$(1)$]$H$ is a strict subgroup of $F_n$, and conditions $(a)-(d)$ of Corollary \ref{C_maximal_1} hold for $H$.
        \item[$(2)$]$H$ does not fix any point in the open unit interval, and it is a subgroup of infinite index of $F_n$.
    \end{itemize}
    Note that after we show $(1)$ and $(2)$, we are done.
\begin{figure}[H]
\begin{minipage}{0.45\textwidth}
\centering
\begin{tikzpicture}[grow'=down, level distance=1.4cm,
  level 1/.style={sibling distance=1.8cm},
  level 2/.style={sibling distance=1.2cm},
  level 3/.style={sibling distance=0.8cm}]
\node {$r$}
    child {node {$g$}
        child{node {$g$}}
        child{node {$a_{2n-3}$}}
        child {node {$\cdots$} edge from parent[draw=none]}
        child{node {$a_n$}}
        child{node {$c_0$}
            child{node {$a_{n-1}$}}
            child{node {$\cdots$} edge from parent[draw=none]}
            child{node {$a_{0}$}}}}
    child {node {$a_{2n-3}$}}
    child {node {$\cdots$} edge from parent[draw=none]}
    child {node {$a_{n}$}}
    child {node {$f_0$}
        child{node {$a_{n-1}$}}
        child {node {$\cdots$} edge from parent[draw=none]}
        child{node {$a_1$}}
        child{node {$f$}
            child{node {$c_{n-1}$}
                child{node {$a_0$}}
                child{node {$a_{2n-3}$}}
                child{node {$\cdots$} edge from parent[draw=none]}
                child{node {$a_{n-1}$}}}
            child{node {$a_{n-2}$}}
            child {node {$\cdots$} edge from parent[draw=none]}
            child{node {$a_1$}}
            child{node {$f$}}}};
\end{tikzpicture}
\end{minipage}
\vspace{1.2cm}
\end{figure}
\begin{figure}[H]
\ContinuedFloat
\begin{minipage}{0.45\textwidth}
\centering
\begin{tikzpicture}[grow'=down, baseline=(current bounding box.north), level distance=1.4cm,
  level 1/.style={sibling distance=1cm}]
\node {$a_0$}
    child {node {$a_0$}}
    child {node {$a_{2n-3}$}}
    child {node {$\cdots$} edge from parent[draw=none]}
    child {node {$a_n$}}
    child {node {$c_0$}};
\end{tikzpicture}
\end{minipage}%
\hfill
\begin{minipage}{0.45\textwidth}
\centering
\begin{tikzpicture}[grow'=down, level distance=1.4cm,
  level 1/.style={sibling distance=1cm}]
\node {$a_{n-1}$}
    child {node {$c_0$}}
    child {node {$a_{2n-3}$}}
    child {node {$\cdots$} edge from parent[draw=none]}
    child {node {$a_{n-1}$}};
\end{tikzpicture}
\end{minipage}
\vspace{1.2cm}
\begin{minipage}{0.45\textwidth}
\centering
\begin{tikzpicture}[grow'=down, level distance=1.4cm,
  level 1/.style={sibling distance=1cm}]
\node {$a_i$}
    child {node {$a_i$}}
    child {node {$\cdots$} edge from parent[draw=none]}
    child {node {$a_1$}}
    child {node {$c_{n-1}$}}
    child {node {$a_{n-2}$}}
    child {node {$\cdots$} edge from parent[draw=none]}
    child {node {$a_i$}};
\end{tikzpicture}
\caption*{For $1 \le i \le n-2$}
\end{minipage}%
\hfill
\begin{minipage}{0.45\textwidth}
\centering
\begin{tikzpicture}[grow'=down, level distance=1.4cm,
  level 1/.style={sibling distance=1cm}]
\node {$a_i$}
    child {node {$a_i$}}
    child {node {$\cdots$} edge from parent[draw=none]}
    child {node {$a_n$}}
    child {node {$c_{0}$}}
    child {node {$a_{2n-3}$}}
    child {node {$\cdots$} edge from parent[draw=none]}
    child {node {$a_i$}};
\end{tikzpicture}
\captionsetup{labelformat=empty}
\caption{For $n \le i \le 2n-3$}
\end{minipage}
\caption{This figure illustrates a minimal tree of the $n$-ary tree automaton $\mathcal{A}_r$. Note that $\mathcal{A}_r$ is obtained by applying foldings of type $2$ to a splitting of $\mathcal{L}(F_{2n-1})$, so $\mathcal{A}_r$ is a rooted $n$-ary tree automaton, and it is a core automaton.}\label{f A_r and the core of F_{2n-1}}
\end{figure}

    We show condition $(1)$ holds. Any element of $F_n$ which has the pair of branches $0\rightarrow 00$ (such an element clearly exists) is not accepted by $\mathcal{A}_r$. Therefore, $H\neq F_n$, so $H$ is a strict subgroup of $F_n$.

    Note that since $\mathcal{A}_r$ is a core automaton, it follows that $\mathcal{L}(H)$ is equal to $\mathcal{A}_r$. So $H$ equals to the subgroup of $F_n$ accepted by $\mathcal{L}(H)$, which is $\Cl(H)$.

    We now show that condition $(a)$ holds. We need to show that $\pi_{ab}(H)=\Z^n$. It is enough to find elements $h_0,\dots,h_{n-1}$ of $F_n$ that are accepted by $\mathcal{A}_r$, such that their images under the abelianization map generate $\Z^n$. Let $x_0,\dots,x_{n-1}$ be the standard generators of $F_n$. We define the following elements. \begin{itemize}
        \item $h_0=x_0x_{n-1}x_0^{-1}$.
        \item for each $1\le i\le n-2$, $h_i=x_0x_ix_{n-1}x_{2n-2}^{-1}x_{n-1}^{-1}x_0^{-1}$.
        \item $h_{n-1}=x_0^2x_{n-1}x_{2n-2}^{-1}x_{n-1}^{-1}x_0^{-1}$.
    \end{itemize}
    Note that they are accepted by $\mathcal{A}_r$. See Figure \ref{elements for abelianization} for tree-diagrams of these elements.
    Consider the following presentation of $F_n$: $$\langle x_0,x_1,\dots|x_j^{x_i}=x_{j+n-1} \text{ for }i<j\rangle$$ For natural $i\ge 0$, denote $[x_i]$ the coset of $x_i$ in the abelianization. Recall that for $i\ge1$, $[x_i]=[x_{i+n-1}]$, so the abelianization of $F_n$ is generated by $[x_0],\dots,[x_{n-1}]$. Note that
    \begin{itemize}
        \item $[h_0]=[x_{n-1}]$.
        \item for each $1\le i\le n-2$, $[h_i]=[x_i][x_{n-1}^{-1}]$.
        \item $[h_{n-1}]=[x_0][x_{n-1}^{-1}]$.
    \end{itemize}
    So $[h_0],\dots,[h_{n-1}]$ generate the abelianization, and condition $(a)$ holds.
\begin{figure}[H]
\centering

\begin{minipage}[t]{0.48\textwidth}
\centering
\begin{tikzpicture}[grow'=down,
  baseline=(current bounding box.north), 
  level distance=1.6cm,
  every node/.style={minimum size=5mm,inner sep=0pt},
  level 1/.style={sibling distance=1cm},
  level 2/.style={sibling distance=1.2cm},
  level 3/.style={sibling distance=0.8cm},
  ]

\node {$r$}
    child {node {$g$}}
    child {node {$a_{2n-3}$}}
    child {node {$\cdots$} edge from parent[draw=none]}
    child {node {$a_n$}}
    child {node {$f_0$}
        child{node {$a_{n-1}$}
            child {node {$c_0$}}
            child {node {$a_{2n-3}$}}
            child {node {$\cdots$} edge from parent[draw=none]}
            child {node {$a_{n-1}$}}}
        child {node {$\cdots$} edge from parent[draw=none]}
        child{node {$a_1$}}
        child{node {$f$}}};
\end{tikzpicture}
\end{minipage}%
\hfill
\begin{minipage}[t]{0.48\textwidth}
\centering
\begin{tikzpicture}[grow'=down,
  baseline=(current bounding box.north), 
  level distance=1.6cm,
  every node/.style={minimum size=5mm,inner sep=0pt},
  level 1/.style={sibling distance=1cm},
  level 2/.style={sibling distance=0.8cm},
  ]

\node {$r$}
    child {node {$g$}
        child{node {$g$}}
        child{node {$a_{2n-3}$}}
        child {node {$\cdots$} edge from parent[draw=none]}
        child{node {$a_n$}}
        child{node {$c_0$}}}
    child {node {$a_{2n-3}$}}
    child {node {$\cdots$} edge from parent[draw=none]}
    child {node {$a_n$}}
    child {node {$f_0$}
        child{node {$a_{n-1}$}}
        child {node {$\cdots$} edge from parent[draw=none]}
        child{node {$a_1$}}
        child{node {$f$}}};
\end{tikzpicture}
\end{minipage}%
\hfill
\caption*{A tree diagram of the element $h_0$}
\vspace{1.2cm}
\centering
\begin{minipage}[t]{0.48\textwidth}
\centering
\begin{tikzpicture}[grow'=down,
  baseline=(current bounding box.north), 
  level distance=1.6cm,
  every node/.style={minimum size=5mm,inner sep=0pt},
  level 1/.style={sibling distance=1cm},
  level 2/.style={sibling distance=1.2cm},
  level 3/.style={sibling distance=0.8cm},
  ]

\node {$r$}
    child {node {$g$}}
    child {node {$a_{2n-3}$}}
    child {node {$\cdots$} edge from parent[draw=none]}
    child {node {$a_n$}}
    child {node {$f_0$}
        child{node {$a_{n-1}$}}
        child {node {$\cdots$} edge from parent[draw=none]}
        child{node {$a_i$}
            child {node {$a_i$}}
            child {node {$\cdots$} edge from parent[draw=none]}
            child {node {$a_1$}}
            child {node {$c_{n-1}$}
                child{node {$a_0$}}
                child{node {$a_{2n-3}$}}
                child{node {$\cdots$} edge from parent[draw=none]}
                child{node {$a_{n-1}$}}}
            child {node {$a_{n-2}$}}
            child {node {$\cdots$} edge from parent[draw=none]}
            child {node {$a_i$}}}
        child {node {$\cdots$} edge from parent[draw=none]}
        child{node {$a_1$}}
        child{node {$f$}}};
\end{tikzpicture}
\end{minipage}%
\hfill
\begin{minipage}[t]{0.48\textwidth}
\centering
\begin{tikzpicture}[grow'=down,
  baseline=(current bounding box.north), 
  level distance=1.6cm,
  every node/.style={minimum size=5mm,inner sep=0pt},
  level 1/.style={sibling distance=1cm},
  level 2/.style={sibling distance=1.2cm},
  level 3/.style={sibling distance=0.8cm},
  ]

\node {$r$}
    child {node {$g$}}
    child {node {$a_{2n-3}$}}
    child {node {$\cdots$} edge from parent[draw=none]}
    child {node {$a_n$}}
    child {node {$f_0$}
        child{node {$a_{n-1}$}
            child {node {$c_0$}
                child{node {$a_{n-1}$}}
                child{node {$\cdots$} edge from parent[draw=none]}
                child{node {$a_{0}$}}}
            child {node {$a_{2n-3}$}}
            child {node {$\cdots$} edge from parent[draw=none]}
            child {node {$a_{n-1}$}}}
        child {node {$\cdots$} edge from parent[draw=none]}
        child{node {$a_1$}}
        child{node {$f$}}};
\end{tikzpicture}
\end{minipage}%
\caption*{A tree diagram of the element $h_i$ for $1\le i\le n-2$}
\end{figure}
\begin{figure}[H]
\ContinuedFloat
\begin{minipage}[t]{0.48\textwidth}
\centering
\begin{tikzpicture}[grow'=down,
  baseline=(current bounding box.north), 
  level distance=1.6cm,
  every node/.style={minimum size=5mm,inner sep=0pt},
  level 1/.style={sibling distance=1cm},
  level 2/.style={sibling distance=1.2cm},
  level 3/.style={sibling distance=0.8cm},
  ]

\node {$r$}
    child {node {$g$}}
    child {node {$a_{2n-3}$}}
    child {node {$\cdots$} edge from parent[draw=none]}
    child {node {$a_n$}}
    child {node {$f_0$}
        child{node {$a_{n-1}$}}
        child {node {$\cdots$} edge from parent[draw=none]}
        child{node {$a_1$}}
        child{node {$f$}
            child{node {$c_{n-1}$}
                child{node {$a_{0}$}}
                child{node {$a_{2n-3}$}}
                child{node {$\cdots$} edge from parent[draw=none]}
                child{node {$a_{n-1}$}}}
            child{node {$a_{n-2}$}}
            child {node {$\cdots$} edge from parent[draw=none]}
            child{node {$a_1$}}
            child{node {$f$}}}};
\end{tikzpicture}
\end{minipage}%
\begin{minipage}[t]{0.48\textwidth}
\centering
\begin{tikzpicture}[grow'=down,
  baseline=(current bounding box.north), 
  level distance=1.6cm,
  every node/.style={minimum size=5mm,inner sep=0pt},
  level 1/.style={sibling distance=1cm},
  level 2/.style={sibling distance=1.2cm},
  level 3/.style={sibling distance=0.8cm},
  ]

\node {$r$}
    child {node {$g$}}
    child {node {$a_{2n-3}$}}
    child {node {$\cdots$} edge from parent[draw=none]}
    child {node {$a_n$}}
    child {node {$f_0$}
        child{node {$a_{n-1}$}
            child {node {$c_0$}
                child{node {$a_{n-1}$}}
                child{node {$\cdots$} edge from parent[draw=none]}
                child{node {$a_{0}$}}}
            child {node {$a_{2n-3}$}}
            child {node {$\cdots$} edge from parent[draw=none]}
            child {node {$a_{n-1}$}}}
        child {node {$\cdots$} edge from parent[draw=none]}
        child{node {$a_1$}}
        child{node {$f$}}};
\end{tikzpicture}
\end{minipage}
\captionsetup{labelformat=empty}
\caption{A tree diagram of the element $h_{n-1}$}
\captionsetup{labelformat=default}
\caption{Tree-diagrams of the elements $h_0,\dots,h_{n-1}$. Each vertex in those tree-diagrams is labeled by its appropriate label in $\mathcal{T}_{\mathcal{A}_r}$. Note that every vertex is labeled, and that the end-vertices of each pair of branches share a label, so the elements are accepted by $\mathcal{A}_r$.}
\label{elements for abelianization}
\end{figure}

    We move onto condition $(b)$. We take the elements $h_0,\dots,h_{n-2}\in H$. Let $\alpha_0=.0(n-1)=\frac{n-1}{n^2}$, and for $1\le i\le n-2$, let $\alpha_i=.0i=\frac{i}{n^2}$.
    Note that for each $i\in\{0,\dots,n-2\}$, we have $\alpha_i\in D_{n,i}$, and $h_i(\alpha_i)=\alpha_i,h_i'(\alpha_i^-)=1,h_i'(\alpha_i^+)=n$, so condition $(b)$ holds.

    We move to conditions $(c)$ and $(d)$. Let $g_1\in F_n\setminus H$. Recall that $H=\Cl(H)$ and that every $n$-ary word labels a path on $\mathcal{L}(H)$. So there are $n$-ary words such that $g_1$ has a pair of branches $u\rightarrow v$ such that $u^+\neq v^+$ on $\mathcal{L}(H)$. Therefore, in $\mathcal{L}_{sem}(H\cup\{g_1\})$, we identify the vertices $u^+$ and $v^+$. Note that in particular, $t(u)=t(v)$. Also note that since for any $k\in\N\setminus\{0\}$, $((n-1)^k)^+=g$, neither $u$ nor $v$ can be of the form $(n-1)^k$ (since if one of them is of this form, the other must be of the form $(n-1)^\ell$ and then  they both end in the vertex $g$ in $\mathcal{A}_r$). Also, if $u$ or $v$ are of the form $0^k$ for some $k\in\N\setminus\{0\}$, then we must have $\{u^+,v^+\}=\{f,f_0\}$. So $u^+\neq v^+$, and one of the following must hold:
    \begin{itemize}
        \item[$(1)$] $u^+,v^+\in\{f_0,f\}$
        \item[$(2)$] $u^+,v^+\in\{a_0,a_{n-1},c_0,c_{n-1}\}$
        \item[$(3)$] $u^+,v^+\in\{a_i,a_{i+n-1}\}$ for some $1\le i \le n-2$
    \end{itemize}
    Consider $\mathcal{L}_{sem}(H\cup\{g_1\})$. One can construct it by taking $\mathcal{L}(H)$, identifying $u^+$ and $v^+$, and the other end vertices of the pairs of branches of $g_1$, and apply foldings of type $1$. We will show that in each case between $(1),(2),(3)$, there is an $n$-ary word $w$ such that for every two $n$-ary words $w_1,w_2$ such that $t(w_1)=t(w_2)$, the words $ww_1,ww_2$ have $(ww_1)^+=(ww_2)^+$ on $\mathcal{L}_{sem}(H\cup\{g_1\})$, showing condition $(c)$ holds. We will also show that after applying foldings of type $2$ as well, we get the core of $F_n$, proving condition $(d)$ holds.

    We now construct a directed graph $G(\mathcal{A}_r)=(V,E)$ that simulates the identifications of vertices in the automaton after foldings of type $1$. The set of vertices is the set of all possible pairs of vertices we identify. Namely,
    \begin{align*}
        V=\{(f_0,f),(c_0,c_{n-1}),(c_0,a_0),(c_0,a_{n-1}),(c_{n-1},a_0),\\(c_{n-1},a_{n-1}),(a_0,a_{n-1}),(a_1,a_n),\dots,(a_{n-2},a_{2n-3})\}
    \end{align*}We add an edge from the vertex $(u_1,v_1)$ to the vertex $(u_2,v_2)$ if the identification of $(u_2,v_2)$ follows from the identification of $(u_1,v_1)$ by applying a single folding of type $1$. Let us first consider the case where $n\ge3$.
    See Figure \ref{f_maximal_dependencies_graph} for the graph.

    Note that there is a path between any vertex in $V$ to any vertex in \\$V\setminus\{(c_{n-1},a_{n-1}),(c_{n-1},a_0),(f,f_0)\}$. So in any case between $(1),(2),(3)$, the identification $u^+=v^+$ will result in the identifications of $c_0=c_{n-1}=a_0=a_{n-1}$ and $a_i=a_{i+n-1}$ for each $i\in\{1,\dots,n-2\}$ after some foldings of type $1$. Therefore $\mathcal{L}_{sem}(H\cup\{g_1\})$ has exactly $n-1$ inner vertices. Since every $n$-ary word labels a path on $\mathcal{L}(H)$, every $n$-ary word also labels a path on $\mathcal{L}_{sem}(H\cup\{g_1\})$. In particular, for the word $w=1$, for every $n$-ary words $w_1,w_2$ such that $t(w_1)=t(w_2)$ it holds that $ww_1,ww_2$ label paths on $\mathcal{L}_{sem}(H\cup\{g_1\})$ and $(ww_1)^+=(ww_2)^+$ on $\mathcal{L}_{sem}(H\cup\{g_1\})$. So condition $(c)$ is satisfied.

    Also note that the children of $f$ and $f_0$ will be the same in $\mathcal{L}_{sem}(H\cup\{g_1\})$, so $f$ and $f_0$ will be identified in $\mathcal{L}(H\cup\{g_1\})$ by a folding of type $2$. So $\mathcal{L}(H\cup\{g_1\})=\mathcal{L}(F_n)$, so it accepts the derived subgroup of $F_n$, and condition $(d)$ holds.

    Finally, if $n=2$, we get that $a_0,a_{n-1}$ and $c_0$ must be identified after applying foldings of type $1$. The children of $a_0$ and $c_{n-1}$ will be the same, so after applying foldings of type $2$ as well, the new automaton will accept $[F_n,F_n]$, so condition $(d)$ holds. By results of Golan (\cite{G1}) for the case of $F_2$, we do not have to show condition $(c)$, and we will get that $\langle H,g\rangle=F_2$, so we are done.

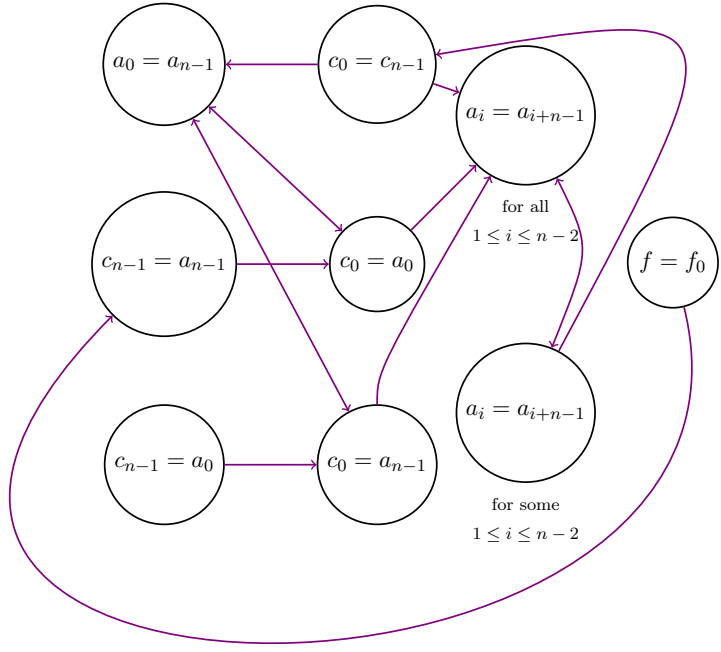
\begin{figure}[H]
    \centering
    \resizebox{1\linewidth}{!}{
    \begin{tikzpicture}[node distance={15mm}, thick, main/.style = {draw, circle,align=center}] 

    \node at (-8,0) {}; 
    \node at (8,0) {};  

    \node[main,minimum size=1 cm] (1) {$c_0=c_{n-1}$}; 
    \node[main,minimum size=1 cm] (2) [below=of 1] {$c_0=a_0$}; 
    \node[main,minimum size=1 cm] (3) [below=of 2] {$c_0=a_{n-1}$}; 
    \node[main,minimum size=1 cm] (4) [left=of 1] {$a_0=a_{n-1}$}; 
    \node[main,minimum size=1 cm] (5) [left=of 2] {$c_{n-1}=a_{n-1}$};
    \node[main,minimum size=1 cm] (6) [left=of 3] {$c_{n-1}=a_0$}; 
    \node[main, minimum size=1cm] (7) [below right=of 2] {$a_i = a_{i+n-1}$};
    \node[below=0.1cm of 7] (7') {\scriptsize for some};
    \node[below=0.001cm of 7'] (7'') {\scriptsize $1 \le i \le n-2$};
    \node[main,minimum size=1 cm] (8) [above right=of 2] {$a_i=a_{i+n-1}$};
    \node[below=0.1cm of 8] (8') {\scriptsize for all};
    \node[below=0.001cm of 8'] (8'') {\scriptsize $1 \le i \le n-2$};
    \node[main,minimum size=1 cm] (9) [below right= of 8] {$f=f_0$};
    
    \draw[violet,->] (1) -- (4);
    \draw[violet,->] (1) -- (8);
    \draw[violet,<->] (2) -- (4);
    \draw[violet,->] (2) -- (8);
    \draw[violet,<->] (3) -- (4);
    \draw[violet,->] (5) -- (2);
    \draw[violet,->] (6) -- (3);
    \draw[violet,->] (3) ..controls(0,-5).. (8);
    \draw[violet,->] (7) ..controls(6,1).. (1);
    \draw[violet,<->] (7) ..controls(3.5,-3).. (8);
    \draw[violet,->] (9) ..controls +(2,-8) and +(-8,-8).. (5);
    \end{tikzpicture}}
    \caption{The graph $G(\mathcal{A}_r)$. The nodes represent possible identifications, and the edges represent following identifications, after foldings of type $1$. Since the identifications of vertices that follow from the identification $a_i=a_{i+n-1}$ do not depend on $i$ for all $i\in\{1,\dots,n-2\}$, we do not draw all of them}
    \label{f_maximal_dependencies_graph}
\end{figure}
This shows that condition $(1)$ holds.

Finally, for condition $(2)$, we need to show that $H$ is of infinite index, and it does not fix any point in the open unit interval. To show the first claim, note that since $[F_n,F_n]$ is a simple subgroup of infinite index of $F_n$ (\cite{Higman1974,B1}), it follows that the subgroups of finite index of $F_n$ contain the commutator $[F_n,F_n]$. So, if $H$ was a subgroup of finite index of $F_n$, then by Corollary \ref{c_gen_closure_commutator}, $\mathcal{L}(H)=\mathcal{A}_r$ would have exactly $n-1$ inner vertices. Since $\mathcal{A}_r$ has $2n$ inner vertices, $H$ is not a subgroup of finite index of $F_n$.

For the latter claim, note that for any inner $n$-ary word $u$, there exists some inner $n$-ary word $v$ such that $u^+=v^+$ and the intervals $[u]$ and $[v]$ are disjoint. To see this, note that for every inner vertex $a$ of $\mathcal{A}_r$, there exists a path ending in $a$ in $\mathcal{A}_r$ whose label starts with the string $00$, and there exists another path ending in $a$ in $\mathcal{A}_r$ whose label starts with $n-1$. So, for every point $x\in (0,1)$, we can find some inner $n$-ary word $u$ such that $x\in [u]$, and an element $h\in H$ such that $h$ maps $[u]$ to a disjoint interval, and in particular $h$ does not fix $x$.

\end{proof}

%% file: Openproblems.tex
\section{Further research and subsequent developments}
\label{section:open_problems}

We conclude with several questions suggested by the generation criterion proved in this paper, and with some comments on subsequent developments.

Theorem \ref{gen_t_main_2} gives a sufficient criterion for a finite set $X\subseteq F_n$ to generate $F_n$. The criterion consists of the closure condition, the abelianization condition, the tuple condition, and the semi-core condition. The first open problem is whether the semi-core condition is superfluous.

\begin{Problem}\label{p_semicore_automatic}
Let $X\subseteq F_n$ be finite, and let $H=\langle X\rangle$. Suppose that conditions $(1)$, $(2)$ and $(3)$ of Theorem \ref{gen_t_main_2} hold. Does it follow that $H=F_n$?
\end{Problem}

A positive answer would give a  solution to the generation problem in $F_n$. Indeed, in that case, conditions $(1)$, $(2)$ and $(3)$ would be necessary and sufficient for $H$ to be equal to $F_n$, and they are effectively checkable for finitely generated subgroups using the core, the abelianization, and the tuple algorithm.

For Thompson's group $F=F_2$, the answer to Problem \ref{p_semicore_automatic} is positive \cite{G1}. This yields an effective necessary and sufficient criterion for generation in $F$. In subsequent work on maximal subgroups of $F$ \cite{G2}, it was proved that, for subgroups of $F=F_2$, the conditions $(1)$ and $(2)$ already imply  condition $(3)$. Thus, in $F$, there is an even shorter criterion: the closure and the image in the abelianization determine whether a subgroup is the whole group.

It is natural to ask whether the same short criterion holds for the groups $F_n$.

\begin{Problem}\label{p_short_generation_criterion}
Let $H\leq F_n$. Suppose that
\[
    [F_n,F_n]\subseteq \Cl(H)
\]
and
\[
    H[F_n,F_n]=F_n.
\]
Does it follow that $H=F_n$?
\end{Problem}

A positive answer to Problem \ref{p_short_generation_criterion} would remove the tuple condition from the final statement of the generation criterion and would give a short generation test for $F_n$, analogous to the criterion known for $F$.

Problem \ref{p_short_generation_criterion} is also closely related to maximal subgroups of infinite index. In fact, a positive answer is equivalent to the assertion that every maximal subgroup of infinite index of $F_n$ is closed.

Indeed, assume first that Problem \ref{p_short_generation_criterion} has a positive answer, and let $M$ be a maximal subgroup of infinite index in $F_n$. 
In particular, $M$ is not contained in any proper finite index subgroup of $F_n$. Hence, $M[F_n,F_n]=F_n$.
If $M$ is not closed, then by maximality $\Cl(M)=F_n$. Hence $M$ satisfies the two hypotheses of Problem \ref{p_short_generation_criterion}, and so $M=F_n$, a contradiction. Therefore $M$ is closed.

Conversely, suppose that every maximal subgroup of infinite index in $F_n$ is closed. Let $H< F_n$ be a strict subgroup such that
\[
    [F_n,F_n]\subseteq \Cl(H)
    \quad\text{and}\quad
    H[F_n,F_n]=F_n.
\]
Then $\Cl(H)=F_n$. Moreover, since $H[F_n,F_n]=F_n$, $H$ is not contained in any proper finite-index subgroup of $F_n$. Since $F_n$ is finitely generated, Zorn's lemma implies that $H$ is contained in a maximal subgroup $M$ of $F_n$. This maximal subgroup cannot have finite index, and hence is closed by assumption. But then
\[
    F_n=\Cl(H)\leq \Cl(M)=M,
\]
a contradiction. Thus no such proper subgroup $H$ exists, and Problem \ref{p_short_generation_criterion} has a positive answer.

Therefore the short generation criterion in Problem \ref{p_short_generation_criterion} holds if and only if every maximal subgroup of infinite index of $F_n$ is closed. This equivalence is one of the reasons the problem is important: it connects the generation problem directly to the structure of maximal subgroups.

The question whether every maximal subgroup of $F_n$ of infinite index is closed is relevant to a broader classification problem for maximal subgroups of Higman--Thompson groups. After the completion of the present work, the first author formulated the following problem.

\begin{Problem}\label{p_minimal_maximal_F}
Is every maximal subgroup of infinite index of Thompson's group $F$ which acts minimally on the open interval $(0,1)$ isomorphic to a Higman--Thompson group $F_m$, for some $m\geq 2$?
\end{Problem}

The point-stabilizer examples are excluded by the minimality hypothesis. Problem \ref{p_minimal_maximal_F} is consistent with all examples of maximal subgroups of $F$ which appear in the literature. More generally, the same pattern occurs in the known examples of maximal subgroups of infinite index in Higman--Thompson groups which do not arise as point stabilizers. The maximal subgroups constructed in the present paper are isomorphic to $F_{2n-1}$. In addition, after the completion of the present work, the first author proved that the Jones subgroup of $F_3$ used by Aiello and Nagnibeda is isomorphic to $F_4$ (the result will appear in a future paper).

If every maximal subgroup of infinite index of $F_n$ is closed, then the study of closed subgroups  becomes even more central to the classification of maximal subgroups. This motivates the following broader problems.

\begin{Problem}\label{p_closed_maximal_subgroups}
Describe the  maximal subgroups of infinite index of $F_n$. In particular, determine under which natural hypotheses such a subgroup must be isomorphic to a Higman--Thompson group $F_m$. Is this true, for example, for all (closed) maximal subgroups of infinite index which act minimally on $(0,1)$?
\end{Problem}

Several results obtained by the first author after the completion of this paper support this direction. The first author has proved that a wide class of closed subgroups of $F$, and more generally of $F_n$, are isomorphic to Higman--Thompson groups. These results apply, in particular, to the diagram groups associated with the  classes $\mathcal C_n$ of indecomposable geometrically fast groups generated by $n$ one-bump functions. For the terminology of geometrically fast sets, dynamical diagrams, and the classes $\mathcal C_n$, see \cite{BBKMZ,BelkStott}. Belk and Stott proved that fast one-bump groups are isomorphic to diagram groups determined by their dynamical diagrams, and used this description to prove that pseudo-$F_4$ is isomorphic to $F_4$ \cite{BelkStott}. Although the original fast one-bump groups are not, in general, closed subgroups of $F$, the corresponding diagram groups can be realized as closed subgroups of $F$ using methods from \cite{GSclosed}. The first author proved that, for every $n\geq 2$, every group in the class $\mathcal C_n$ is isomorphic to $F_n$. The cases $n=2,3$ were previously known, and the case $n=4$ is the theorem of Belk and Stott. Thus, for the classes $\mathcal C_n$, this answers the strong form of a question of Brin and Zaremsky from \cite{BurilloBuxNucinkis2018}. The first author has also constructed, for every pair $m,n$ with $n\geq m\geq 2$, a maximal subgroup of $F_m$ isomorphic to $F_n$. The proofs of these results will appear in subsequent papers.

These later results make essential use of the generation criterion and core methods developed in the present paper, together with additional arguments.

%% file: Bibliography.tex
\bigskip

\noindent\textsc{Gili Golan}\\
Department of Mathematics,\\
Ben-Gurion University of the Negev,\\
\texttt{golangi@bgu.ac.il}

\medskip

\noindent\textsc{Eytan Sapir}\\
Department of Mathematics,\\
Ben-Gurion University of the Negev,\\
\texttt{esapir@post.bgu.ac.il}

%% file: main.bbl
\begin{thebibliography}{99}

\bibitem{AN} 
V.~Aiello and T.~Nagnibeda, 
\newblock \emph{On the oriented Thompson subgroup $\vec{F}_3$ and its relatives in higher Brown-Thompson groups}, 
\newblock Journal of Algebra and Its Applications, \textbf{21}(07), 2250139 (2022).

\bibitem{AN2}
V.~Aiello and T.~Nagnibeda,
\newblock On the 3-colorable subgroup $\mathcal{F}$ and maximal subgroups of Thompson's group $F$,
\newblock \emph{Annales de l'Institut Fourier} \textbf{73} (2023), no.~2, 783--828.
\newblock \doi{10.5802/aif.3555}

\bibitem{BBKMZ}
C.~Bleak, M.~G.~Brin, M.~Kassabov, J.~Tatch Moore, and M.~C.~B.~Zaremsky,
\newblock \emph{Groups of fast homeomorphisms of the interval and the ping-pong argument},
\newblock J. Comb. Algebra \textbf{3} (2019), no.~1, 1--40.
\newblock \doi{10.4171/JCA/25}


\bibitem{BelkForest2023}
J.~Belk and B.~Forrest,
\newblock \emph{Divergence in the Brown--Thompson groups},
\newblock Transform. Groups (2023), published online.
\newblock \doi{10.1007/s00031-023-09839-8}

\bibitem{BelkStott}
J. Belk and L. Stott,
\emph{Pseudo-$F_4$ is isomorphic to $F_4$},
arXiv:2303.16868.


\bibitem{Bleak2008}
Collin Bleak,
\newblock A geometric classification of some solvable groups of homeomorphisms,
\newblock {\em Journal of the London Mathematical Society}, Series 2, vol. 78 (2008), no. 2, pp. 352--372.

\bibitem{BrownGeoghegan1984}
K.~S.~Brown and R.~Geoghegan,
\newblock \emph{An infinite-dimensional torsion-free ${\mathrm{FP}}_\infty$ group},
\newblock Invent. Math., \textbf{77} (1984), 367--381.

\bibitem{B1} 
K.~S.~Brown, 
\newblock \emph{Finiteness properties of groups}, 
\newblock Journal of Pure and Applied Algebra, \textbf{44} (1987).

\bibitem{BurilloBuxNucinkis2018}
J. Burillo, K.-U. Bux and B. Nucinkis,
\emph{Cohomological and metric properties of groups of homeomorphisms of $\mathbb R$},
Oberwolfach Reports \textbf{15} (2018), 1579--1633.


\bibitem{CFP1996}
J.~W.~Cannon, W.~J.~Floyd, and W.~R.~Parry.
\newblock Introductory notes on Richard Thompson's groups.
\newblock \emph{L'Enseignement Math\'{e}matique}, 42(3):215--256, 1996.

\bibitem{G1}
G.~Golan-Polak,
\newblock \emph{The generation problem in Thompson group $F$},
\newblock Memoirs of the AMS, \textbf{292}, no.~1451 (2023).

\bibitem{G2}
G.~Golan-Polak,
\newblock \emph{On maximal subgroups of Thompson's group $F$},
\newblock Groups, Geometry, and Dynamics  19 (3), 797-860.
\newblock \doi{10.4171/GGD/795}

\bibitem{GS1}
G.~Golan and M.~Sapir,
\newblock \emph{On subgroups of R. Thompson group $F$},
\newblock Trans. Amer. Math. Soc. \textbf{369} (2017), 8857--8878.

\bibitem{GSclosed}
G.~Golan-Polak and M.~Sapir,
\newblock \emph{On closed subgroups of the R. Thompson group $F$},
\newblock Israel J. Math. \textbf{267} (2025), 35--84.
\newblock \doi{10.1007/s11856-024-2692-z}

\bibitem{GubaSapir1997}
V.~S.~Guba and M.~V.~Sapir,
\newblock \emph{Diagram groups},
\newblock Memoirs Amer. Math. Soc., \textbf{130}(620), 1997.

\bibitem{GubaSapir1999subgroups}
V.~S.~Guba and M.~V.~Sapir,
\newblock \emph{On subgroups of R. Thompson's group $F$ and other diagram groups},
\newblock Sb. Math. (English), \textbf{190}(8) (1999), 1077--1130.

\bibitem{Higman1974}
Graham Higman,
\newblock {\em Finitely Presented Infinite Simple Groups},
\newblock Notes on Pure Mathematics, no.~8, Australian National University, Canberra, 1974.

\bibitem{LyndonSchupp1977}
R.~C. Lyndon and P.~E. Schupp,
\newblock {\em Combinatorial Group Theory},
\newblock Ergebnisse der Mathematik und ihrer Grenzgebiete, Vol.~89,
\newblock Springer--Verlag, Berlin, 1977.

\bibitem{Sav1}
D.~Savchuk,
\newblock \emph{Some graphs related to Thompson's group $F$},
\newblock In: Combinatorial and Geometric Group Theory, 279--296, Trends Math., Birkhäuser/Springer, Basel, 2010.

\bibitem{Sav2}
D.~Savchuk,
\newblock \emph{Schreier graphs of actions of Thompson's group $F$ on the unit interval and on the Cantor set},
\newblock Geom. Dedicata, \textbf{175} (2015), 355--372.

\bibitem{Stein}
M. Stein, \emph{Groups of piecewise linear homeomorphisms}, Trans. Amer. Math. Soc. \textbf{332} (1992), no.~2, 477--514.

\bibitem{Wladis2007}
C.~Wladis,
\newblock \emph{Thompson's group $F(n)$ is not minimally almost convex},
\newblock New York J. Math., \textbf{13} (2007), 319--326.
\newblock \url{https://nyjm.albany.edu/j/2007/13-19.pdf}

\end{thebibliography}
